\documentclass[11pt,fullpage,doublespace]{article}
\usepackage{graphicx} 
\usepackage{setspace} 
\usepackage{latexsym} 
\usepackage{amsfonts} 
\usepackage{amsmath} 
\usepackage{amssymb}
\usepackage{accents}
\usepackage{textcomp}
\usepackage{enumerate}
\usepackage{mathrsfs}
\usepackage{bm}
\usepackage{stmaryrd}
\usepackage[margin=1in]{geometry}
\usepackage{amsthm}
\usepackage{ifsym}
\usepackage{amssymb,latexsym,amsmath}
\usepackage{graphics}
\usepackage{titlesec}
\usepackage{changepage}
\usepackage{lipsum}
\usepackage[retainorgcmds]{IEEEtrantools}
\usepackage{thmtools}
\usepackage{comment}
\usepackage[colorlinks=true,linkcolor=blue]{hyperref}
\usepackage{centernot}
\usepackage{mathtools}
\usepackage{tikz}
\usepackage{hyperref}

\usetikzlibrary{arrows,calc,through,backgrounds,matrix,decorations.pathmorphing}

\tikzset{
  font=\small}
\declaretheorem[style=definition,qed=$\dashv$,numberwithin=section]
{definition}

\declaretheorem[style=plain,sibling=definition]{theorem}
\declaretheorem[style=plain,sibling=definition]{lemma}

\declaretheorem[style=plain,sibling=definition]{proposition}
\declaretheorem[style=plain,sibling=definition]{corollary}
\declaretheorem[style=definition,qed=$\dashv$, sibling=definition]{remark}
\declaretheorem[style=plain,sibling=definition]{claim}
\declaretheorem[style=plain,sibling=definition]{subclaim}
\declaretheorem[style=plain,sibling=definition]{claim*}

\declaretheorem[style=definition,qed=$\dashv$, sibling=definition]{remarks}

\titleformat{\section}{\normalsize\centering}{\thesection.}{1em}{}
\titleformat{\subsection}{\normalsize\centering}{\thesubsection.}{1em}{}
\titleformat{\subsubsection}{\normalsize}{\thesubsubsection.}{1em}{}
\numberwithin{equation}{section}

\newcommand{\df}{\dot{F}}
\newcommand{\vep}{\varepsilon}

\newcommand{\la}{\langle}
\newcommand{\ra}{\rangle}

\newcommand{\dom}{{\rm dom}}

\newcommand{\cf}{{\rm cf}}
\newcommand{\lh}{{\rm lh}}

\newcommand{\crt}{{\rm crt}}
\newcommand{\hc}{{\mathrm HC}}
\newcommand{\Ult}{{\rm Ult}}
\newcommand{\ran}{{\rm ran}}
\newcommand{\rng}{{\rm ran}}
\newcommand{\within}{{\restriction}}
\newcommand{\ohat}{\hat{o}}

\def\k{\kappa}

\renewcommand{\models}{\vDash}
\newcommand{\powerset}{{\wp }}

\def\P{{\mathcal{P} }}

\def\Q{{\mathcal{ Q}}}

\def\R{{\mathcal R}}

\def\M{{\mathcal{M}}}
\def\N{{\mathcal{N}}}

\def\T {{\mathcal{T}}}
\def\U{{\mathcal{U}}}

\def\itW{{\mathcal{W} }}

\def\itM{{\mathcal{M}}}

\def\itF{{\mathcal{F}}}
\def\itT {{\mathcal{T}}}
\def\itU{{\mathcal{U}}}

\def\itS{{\mathcal{S}}}

\newcommand{\Hull}{\mathrm{Hull}}
\newcommand{\crit}{\mathrm{crt}}

\newcommand{\OR}{\text{OR}}
\newcommand{\J}{\mathcal J}

\newcommand{\lex}{\mathrm{lex}}

\renewcommand{\OR}{\textit{o}}

\newcommand{\cof}{\mathrm{cof}}

\newcommand{\adp}{\mbox{{\sf AD}}^+}
\newcommand{\db}{\dot{B}}
\newcommand{\degr}{\textrm{deg}}
\newcommand{\hM}{\hat{M}}
\newcommand{\bfc}{\bar{\mathfrak{C}}}
\newcommand{\mfc}{\mathfrak{C}}

\newcommand{\rh}{\hat{\rho}}
\newcommand{\zfc}{\textsf{ZFC}}
\newcommand{\Res}{\textrm{Res}}
\newcommand{\lift}{\textrm{lift}}

\newcommand{ \tpred}{\textrm{-pred}}
\newcommand{ \cHull}{\textrm{cHull}}

\newcommand*{\TitleFont}{%
      \usefont{\encodingdefault}{\rmdefault}{b}{n}%
      \fontsize{12}{16}%
      \selectfont}

 \input xy
 \xyoption{all}
\begin{document}
\title{\TitleFont CONDENSATION FOR MOUSE PAIRS}
\renewcommand{\thefootnote}{\fnsymbol{footnote}} 
\renewcommand{\thefootnote}{\arabic{footnote}}
\author{ \selectfont 
JOHN STEEL\footnote
{Department of Mathematics, University of California, Berkeley, CA, USA. Email: coremodel@berkeley.edu}
\\
NAM TRANG\footnote{Department of Mathematics, University North Texas, Denton, TX USA. Email: Nam.Trang@math.unt.edu}}

\date{}
\maketitle
\begin{abstract}
  This is the first of two papers on the fine structure of HOD in models of the Axiom of Determinacy ($\sf{AD}$). 
 Let $M\vDash\sf{AD}^+ + V=L(\powerset(\mathbb{R}))$. \cite{normalization_comparison} shows that under
 a natural hypothesis on the existence of iteration strategies, the basic fine structure
 theory for pure extender models goes over to $\mbox{HOD}^M$. In this paper, we prove
 a fine condensation theorem, quite similar to Theorem 9.3.2 of Zeman's book
 \cite{Zeman}, except that condensation for iteration strategies has been added to the mix.
 In the second paper, we shall use this theorem to show that in $\mbox{HOD}^M$,
 $\square_\kappa$ holds iff $\kappa$ is not subcompact. 

  \end{abstract}

\section{INTRODUCTION}\label{sec:intro}

 One goal of descriptive inner model theory is to elucidate the structure of HOD
 (the universe of hereditarily ordinal definable sets) in models $M$ of the Axiom of
 Determinacy. $\mbox{HOD}^M$ is close to $M$ in various ways; for example, if
 $M\vDash \sf{AD}^+ + V=L(\powerset(\mathbb{R}))$\footnote{$\sf{AD}^+$ is a technical strengthening of $\sf{AD}$.
  It      
is not known whether $\sf{AD} \Rightarrow \sf{AD}^+$, though in every model of $\sf{AD}$ constructed so far,
   $\sf{AD}^+$ also holds. The models of $\sf{AD}$ that we deal with in this paper satisfy
    $\sf{AD}^+$.}, then 
$M$ can be realized as a symmetric forcing extension of $\mbox{HOD}^M$, so that the first order
theory of $M$ is part of the first order theory of its HOD.
\footnote{This is a theorem of Woodin
from the early 1980s. Cf. \cite{trang2014hod}.} For this and many other reasons,
 the study of HOD in models of {\sf AD} has a long history. We refer the reader to
 \cite{cabalsurvey} for a survey of this history.  

The study of HOD involves ideas from descriptive set theory (for example, games and definable scales)
and ideas from inner model theory (mice, comparison, fine structure). One early result
showing that inner model theory is relevant is due to the first author, who showed in 1994 (\cite{bulletinpaper})
that if there are $\omega$ Woodin cardinals with a measurable above them all,
then in $L(\mathbb{R})$, HOD up to $\theta$ is a pure extender mouse. Shortly afterward, this result
was improved by W. Hugh Woodin, who reduced its hypothesis to ${\sf{AD}}^{L(\mathbb{R})}$, and identified
the full $\mbox{HOD}^{L(\mathbb{R})}$ as a model of the form $L[M,\Sigma]$, where
$M$ is a pure extender premouse, and $\Sigma$ is a partial iteration strategy for $M$.
$\mbox{HOD}^{L(\mathbb{R})}$ is thus a new type of mouse, sometimes called a
{\em strategy mouse}, sometimes called a {\em hod mouse}.
See \cite{steelwoodincabal} for an account of this work.

Since the mid-1990s, there has been a great deal of work devoted to extending
these results to models of determinacy beyond $L(\mathbb{R})$. Woodin analyzed
HOD in models of $\sf{AD}^+$ below the minimal model of $\sf{AD}_{\mathbb{R}}$ 
fine structurally, and Sargsyan pushed the analysis further, first to determinacy
models below $\sf{AD}_{\mathbb{R}} + ``\theta$ is regular" (see \cite{hod_mice}),
and more recently, to determinacy models below the
 minimal model of the theory 
 ``$\sf{AD}^+ + \Theta = \theta_{\alpha+1} + \theta_\alpha$ is the largest Suslin cardinal" (commonly known as $\sf{LSA}$).
 (See \cite{hod_mice_LSA}.) 
 The hod mice used in this work have the
 form $M=L[\vec{E},\Sigma]$, where $\vec{E}$ is a coherent sequence of extenders, and
 $\Sigma$ is an iteration strategy for $M$. The strategy information is fed into
 the model $M$ slowly, in a way that is dictated in part by the determinacy model
 whose HOD is being analyzed. One says that the hierarchy of $M$ is
 {\em rigidly layered}, or {\em extender biased}. The object $(M,\Sigma)$ is
 called a rigidly layered (extender biased) hod pair.
 
  Putting the strategy information in this way
 makes comparison easier, but it has serious costs. The definition of
 ``premouse" becomes very complicated, and indeed it is not clear how to
 extend the definition of rigidly layered hod pairs much past that given in
 \cite{hod_mice_LSA}. The definition of ``extender biased hod premouse" is not uniform,
 in that the extent of extender bias depends on the determinacy model whose
 HOD is being analyzed. Fine structure, and in particular condensation, become 
 more awkward. For example, it is not true in general that the pointwise definable 
 hull of a level of $M$ is a level of $M$. (The problem is that
 the hull will not generally be sufficiently extender biased.) Because of this, it is open whether
 the hod mice of \cite{hod_mice_LSA} satisfy $\forall \kappa \square_\kappa$.
 (The second author did show that $\forall \kappa \square_{\kappa,2}$ holds in
 these hod mice; cf. \cite{hod_mice_LSA}.)
 
 The more naive notion of hod premouse would abandon extender bias, and simply
 add the least missing piece of strategy information at essentially every stage.
 This was originally suggested by Woodin. The first author has recently proved
 a general comparison theorem that makes it possible to use this approach, at least
 in the realm of short extenders. The resulting premice are called
 {\em least branch premice} (lpm's), and the pairs $(M,\Sigma)$ are called
 {\em least branch hod pairs} (lbr hod pairs).\footnote{The (pure extender or least branch hod) premice in the paper are called pfs (projectum-free space) premice in \cite{normalization_comparison}. We will occasionally omit the ``pfs" for brevity. All premice used in this paper are pfs premice (and their strong cores), see Section \ref{sec:lbh} for more discussion.} Combining results of 
 \cite{normalization_comparison} and \cite{localhodcomputation}, one has
 
 \begin{theorem}[\cite{normalization_comparison},\cite{localhodcomputation}] \label{hodingeneral} Assume
 $\sf{AD}^+ + $ ``there is an $(\omega_1,\omega_1)$ iteration strategy for a pure extender premouse
 with a long extender on its sequence". Let $\Gamma \subseteq P(\mathbb{R})$ be such that
 $L(\Gamma,\mathbb{R}) \models \sf{AD}_{\mathbb{R}} +$  ``there is no $(\omega_1,\omega_1)$ iteration strategy 
 for a pure extender premouse
 with a long extender on its sequence"; then $\mbox{HOD}^{L(\Gamma,\mathbb{R})}$ is
 a least branch premouse.
 \end{theorem}
 
 Of course, one would like to remove the iterability hypothesis of \ref{hodingeneral}, and prove its
 conclusion under $\sf{AD}^+$ alone. Finding a way to do this is one
 manifestation of the long standing
 iterability problem of inner model theory. Although we do not yet
 know how to do this, the theorem does make it highly likely that
 in models of $\sf{AD}_{\mathbb{R}}$ that have not reached an iteration strategy
 for a pure extender premouse with a long extender, HOD is an lpm.
 
 Least branch premice have a fine structure much closer to that of pure
 extender models than that of rigidly layered hod premice. The book
 \cite{normalization_comparison} develops the basics, the solidity and
 universality of standard parameters, and a coarse form of condensation.
 The main theorem of this paper, Theorem \ref{thm:cond_lem}, is
 a stronger condensation theorem. The statement of
 \ref{thm:cond_lem} is parallel to that of Theorem 9.3.2 of
 \cite{Zeman}, but it has a strategy-condensation feature
 that is new even in the pure extender model context.
 The proof of \ref{thm:cond_lem} follows the same outline as
 the proofs of solidity, universality, and condensation given in
 \cite{normalization_comparison}, but there are a number of additional
 difficulties to be overcome. These stem from the restricted
 elementarity we have for the ultrapowers of phalanxes that are taken
 in the course of the proof.

Theorem \ref{thm:cond_lem} is one of the main ingredients  in the proof of 
the main theorem of our third paper. We say that $(M,\Sigma)$ is a {\em mouse pair}
iff $M$ is either a pure extender pfs premouse or a least branch pfs premouse, and
$\Sigma$ is an iteration strategy for $M$ that has strong hull condensation, normalizes well,
is internally lift consistent, and in the least branch case, is pushforward consistent.
See \cite[Chapter 9]{normalization_comparison} and Section \ref{sec:lbh} below for a
full definition.\footnote{Theorem \ref{thm:cond_lem} is also used heavily in the proof in
\cite{siskindsteel} that the iteration strategy component of a mouse pair 
fully normalizes well, and is therefore positional.}

\begin{theorem}[$\sf{AD}^+$]\label{thm:main_theorem}
Let $(M,\Sigma)$ be a mouse pair. Let $\kappa$ be a cardinal of $M$ such that $M \models
``\kappa^+$ exists"; then in $M$, the following are equivalent.
\begin{enumerate}
\item $\square_\kappa$.
\item $\square_{\kappa,<\kappa}$.
\item $\kappa$ is not subcompact.
\item The set of $\nu<\kappa^+$ such that $M|\nu$ is extender-active is non-stationary in $\kappa^+$.
\end{enumerate}
\end{theorem}

The special case of this theorem in which $M$ is
a pure extender model is a landmark result of
Schimmerling and Zeman. (See
\cite{schimmerling2004characterization}.) Our proof follows
the Schimmerling-Zeman proof quite closely.

 Theorem \ref{thm:main_theorem} has applications to consistency strength lower bound questions
 that we discuss in the second paper.
  But our work was also motivated
 by the desire to put the fine structure theory of \cite{normalization_comparison}
 to the test, so to speak. Determining the pattern of $\square$ is a good way to
 go one level deeper into the world of projecta, standard parameters, 
 restricted elementarity, and condensation theorems. We found when we did so
 that the definition of {\em hod premouse} given in the first draft of \cite{normalization_comparison}
 had problems, in that strategy information was being added in a way that might not in
 general be preserved by $\Sigma_1$ hulls.\footnote{Remark {\bf 2.47} of \cite{scales_hybrid_mice}
shows that in fact it is preserved by $\Sigma_1$ hulls, but the proof involves
a phalanx comparison, and so a lot of theory just to prove a property of mice
one would like to have available at the beginning.} The better method for
 strategy insertion comes from \cite{scales_hybrid_mice}, and we describe it further
 below. \cite{normalization_comparison} has been revised so that it now uses this method.

\textbf{Acknowledgements.}  The work reported here began when
 the second author visited the first author in March and June of 2016 at UC Berkeley. 
 The second author thanks the NSF for its generous support through grants No DMS-1565808 and DMS-1945592.   

\section{LEAST-BRANCH HOD PREMICE}\label{sec:lbh}

All premice used in this paper are in the pfs hierarchy, defined in \cite{normalization_comparison}.
 We adopt for the most part the fine structure and notation from \cite[Chapter 9]{normalization_comparison}
 concerning least-branch hod premice (lpm's) and lbr hod pairs. A similar, albeit simpler, fine structure for pure extender premice is discussed in  \cite[Chapter 4]{normalization_comparison}. Fine structural condensation is the point of this
paper, so we cannot avoid going into the details covered in these chapters. (Our main result is a fine structural
refinement of \cite[Theorem 4.10.10]{normalization_comparison}.)  We summarize some main points below. 
 The reader can see \cite[Chapters 4,9]{normalization_comparison} for more details.

\subsection{Potential least branch premice} The language for lpm's is $\mathcal{L}_1$ with symbols
 $\in, \dot{E},\dot{F},\dot{\Sigma},\dot{B}, \dot{\gamma}$. $\mathcal{L}_0 = \mathcal{L}_1 - \{\dot{B},\dot{\Sigma}\}$ is the language of pure extender mice. An lpm $M$ is of the
  form $(N,k)$ where $N$ is an $\mathcal{L}_1$ amenable structure that is $k$-sound. We write
  $k = \textrm{deg}(M)$.  We often identify $M$ with $N$ and suppress $k$. $o(M)$ denotes the 
   ordinal height of $M$, and $\hat{o}(M)$ denotes the $\alpha$ such that $o(M)=\omega\alpha$.
    $l(M)=(\hat{o}(M),\textrm{deg}(M))$ is the index of $M$. For $(\nu,l)\leq_{\rm{lex}}$$l(M)$, $M|(\nu,l)$
     is the initial segment of $M$ with index $(\nu,l)$. We write $N \unlhd M$ iff
     $N = M|(\nu,l)$ for some  $(\nu,l)\leq_{\rm{lex}}$$l(M)$. If $\nu\leq \hat{o}(M)$, write $M|\nu$ for $M|(\nu,0)$.\footnote{
$M|\nu$ can be active. We write $M||\nu$, or $M|\la \nu,-1\ra$, for $M|\nu$ cut a second time by removing
its last extender if it has one.}

$\dot{E}^M$ codes the sequence of extenders that go into constructing $M$.
 $\dot{F}^M$ if non-empty is the amenable code for a new extender being added; 
 in this case, we say that $M$ is \textit{extender-active} (or just $E$-active).
  If $\dot{F}^M =F$ is nonempty, then $M\vDash \rm{crt}$$(F)^+$ exists and $o(M)=i^M_F(\mu)$,
 where $\mu=\rm{crt}$$(F)^+$. Also $F$ must satisfy the Jensen initial segment condition (ISC),
  that is, whole initial segments of $F$ must be in $\dot{E}^M$ (see \cite{Zeman}
   for a detailed discussion of ISC). $\dot{\gamma}$ is the index of
    the largest whole initial segment of $F$ if exists; otherwise, $\dot{\gamma}=0$.
 We also demand $M$ is \textit{coherent}, that is $i^M_F(\dot{E}^M)\restriction o(M)+1 =
  (\dot{E}^{M})^\smallfrown \langle \emptyset \rangle$. 

$\dot{\Sigma}^M$ and $\dot{B}^M$ are used to record information about an iteration strategy
 $\Omega$ of $M$. $\dot{\Sigma}^M$ codes the strategy information added at
  earlier stages; $\dot{\Sigma}^M$ acts on $\lambda$-separated trees.\footnote{See \cite[Chapter 4]{normalization_comparison} for detailed discussions on $\lambda$-separated trees.} $\dot{\Sigma}^M(s,b)$ implies 
  that $s=\langle \nu,k,\mathcal{T}  \rangle$,
   where $(\nu,k)\leq l(M)$ and $\mathcal{T}$ is a $\lambda$-separated tree on
    $M|(\nu,k)$ in $M$ of limit length and $\mathcal{T}^\smallfrown b$ is according to the strategy.
 We say that $s$ is an $M$-tree, and write $s = \langle \nu(s),k(s),\mathcal{T}(s)\rangle$. We 
  write $\dot{\Sigma}^M_{\nu,k}$ for the partial iteration strategy for $M|(\nu,k)$
   determined by $\dot{\Sigma}$. We write $\Sigma^M(s)=b$ when $\dot{\Sigma}^M(s,b)$, and
    we say that $s$ is according to  $\Sigma^M$ if $\T(s)$ is according to $\dot{\Sigma}^M_{\nu(s),k(s)}$.

Now we discuss how to code branch information for a tree $\T(s)$ such that $\Sigma^M(s)$
has not yet been
 defined into the $\dot{B}^M$ predicate. Here we use
the $\mathfrak{B}$-operator in \cite{scales_hybrid_mice}. 
We are correcting some errors in the original version of
\cite{normalization_comparison}. These corrections have been incorporated
in its latest version.

$M$ is \textit{branch-active} (or just $B$-active) iff 
\begin{itemize}
\item[(a)] there is a largest $\eta < o(M)  $ such that $M|\eta \models {\sf KP}$,
and letting $N=M|\eta$,
\item[(b)] there is a
 $<_N$-least $N$-tree $s$ such that $s$ is by $\Sigma^N$,
		$\T(s)$ has limit length, and $\Sigma^N(s)$ is undefined.
		\item[(c)] for $N$ and $s$ as above, $o(M) \le
			o(N) + \lh(\T(s))$.
\end{itemize}

Note that being branch-active can be expressed by a $\Sigma_2$ sentence in
$\mathcal{L}_1 -\lbrace \dot{B} \rbrace$. This contrasts with being extender-active,
which is not a property of the premouse with its top extender removed. In contrast
with extenders, we know when branches must be added before we do so.

\begin{definition}\label{branchactivelpm}
Suppose that $M$ is branch-active. We set
\begin{align*}
	\eta^M &= \text{the largest $\eta$ such that $M|\eta \models {\sf KP}$,}\\
	\nu^M &= \text{unique $\nu$ such that $\eta^M + \nu = o(M)$},\\
	s^M &= \text{least $M|\eta^M$-tree such that $\dot{\Sigma}^{M|\eta^M}$ is undefined, and}\\ 
	b^M &= \text{$\lbrace \alpha \mid \eta + \alpha \in \dot{B}^M \rbrace.$}	
\end{align*}
	Moreover,
	\begin{itemize}
	\item[(1)]
$M$ is a {\em potential lpm} iff $b^M$ is a cofinal
	branch of $\T(s) \within \nu^M$.
	\item[(2)] $M$ is {\em honest}
	iff $\nu^M = \lh(\T(s))$, or $\nu^M < \lh(\T(s))$ and
	$b^M = [0,\nu^M)_{T(s)}$.
	\item[(3)] $M$ is an lpm iff $M$ is an honest potential lpm.
	\item[(4)] $M$ is {\em strategy active} iff $\nu^M = \textrm{lh}(\T(s))$.
	\end{itemize}
	\end{definition}

Note that $\eta^M$ is a
$\Sigma_0^M$ singleton, because it is the least ordinal in $\dot{B}^M$ (because
0 is in every branch of every iteration tree), and thus $s^M$ is also a
$\Sigma_0^M$ singleton. We have separated honesty from the other conditions
because it is not expressible by a $Q$-sentence, whereas the rest is.
Honesty is expressible by a Boolean combination of $\Sigma_2$ sentences.
See \ref{lem:B_active_case} below.

The original version of \cite{normalization_comparison}
required that when $o(M) < \eta^M + \lh(\T(s))$, $\dot{B}^M$ is empty,
 whereas here we require that it code 
$[0,o(M))_{T(s)}$, in the same way that $\dot{B}^M$ will have to code a new
branch when $o(M) = \eta^M + \lh(\T(s))$. Of course, $[0,\nu^M)_{T(s)} \in M$ 
when $o(M) < \eta^M + \lh(\T(s))$ and $M$ is honest, so the current
$\dot{B}^M$ seems equivalent to the original $\dot{B}^M = \emptyset$.
However, $\dot{B}^M=\emptyset$ leads to $\Sigma_1^M$ being too weak,
with the consequence that a $\Sigma_1$ hull of $M$ might collapse to
something that is not an lpm.\footnote{The hull could satisfy
$o(H) = \eta^H + \lh(\T(s^H))$, even though
$o(M) < \eta^M + \lh(\T(s^M))$. But then being an lpm
requires $\dot{B}^H \neq \emptyset$. See Remark 2.47 in \cite{scales_hybrid_mice} for a more detailed discussion. Basically, one can show that the $L^\Sigma[\vec{E}]$ constructions doesn't break down because the the models constructed are not $\Sigma$-premice; \cite[Remark 2.47]{scales_hybrid_mice} outlines an argument that cores of $\Sigma$-premice that are constructed in the $L^\Sigma[\vec{E}]$-constructions are $\Sigma$-premice.} Our current choice for
$\dot{B}^M$ solves that problem.

\begin{remark} Suppose $N$ is an lpm, and $N \models {\sf KP}$.
It is very easy to see that $\dot{\Sigma}^N$ is defined on all $N$-trees $s$
 that are by
$\dot{\Sigma}^N$ iff there are arbitrarily large
$\xi < o(N)$ such that $N|\xi \models {\sf KP}$.
Thus if $M$ is branch-active, then $\eta^M$ is a successor
admissible; moreover, we do add branch information, related to
exactly one tree, at each successor
admissible. Waiting until the next admissible
to add branch information is just a convenient way to make 
sure we are done coding in the branch information for a
given tree before we move on to the next one. One could go faster.
\end{remark}

We say that an lpm $M$ is (fully) passive if $\dot{F}^M=\emptyset$ and 
$\dot{B}^M=\emptyset$. It cannot be the case that $M$ is both $E$-active and $B$-active.
 In the case that $M$ is $E$-active, using the terminology of \cite{schimmerling2004characterization},
  the extender $\dot{F}^M$ can be of type $A$, $B$, or $C$.

\subsection{Solidity and soundness}
      
      We adopt the projectum-free space (pfs) fine structure in \cite[Chapter 4]{normalization_comparison}. We write $\rho_n(M)$ for the $n$-th projectum of $M$ and $p_n(M)$ for the $n$-th standard parameter of $M$. 
    We set $\rho(M) = \rho_{\textrm{deg}(M)+1}(M)$ and $p(M) = p_{\textrm{deg}(M)+1}(M)$, and call them {\em the} projectum
    and parameter of $M$. We say $M$ is {\em sound} iff it is $\textrm{deg}(M)+1$-sound. An lpm  $M$ must be
    $\textrm{deg}(M)$-sound, but it need not
    be $\textrm{deg}(M)+1$-sound. There are two types of premice, type 1 and type 2, with the distinction being
based on the soundness pattern of the premice. Type 1 are the most important.
All proper initial segments of an lpm must be sound type 1 lpms. Type 2 premice can be produced by taking
a $k$-ultrapower that is discontinuous at $\rho_k$.
    
    If $M$ is type 1 and $k$-sound for $k \geq 1$, then it is coded by its reduct $M^k$, where
\begin{align*}
 M^k &= (M||\rho_k(M), A^k_M), \\
\intertext{ and}
   A^k_M &= \{\la \varphi, b\ra \ | \  \varphi \textrm{ is } \Sigma_1 \wedge b \in M||\rho_k \wedge M^{k-1} \models \varphi[b, w_k]  \},
    \end{align*} 
    where $\rho_k = \rho_k(M)$, $w_k = w_k(M) = \la \rho_k(M), \eta_k(M), p_k(M) \ra$ and $\eta_k(M)$ is the $\Sigma_1$-cofinality of $\rho_k(M)$ over $M^{k-1}$. We also have the decoding function $d^k: M^k \rightarrow M$ and canonical $\Sigma_1$-Skolem function $h^1_{M^k}$ over $M^k$ defined as in \cite[Chapter 4]{normalization_comparison}. We have the $k+1$-st projectum, parameter, strong core $\bar{\mathfrak{C}}_{k+1}$, and core $\mathfrak{C}_{k+1}$ defined by (we will omit the $M$ from the notation)
\begin{align*}
	\rho_{k+1} &= \rho_1(M^k),\\
	p_{k+1} &= p_1(M^k),\\
	\bar{\mathfrak{C}}_{k+1} &= \text{transitive collapse of $d^k\circ h^1_{M^k}[(\rho_{k+1}\cup \{p_{k+1}, w_k\})]$,}\\
	\bar{p}_{k+1} &= \sigma^{-1}(p_{k+1}),\\
	\mathfrak{C}_{k+1} &= \text{transitive collapse of $d^k\circ h^1_{M^{k}}[(\rho_{k+1}\cup \{p_{k+1}, \rho_{k+1},w_k\})].$}	
\end{align*}
Here, we let $\sigma: \bar{\mathfrak{C}}_{k+1} \rightarrow M$ and $\pi: \mathfrak{C}_{k+1}\rightarrow M$ be the uncollapse maps.

    For $M$ of type 1, $M$ is $k+1$-solid iff 
$M^k$ is parameter solid, projectum solid, stable (see Definition \ref{type2core}), and $M$ is weakly ms-solid.\footnote{$M$ is weakly 
ms-solid iff either $M$ is passive or the last extender of $\mathfrak{C}_1$ and $\bar{\mathfrak{C}_1}$ 
satisfies the weak $\sf{ms}$-$\sf{ISC}$. $M$ satisfies the weak $\sf{ms}$-$\sf{ISC}$ if letting 
$E$ be the top extender of $M$ and $\kappa$ be the critical point of $E$, then the Jensen 
completion of $E_{\kappa}$ is on the sequence of $M|\lh(E)$. For the other components of solidity,
see Definition \ref{type2core} below.} $M$ is $k+1$-sound iff $M$ 
is $k+1$-solid and $M = \mathfrak{C}_{k+1}(M)$. $M$ is $k+1$-strongly sound if 
$M$ is $k+1$-sound and $M = \bar{\mathfrak{C}}_{k+1}(M)^-$\footnote{For a pfs 
premouse $N$, $N^-$ is just $N$ but with soundness degree $\textrm{deg}(N)-1$.}. 
    

    Let $M$ be a pfs premouse of type 1 and $\degr(M) = k < \omega$. We set
\begin{align*}
\rh_{k+1}(M) =\rh_1(M^k)  &= \begin{cases} 0 & \text{ if $M^k$ is strongly sound,}\\
                               \text{ least $\kappa$ s.t. $\kappa \notin 
\textrm{Hull}_1^{M^k}(\rho_1(M^k) \cup p_1(M^k))$} & \text{ otherwise.}\\
                            \end{cases}\\
\hat{\eta}_{k+1}^M = \hat{\eta}_1(M^k) &=
\textrm{cof}_1^{M^{k}}(\hat{\rho}_{k+1}(M)). 
\end{align*}
 We say that $M$ is \textit{almost sound} iff
    \begin{enumerate}[(a)]
    \item $M$ is solid;
    \item $M^k = \textrm{Hull}_1^{M^k}(\rho_1(M^k) \cup \{p_1(M^k)\cup \hat{\rho_1}(M^k)\})$;
    \item if $\rho_1(M^k) \leq \hat{\rho_1}(M^k)$, then letting 
    \begin{center}
    $(H,B) = \textrm{cHull}_1^{M^k}(\rho_1(M^k)\cup p_1(M^k))$,
    \end{center}
    with anticollapse map $\pi: (H,B)\rightarrow M^k$, we have $$M^k = \textrm{Ult}((H,B),D),$$ where $D$ is the order zero measure of $H$ on $\hat{\rho_1}(M^k)$ and $\pi = i_D$, and
    \item if $\rho_1(M^k) < \hat{\rho_1}(M^k)$ then $\hat{\eta_1}(M^k) < \rho_1(M^k)$.
    \end{enumerate}

  Now suppose $M$ is an acceptable $J$ structure, and for $k = \textrm{deg}(M)$, either
$k=0$ or $M^{k-1}$ is a pfs premouse of type 1. We say that $M$ is a pfs premouse of type $1A$ iff 
$\rho_k(M) = \rho_{k-1}(M)$ or $\rho_k(M)\in \text{Hull}_1^{M^{k-1}}(\rho_k(M)\cup p_k(M))$, 
equivalently iff $\textrm{deg}(M) = 0$ or $\textrm{deg}(M)>0$ and 
$M^-$ is strongly sound. $M$ has type $1B$ iff $\textrm{deg}(M)> 0$ and $M^-$ is
 sound, but not strongly sound. $M$ has type $2$ iff $M^-$ is almost sound,
 but not sound.\footnote{If $M$ is a pfs premouse of type 1B,
$k = \textrm{deg}(M)>0$, and $N=\Ult_k(M,E)$ where $\eta_k^M = \crit(E)$, then
 $\sup i_E``\rho_k(M) = \rho_k(N) < \hat{\rho}_k(N) =
i_E(\rho_k(M))$, so $N$ is a type 2 pfs premouse of degree $k$. Fine structural hulls of
type 1 premice can have type 2 as well.} So for $M$ a pfs premouse of degree $k$,
\begin{itemize}
\item
$M$ has type 1A iff $\rh_k(M)=0$,\\ 
\item $M$ has type 1B iff $\rh_k(M)=\rho_k(M)$, and\\
\item $M$ has type 2 iff $\rh_k(M)>\rho_k(M)$.
\end{itemize}
  
    Let $M$ be a pfs premouse of either type, with $k = \textrm{deg}(M) > 0$. We set 
\[
\hat{w}_k(M) = \la \hat{\eta}_k(M), \hat{\rho}_k(M), p_k(M)\ra,
\]
 where $\hat{\rho}_k(M)$ is the least $\rho$ that is not in Hull$_1^{M^{k-1}}(\rho_k(M)\cup p_k(M))$
 and $\hat{\eta}_k(M) = \textrm{cof}_1^{M^{k-1}}(\hat{\rho}_k(M))$. When $M$ is clear from the context,
 we write $\hat{w}_k$ for $\hat{w}_k(M)$ etc. Let 

 \begin{center}
    $\hat{A}^k_M = \{\la \varphi, b\ra \ | \  \varphi \textrm{ is } \Sigma_1 \wedge b 
\in M||\rho_k \wedge M^{k-1} \models \varphi[b, \hat{w}_k]  \}$,
    \end{center} 
and
\begin{center}
$\hat{M}^k = (M||\rho_k, \hat{A}^k_M)$.
\end{center}
We also have the decoding function $\hat{d}^k: M^k \rightarrow M$ and canonical $\Sigma_1$-Skolem 
function $h^1_{\hat{M}^k}$ over $\hat{M}^k$ defined as in \cite[Chapter 4]{normalization_comparison}.
 We write $\rho^-(M)$ for $\rho_{\textrm{deg}(M)-1}(M)$ etc.\footnote{ It is sometimes useful to have also a coding structure for
the strong core. Let $M$ be a pfs premouse and $k = \textrm{deg}(M)>0$; then 
\begin{center}
$B^k = \{\langle \varphi, b \rangle \ | \ \varphi \textrm{ is } \Sigma_1 \wedge b\in M||\rho_k \wedge M^{k-1} \models \varphi[b, p_k]\}$,
\end{center}
and the reduct of $\bar{\mathfrak{C}}_k(M)$ is defined as
\begin{center}
$M^k_0 = (M || \rho_k, B^k)$.
\end{center}}

We define solidity and soundness for type 2 premice exactly as we did
for type 1 premice, but with $\hM^k$ replacing $M^k$. Since $\hM^k = M^k$
when $M$ has type 1, the type 1 and type 2 definitions can be unified.
For the record:

\begin{definition}\label{type2core} Let $M$ be a premouse
of degree $k$; then
\begin{align*}
 \rho_{k+1}(M) &= \rho_1(\hM^k),\\
  p_{k+1}(M) &= p_1(\hM^k),\\
 \bar{\mfc}_{k+1}(M)
 &= \text{transitive collapse of $\hat{d}^k \circ h^1_{\hM^k}``(\rho_{k+1} \cup \lbrace p_{k+1} \rbrace)$, and}\\
 \mfc_{k+1}(M) & = \text{transitive collapse of
        $\hat{d}^k \circ h^1_{\hat{M}^k}``(\rho_{k+1} \cup \lbrace p_{k+1}, \rho_{k+1}\rbrace)$.}
\end{align*}
Let $\sigma \colon \bar{\mfc}_{k+1} \to M$ and $\pi \colon \mfc_{k+1} \to M$
be the anticollapse maps, and
$\bar{p}_{k+1} = \sigma^{-1}(p_{k+1})$; then
\begin{itemize}
        \item[(a)]$\hM^k$ is {\em parameter solid} iff
$p_{k+1}$ is solid and universal
over $\hM^k$ and $\bar{p}_{k+1}$ is
solid and universal over the reduct
$(\bar{\mfc}_{k+1})^k$ of $\bar{\mfc}_{k+1}$.
\item[(b)] $\hM^k$ is {\em projectum solid}
iff $\rho_{k+1}$ is not measurable by the $M$-sequence, and either $\bar{\mfc}_{k+1} = \mfc_{k+1}$,
or $\mfc_{k+1} = \Ult(\bar{\mfc}_{k+1},D)$ and $\sigma = \pi \circ i_D$,
for $i_D$ the order zero measure of $\bar{\mfc}_{k+1}$ on $\rho_{k+1}$.
\item[(c)]
$\hM^k$ is {\em stable} iff either
$\hat{\eta}_k^M < \rho_{k+1}$, or $\hat{\eta}_k^M$ is not measurable
by the $M$-sequence.
\item[(d)]  We say that $M$ is \textit{projectum stable} if $\eta_k(M)$ is not measurable by the $M$-sequence. 
\item[(e)] We say that $M$ has {\em stable type 1} iff $M$ has type 1A, or $M$ has type 1B
and is projectum stable.
    \end{itemize}
We say that $M$ is {\em solid} (or {\em k+1-solid}) iff $\hM^k$ is parameter solid,
projectum solid, and stable.
We say that $M$ is {\em sound} (or $k+1$-sound)
iff $M$ is solid and $M = \mfc_{k+1}(M)$.\footnote{ Similarly, if $k=\degr(M)$,
then $\rho(M) = \rho_{k+1}(M)$, $\rh(M) = \rh_{k+1}(M)$, $p(M) = p_{k+1}(M)$,
and so on for the other $k$-free notations.}
\end{definition}

      It is easy to see that $M$ has stable type 1 iff $M$ has type 1,
and for any $E$ on the $M$-sequence, letting $P \unlhd M$ be longest
such that $\Ult(P,E)$ is defined, $\Ult(P,E)$ has type 1.

\subsection{Elementarity of maps}

Suppose that $M$ is an lpm, and $\pi \colon H \to M$. What sort of
elementarity  for $\pi$ do we need to conclude that $H$ is an lpm?
In the proof of square for ordinary mice, we have to deal with embeddings that
are only weakly elementary.\footnote{See section 1.4 of \cite{normalization_comparison}
for a discussion of the degrees of elementarity. If $\textrm{deg}(H)=\textrm{deg}(M)=0$, then
$\pi$ is weakly elementary iff it is $\Sigma_0$ elementary and cardinal-preserving.} In the context of proving square in pfs premice, we need a slight strengthening of weak elementarity, called {\em near elementarity} and 
defined in \cite[Chapters 2,4]{normalization_comparison}. Nearly elementary maps are produced by
 lifting constructions, and they will occur in the square construction.

Roughly a map $\pi: H\rightarrow M$ with $k = \textrm{deg}(M) = \textrm{deg}(H)$ is
 nearly elementary if it is weakly elementary and maps 
$\hat{\eta}_k(H)$ to $\hat{\eta}_k(M)$. More precisely, 
we say that $\pi$ is \textrm{nearly elementary} if $\pi$ is the 
completion of $\pi\restriction \hat{H}^k$ and $\pi\restriction \hat{H}^k$ is 
a $\Sigma_0$-preserving and cardinal preserving map from $\hat{H}^k$ to
  $\hat{M}^k$. $\pi$ is {\em elementary} if it is nearly elementary and $\pi\restriction \hat{H}^k$ is $\Sigma_1$-elementary. 

We note in the above that 
\begin{enumerate}[(a)]
\item If $H$ is of type 1A, then $H^-$ is strongly sound and $\hat{w}_k(H)
 = \langle 0,0,p_k(H) \rangle$. In this case, $M$ is of type 1A, 
but $\pi$ may or may not preserve $\rho_k(H)$.
\item If $H$ is of type 1B, then $\rho_k(H) = \hat{\rho}_k(H)$ and hence
 $\hat{w}_k(H) = w_k(H)$. $M$ may have type 1B or type 2. $M$ is of type 1B 
if and only if $\hat{w}_k(M)=w_k(M)$ if and only if $\pi(w_k(H)) = w_k(M)$.
\item If $H$ is of type 2, then $M$ is of type 1B or 2.
\end{enumerate}

The existence of a nearly elementary $\pi \colon H \to M$ does not imply 
that $H$ is a premouse when $\textrm{deg}(H)=\textrm{deg}(M)=0$. 
If $M$ is a passive lpm, then so is $H$, and there is no problem.
If $M$ is extender-active, then it could be that $H$ is only a protomouse,
in that its last extender predicate is not total. The problem here is
solved by the parts of the Schimmerling-Zeman
proof related to protomice, which work in our context. Finally, we must consider
the case that $M$ is branch-active.

   \begin{definition}\label{qformula} A {\em $rQ$-formula} of $\mathcal{L}_1$ is a
           conjunction of formulae\index{$rQ$-formula}
           of the form
   \begin{itemize}
                   \item[(a)]$\forall u \exists v ( u \subseteq v \wedge \varphi)$, where $\varphi$
           is a $\Sigma_1$ formula of $\mathcal{L}_1$ such that $u$ does not occur
           free in $\varphi$,
   \end{itemize}
 or of the form
           \begin{itemize}
                   \item[(b)]``$\df \neq \emptyset$, and for
           $\mu = \crit(\df)^+$, there are cofinally many $\xi < \mu$ such that
           $\psi$", where $\psi$ is $\Sigma_1$.
           \end{itemize}
   \end{definition}

Formulae of type (a) are usually called $Q$-formulae. Being a passive
lpm can be expressed by a $Q$-sentence, but in order to express being
an extender-active lpm, we need type (b) clauses, in order to say
that the last extender is total. $rQ$ formulae are $\Pi_2$, and hence
preserved downward under $\Sigma_1$-elementary maps. They are preserved
upward under $\Sigma_0$ maps that are {\em strongly cofinal}.

\begin{definition}\label{stronglycofinal}
        Let $M$ and $N$ be $\mathcal{L}_0$-structures
        and $\pi \colon M \to N$ be $\Sigma_0$ and cofinal. We say
        that $\pi$ is {\em strongly cofinal} iff $M$ and $N$ are not extender
        active, or $M$ and $N$ are extender active, and
        $\pi `` (\crit(\df)^+)^M$ is cofinal in $(\crit(\df)^+)^N$.
\end{definition}

It is easy to see that

\begin{lemma}\label{rqpreserved} $rQ$ formulae are preserved downward
        under $\Sigma_1$-elementary maps, and upward under strongly
        cofinal $\Sigma_0$-elementary maps.
\end{lemma}

\begin{lemma} \label{lem:B_active_case}
        \begin{itemize}
\item[(a)] There is a $Q$-sentence $\varphi$ of $\mathcal{L}_1$
such that for all transitive $\mathcal{L}_0$ structures $M$,
$M \models \varphi$ iff $M$ is a passive lpm.
\item[(b)] There is a $rQ$-sentence $\varphi$ of $\mathcal{L}_1$
such that for all transitive $\mathcal{L}_0$ structures $M$,
$M \models \varphi$ iff $M$ is an extender-active lpm.
\item[(c)] There is a $Q$-sentence $\varphi$ of $\mathcal{L}_1$
such that for all transitive $\mathcal{L}_0$ structures $M$,
$M \models \varphi$ iff $M$ is a potential branch-active lpm.
        \end{itemize}
\end{lemma}

\begin{proof} (Sketch.) We omit the proofs of (a) and (b).
For (c), note that ``$\dot{B} \neq \emptyset$" is
$\Sigma_1$. One can go on then to say with a $\Sigma_1$ sentence
that if $\eta$ is least in $\dot{B}$,
then $M|\eta$ is admissible, and $s^M$ exists. One can say with a $\Pi_1$
sentence that $\lbrace \alpha \mid \dot{B}(\eta + \alpha) \rbrace$ is a branch
of $\T(s)$, perhaps of successor order type. One can say that
$\dot{B}$ is cofinal in the ordinals with a $Q$-sentence. Collectively,
these sentences express the conditions on potential lpm-hood related to $\dot{B}$.
That the rest of $M$ constitutes an extender-passive lpm can be expressed
by a $\Pi_1$ sentence.
\end{proof}

\begin{corollary} \label{preservation}
        \begin{itemize}
\item[(a)]If $M$ is a passive ( resp. extender-active, potential branch-active )
lpm, and $\Ult_0(M,E)$ is wellfounded,
 then $\Ult_0(M,E)$ is a passive (resp.extender-active, potential branch-active ) lpm.
\item[(b)]
Suppose that $M$ is a passive (resp. extender-active, potential branch-active) lpm,
 and $\pi \colon H \to M$ is $\Sigma_1$-elementary; then $H$ is a
 passive (resp. potential branch-active)  lpm.
\item[(c)] Let $\textrm{deg}(M)=\textrm{deg}(H) = 0$, and $\pi \colon H \to M$ be $\Sigma_2$ elementary;
                          then $H$ is a branch-active lpm iff $M$ is a branch-active lpm.
        \end{itemize}
\end{corollary}

        \begin{proof}
$rQ$-sentences are preserved upward by strongly cofinal $\Sigma_0$
 embeddings, so we have (a). They are $\Pi_2$, hence preserved downward by $\Sigma_1$-
 elementary embeddings, so we have (b).

It is easy to see that honesty is expressible by a Boolean combination
of $\Sigma_2$ sentences, so we get (c).

        \end{proof}

\begin{remark}{\rm It could happen that $M$ is a branch-active lpm,
$\pi \colon H \to M$ is cofinal and elementary (with $\textrm{deg}(M)=\textrm{deg}(H)=0$), and $b^M$
is not cofinal in $\T(s^M)$, but $b^H$ is cofinal in $\T(s^H)$. If
we were using the branch coding in the original version of
 \cite{normalization_comparison}, then $\db^M = \emptyset$, so $\db^H =\emptyset$, so
$H$ is not an lpm.}
\end{remark}

Part (c) of Lemma \ref{lem:B_active_case} is not particularly useful. In general,
our embeddings will preserve honesty of a potential branch active lpm  $M$ 
because $\dot{\Sigma}^M$ and $\dot{B}^M$ are determined by a complete iteration
strategy for $M$ that has strong hull condensation. So the more useful preservation
theorem in the branch-active case applies to {\em hod pairs}, rather than to hod premice.

\subsection{Plus trees}

The iteration trees we use below are $\lambda$-separated plus trees. 
These notions are defined in detail in \cite[Section 4.4]{normalization_comparison} and we briefly 
summarize the relevant concepts here. Suppose $M$ is a pfs premouse and $E$ is an extender on the $M$-sequence, 
then
\begin{itemize}
\item $E^+$ is the extender with generators $\lambda_E\cup \{\lambda_E\}$ that represents 
$i_F^{\textrm{Ult}(M,E)}\circ i^M_E$, where $F$ is the order zero total measure on $\lambda_E$ in Ult$(M,E)$.
\item $\hat{\lambda}(E^+) = \lambda_E$.
\item lh$(E^+) = \textrm{lh}(E)$.
\item $o(E^+) = (\textrm{lh}(E)^+)^{\textrm{Ult}(M,E^+)}$.
\end{itemize}
We say that an extender $G$ is \textit{of plus type} if $G = E^+$ for some extender $E$ 
on the sequence of a pfs premouse $M$; we let $G^- = E$. In general, if $E$ is an extender
 (of plus type or not)
\begin{itemize}
\item we let $\varepsilon(E) = \textrm{lh}(E)$ if $E$ is of plus type; otherwise,
 $\varepsilon(E) = \lambda(E)$.
\item if $E$ is on the sequence of some premouse, then 
\begin{enumerate}[(i)]
\item $\hat{\lambda}(E)= \lambda(E) = \hat{\lambda}(E^+)$,
\item $E^- = E$.
\end{enumerate}
\end{itemize}
The {\em extended $M$-sequence} consists of all $E$ such that
$E^-$ is on the $M$-sequence.

A \textit{plus tree} $\itT$ on a pfs premouse is like an ordinary normal tree, except that
\begin{itemize}
\item[(i)] We only require that $E_\alpha^\itT$ be on the extended
$\itM_\alpha^\itT$ sequence,
\item[(ii)] $E_\alpha^\itT$ is applied to the longest possible initial
segment of $\itM_\beta^\itT$, where $\beta$ is least such that
$\crit(E_\alpha^\itT) < \hat{\lambda}(E_\beta^\itT)$, and
\item[(iii)] the length-increasing condition is weakened slightly.\footnote{The length-increasing condition
is enough to guarantee that $T\tpred(\alpha+1)$ is the least $\beta$ such that
$\crit(E_\alpha^\itT) < \vep(E_\beta^\itT)$. Thus none of the generators of a plus extender $E$,
including the generator $\hat{\lambda}(E)$, are moved later on a branch in which $E$ has been used.}
\end{itemize}
See \cite[Definition 4.4.3]{normalization_comparison} for the complete definition.

 A \textit{$\lambda$-separated} tree is a plus tree in which every extender 
used along the tree is of plus type. The weakening in (iii) above does not
affect $\lambda$-separated trees; that is,
the lengths of the extenders used in a $\lambda$-separated tree are strictly increasing. Moreover, quasi-normalization
coincides with embedding normalization on stacks of $\lambda$-separated plus trees. \cite[Section 8.1]{normalization_comparison} 
shows that $\lambda$-separated trees are enough for comparisons. For these and other reasons it is convenient to restrict 
one's attention to the way an iteration strategy $\Sigma$ acts on stacks of $\lambda$-separated trees.
By Lemma 9.3.2 of \cite{normalization_comparison}, if $(P,\Sigma)$ is a mouse pair, then $\Sigma$ is determined
by its action on countable $\lambda$-separated trees. 

\subsection{Mouse pairs and Dodd-Jensen}

Two of the main definitions from \cite{normalization_comparison} are 

\begin{definition} $(M,\Omega)$ is a 
\textit{pure extender pair} with scope $H_\delta$ iff
\begin{itemize}
\item[(a)] $M$ is a pure extender pfs premouse.
\item[(b)] $\Omega$ is a complete $(\omega,\delta)$ iteration strategy for $M$\footnote{See
\cite[4.6.3]{normalization_comparison}.}, and
\item $\Omega$ is internally lift-consistent, quasi-normalizes well, and has strong
	hull condensation.\footnote{See \cite[5.4.4, 7.1.1, 7.1.9]{normalization_comparison}.}
\end{itemize}
\end{definition}

\begin{definition} \label{lbrhodpairdef}$(M,\Omega)$ is a 
\textit{least branch hod pair (lbr hod pair)} with scope $H_\delta$ iff
\begin{itemize}
\item[(a)] $M$ is an lpm.
\item[(b)] $\Omega$ is a complete $(\omega,\delta)$ iteration strategy for $M$,
\item[(c)] $\Omega$ is internally lift-consistent, quasi-normalizes well, and has strong
	hull condensation, and
\item[(d)] $\Omega$ is pushforward consistent, that is if $s$ is by $\Omega$ with last model $N$, then
 $\dot{\Sigma}^N\subseteq \Omega_s$, where $\Omega_s(t) = \Omega(s^\smallfrown t)$.\footnote{See
\cite[9.2.1]{normalization_comparison}.}
\end{itemize}
\end{definition}

\begin{definition} $(M,\Omega)$ is a {\em mouse pair} iff it is either a
pure extender pair or an lbr hod pair.
\end{definition}

\begin{remark}{\rm \cite{normalization_comparison} required that  $M$ have
type 1 in order for $(M,\Omega)$ to qualify as a mouse pair, because that
was sufficient generality for its purposes. Here we are allowing $M$ to have
type 2.}
\end{remark}

Included in \ref{lbrhodpairdef}(b) is the requirement that all $\Omega$-iterates of 
$M$ be least branch premice. Because of our honesty requirement
in the branch-active case, this no longer follows automatically from
the elementarity of the iteration maps. That the iterates of $M$ are honest comes out of the
construction of $\Omega$, as a consequence of pushforward consistency.

If  $(M,\Omega)$
is an lbr hod pair  and $\pi \colon H \to M$ is 
nearly elementary, then $\Omega^\pi$ is the {\em pullback strategy},
given by 
\begin{center}
$\Omega^\pi(s) = \Omega(\pi s).$
\end{center} 
We show now that, except in the protomouse case,
$(H,\Omega^\pi)$ is an lbr hod pair.\footnote{This  is
Lemma 9.2.3 of \cite{normalization_comparison}.}

\begin{lemma}\label{preservelbrhodpairs1} Let $(M,\Omega)$ be an lbr
hod pair with scope $H_\delta$, and let $\pi \colon H \to M$ be nearly
elementary. Suppose that one of the following holds:
\begin{itemize}
\item[(a)] $M$ is passive or branch-active, or
\item[(b)] $H$ is an lpm.
\end{itemize}
Then $(H,\Omega^\pi)$ is an lbr hod pair with scope $H_\delta$.
\end{lemma}

\begin{proof} We show first that $H$ is an lpm. If (b) holds,
this is rather easy. If $M$ is passive, we can apply
(a) of \ref{lem:B_active_case}, noting that $Q$ sentences go down
under nearly elementary embeddings. So let us assume that
$M$ is branch-active.

By (b) of \ref{lem:B_active_case}, $H$ is a potential branch active
lpm. So we just need to see that $H$ is honest.
Let $\nu = \nu^H$, $b = b^H$, and $\T = \T(s^H)$. If $\nu = lh(\T)$,
there is nothing to show,
so assume $\nu < lh(\T)$. We must show
that $b = [0,\nu)_{T}$. We have by induction that for $N=H|\eta^H$,
$(N,\Omega^\pi_N )$ is an lbr hod pair. Thus $\T$ is by $\Omega^\pi$, and so we just need to see
that for $\U = \T \within \nu$ and $\mathcal{W} =\U ^\frown b$,
$\mathcal{W}$ is by $\Omega^\pi$, or equivalently,
that $\pi\mathcal{W}$ is by $\Omega$. But it is easy to see that
$\pi\mathcal{W}$ is a psuedo-hull of $\pi(\U)^\frown b^M$, and $\Omega$
has strong hull condensation, so we are done.

For the proof that $(H,\Omega^\pi)$ is internally lift-consistent, normalizes well,
and has strong hull condensation, the reader should see
\cite{normalization_comparison}. We give here the proof that 
$(H,\Omega^\pi)$ is pushforward consistent, because it extends the honesty
proof given above.
 
Let $P$ be an $\Omega^\pi$ iterate of $H$ via the stack of trees
$s$. Let $Q$ be the corresponding $\Omega$ iterate of $M$ via 
$\pi s$, and let $\tau \colon P \to Q$ be the nearly
elementary copy map. Then for $\U \in P$,
\begin{align*}
\U \text{ is by } \dot{\Sigma}^P & \Rightarrow \tau(\U) \text{ is by } \dot{\Sigma}^Q\\
& \Rightarrow \tau\U \text{ is by } \Omega_{\pi s,Q}\\
& \Rightarrow \U \text{ is by } (\Omega^\pi)_{s,P},\\
\end{align*}
as desired.
 
\end{proof}

The basic results of inner model theory, such as the Comparison Lemma and the
Dodd-Jensen Lemma, are better stated and proved as results about mouse pairs
than as results about mice, with the notions of elementary submodel and
iterate adjusted so that this is possible. For example,
if $(H,\Psi)$ and $(M,\Sigma)$ are mouse pairs,
then $\pi \colon (H,\Psi) \to (M,\Sigma)$ is elementary
(resp. nearly elementary) iff $\pi$ is elementary (nearly elementary)
as a map from $H$ to  $M$, and $\Psi = \Sigma^\pi$. We say that
$(M,\Sigma)$ is an {\em iterate of $(H,\Psi)$} iff there is a stack
$s$ on $H$ such that $s$ is by $\Psi$,
and $\Sigma = \Psi_s$. It is a {\em non-dropping iterate} iff the
branch $H$-to-$M$ does not drop. Assuming $\adp$ and that
our pairs have scope $\hc$, \cite{normalization_comparison}
proves the following:
\begin{itemize}
\item[(1)] If $(M,\Sigma)$ is a mouse pair, $H$ is a premouse, and $\pi \colon H \to M$
is nearly elementary, then $(H,\Sigma^\pi)$ is a mouse pair.
\item[(2)] If $(H,\Psi)$ is a mouse pair, and $(M,\Sigma)$ is a
non-dropping iterate of $(H,\Psi)$, then
the iteration map $i_s \colon (H,\Psi)
\to (M,\Sigma)$ is elementary in the category of pairs.
\item[(3)] (Dodd-Jensen) If $(H,\Psi)$ is a mouse pair,
$(M,\Sigma)$ is an iterate of $(H,\Psi)$ via the stack $s$,
and $\pi \colon (H,\Psi) \to (M,\Sigma)$ is nearly elementary,
then
\begin{itemize}
\item[(i)] the branch $H$-to-$M$ of $s$ does not drop, and
\item[(ii)] for all $\eta < o(H)$, $i_s(\eta) \le \pi(\eta)$,
where $i_s$ is the iteration map.
\end{itemize}
\item[(4)](Mouse order) Let $(H,\Psi) \le^* (M,\Sigma)$ iff
there is a nearly elementary embedding of $(H,\Psi)$ into
some iterate of $(M,\Sigma)$; then $\le^*$ is a prewellorder
of the mouse pairs with scope $\hc$ in each of the two types.
\end{itemize}

The prelinearity of the mouse pair order is the content of the Comparison Lemma
for mouse pairs. For pure extender pairs, it is proved
in Theorem 8.4.5 of \cite{normalization_comparison}. The proof
for lbr hod pairs is basically the same; it is
Theorem 9.5.10 of \cite{normalization_comparison}.

\section{COMPARING STRATEGIES FOR UNSOUND MICE}\label{sec:type2comparison}

Let us assume $\adp$ throughout this section.

The comparison theorems of \cite{normalization_comparison}
are stated and proved for mouse pairs of stable type 1. (This includes
all mouse pairs of degree 0.) Mouse pairs that are not of stable type 1
(that is, those of type 2, or of type 1B and not projectum stable)
can
arise in our fine condensation results, so we need a comparison process that applies to them as well.
In this section we generalize the process of \cite{normalization_comparison}
to such pairs. 

The problem is that all levels of a background construction
are type 1 pairs, so a type 2 pair $(P,\Sigma)$ cannot literally iterate into a
level of a background construction.\footnote{We shall see that a projectum stable pair of type 1B
cannot do so either.} In the end, our solution to this problem
amounts to making small adjustments to the formulation
and proof of $(*)(P,\Sigma)$ in \cite{normalization_comparison}. Since there is nothing very new
involved, we shall omit some details.\footnote{\cite[4.6.12]{normalization_comparison} states
a comparison theorem for pure extender pfs mice that are are not of stable type 1. There is
no discussion of strategy
comparison in this case.}

A premouse $M$ has type 2 iff $M^-$ is almost sound, but not sound.
The paradigm example is $M=(\Ult_{k+1}(P,E),k+1)$, where $P$ has type 1B and
degree $k+1$, and 
the canonical embedding
$i_E^P$ is discontinuous at $\rho_{k+1}(P)$. The point of setting $\degr(M)=k+1$ here,
rather than $k$, is that we want to take $k+1$-ultrapowers of $M$,
and $\degr(M) = k+1$ fits better with our definitions of
elementarity and plus trees in that context. 

   However, almost soundness records a
number of properties of $M^-$ that are peculiar to this particular
way of producing unsound mice. If we are going to iterate almost
sound type 1 premice $N$ at the $\degr(N)+1$ level, and compare their iteration
strategies, then it is more natural to go all the way, and look
at $\degr(N)+1$-strategies for arbitrary solid $N$.\footnote{For one thing, it can happen
that $(P,\Sigma)$ and $(Q,\Lambda)$ are type 2 pairs, but the result of comparing them
is some $(R,\Omega)$ that is not type 2, because $\rh^-(P)$ and $\rh^-(Q)$ are
mapped to different points in $R$. In this case $R$ is two order zero ultrapowers away from 
its type 1 core, not just one.}

     So we shall consider arbitrary solid type 1 premice $N$,
and look at $\degr(N)+1$-level iterations of them.\footnote{One could look at
$\degr(N)+2$-level iterations, etc. This is done in a rudimentary way in
\cite{fsit}, and fully and systematically in Jensen's $\Sigma^*$ theory.
\cite{finetame} proves the basic results about the $\degr(N)+1$ case.}
The solidity hypothesis is useful at several points, and we are
building on the theory of \cite{normalization_comparison} rather than
trying to re-do it, so there is no value in dropping solidity.

\subsection{+1-iteration trees}

\begin{definition}\label{ult1} Let $N$ be a solid premouse
and $\degr(N)=k$, and let $E$ be an extender over $N$
such that $\crit(E) < \rho_{k+1}(N)$; then
\begin{itemize}
\item[(a)]  $\Ult_{k+1}(N,E)$
is the decoding of $\Ult_1(N^k,E)$, where the latter ultrapower
is formed using all boldface $\Sigma_1^{N^k}$ functions.
\item[(b)]
We set $\degr(\Ult_{k+1}(N,E)) = k$.
\item[C)] 
$i_E^{N^k} \colon N^k \to \Ult_1(N^k,E)$ is the canonical embedding,
and $i_E^N \colon N \to \Ult_{k+1}(N,E)$ is its completion.
\end{itemize}
\end{definition}

Familiar calculations show

\begin{lemma}\label{ult1elementarity}
Let $N$ be a solid premouse
and $\degr(N)=k$, and let $P=\Ult_{k+1}(N,E)$,
where $E$ is an extender over $N$
such that $\crit(E) < \rho_{k+1}(N)$. Let $i =  i_E^N$ and
$\hat{i} = i \restriction N^k = i_E^{N^k}$. Suppose that $P$ is wellfounded; then
\begin{itemize}
\item[(1)] $\hat{i}$ is $\Sigma_2$-elementary as a map
from $N^k$ to $P^k$.
\item[(2)] $P$ is a type 1 premouse of degree $k$, and $P^k = \Ult_1(N^k,E)$.
\item[(3)] If $\textrm{Th}_1^{N^k}(\alpha \cup \lbrace q \rbrace) \in N^k$,
then $i(\textrm{Th}_1^{N^k}(\alpha \cup \lbrace q \rbrace)) =
\textrm{Th}_1^{P^k}(i(\alpha) \cup \lbrace i(q) \rbrace).$
\item[(4)] $i(\rho_k(N)) = \rho_k(P)$ and $i(\eta_k^N) = \eta_k)$.
\item[(5)] $P$ is solid, $i(p_1(N^k)) = p_1(P^k)$, and
$\rho_1(P^k) = \sup i`` \rho_1(N^k)$.
\end{itemize}
\end{lemma}
\begin{proof}
Of course, Los's theorem holds for $\Ult_1(N^k,E)$
for all $\Sigma_1$ formulae, and the canonical embedding $\hat{i}$
is $\Sigma_2$ elementary. So we have (1). Since
``$z = \textrm{Th}_1(\alpha \cup \lbrace q \rbrace)$" is $\Pi_2$, we also
get (3).

(2) follows from the Upward Extension of Embeddings Lemma.\footnote{ See \cite[4.3.7(b)]{normalization_comparison}.}

For (4): ``$z= \rho_k(N)$" is equivalent to $N^{k-1} \models \theta[z]$,
where $\theta$ is a Boolean combination
of $\Pi_3$ formulae.\footnote{See \cite[2.4.5--2.4.7]{normalization_comparison}.}
Since $i$ is $\Sigma_3$ elementary as a map from
$N^{k-1}$ to $P^{k-1}$, $i(\rho_k(N)) = \rho_k(P)$. By Lemma 4.3.4 of
\cite{normalization_comparison} we get that $i(\eta_k^N)=\eta_k^P$ as well.

For (5): We claim first that $\sup i`` \rho_1(N^k) \le \rho_1^{P^k}$. For if $\crit(E) \le \alpha < \rho_1(N^k)$
and $r = [a,f]^{N^k}_E$ where $f$ is $\Sigma_1^{N^k}$ in $q$, then $\textrm{Th}_1^{P^k}(i(\alpha)
\cup \lbrace r \rbrace)$ can be easily computed from 
 $\textrm{Th}_1^{P^k}(i(\alpha)
\cup \lbrace i(q) \rbrace)$, and the latter is in $P^k$ by (3). We claim next that
$\rho_1(P^k) \le \sup i`` \rho_1(N^k)$. For let $q= p_1(N^k)$ and $\rho = \rho_1(N^k)$; then
the fact that $N$ is solid implies there is a $\Sigma_1^{N^k}(q)$ map
of $\rho$ onto $\rho^{+,N}$, and since $\vep(E) < \sup i``\rho$, this yields
a $\Sigma_1^{P^k}(i(q))$ map of $\sup i``\rho$ onto $(\sup i``\rho)^{+,P}$,
so that
$\textrm{Th}_1^{P^k}(\sup i``\rho
\cup \lbrace i(q) \rbrace) \notin P^k$.

The calculations just done show that $p_1(P^k) \le_{\lex}  i(p_1(N^k))$,
and $i(p_1(N^k)) \le_{\lex} p_1(P^k)$ follows from the preservation of
solidity witnesses given by (3).

Let us show that $P$ is solid. We have already shown that $P$ is parameter solid.
Let $\rho=\rho_1(N^k)$; then since $N$ is projectum solid, $\rho$ is not measurable in $N$,
so $i(\rho)$ is not measurable in $P$, so if $i$ is continuous at $\rho$,
then $P$ is projectum solid. So assume $i$ is discontinuous at $\rho$; then we have
a boldface $\Sigma_1^{N^k}$ function $f \colon \kappa_E \to \rho$ such that
$f$ is order preserving and continuous at limits. This implies $\rho$ is a limit
cardinal in $N$, and $f \within \alpha \in N$ for all $\alpha < \kappa_E$, so that
$N \models ``f(\alpha)$ is singular" for all $\alpha$. But then
$\sup i``\rho = [\lbrace \kappa_E \rbrace,f]^{N^k}_E$ is singular is $P$, and hence
not measurable in $P$.

Finally, we must show that $P$ is stable. But $\eta_k^P = i(\eta_k^N)$,
so if $\eta_k^P$ is measurable in $P$,
then $\eta_k^N$ is measurable in $N$, so $\eta_k^N < \rho_{k+1}(N)$ by stability of
$N$, so $\eta_k^P < \rho_{k+1}(P)$ by (4), as desired.
\end{proof}

Let us turn the lemma into a definition.

\begin{definition}\label{plusoneelementary} Let $i \colon N \to P$,
where $\degr(N) = \degr(P)=k$; then we say $i$ is {\em +1-elementary}
iff $N$ and $P$ are solid premice of type 1, and conclusions (1), (3),
(4), and (5) of Lemma \ref{ult1elementarity} are satisfied by $i$, $N$, and $P$.
\end{definition}

Now we look at iteration trees $\itT$ in which such ``soundness degree +1" ultrapowers
are taken. We record the degrees of the ultrapowers that can be 
taken in the degrees $\degr^\itT(\alpha)$
of the nodes, so now $\degr^\itT(\alpha) = \degr(\itM_\alpha^\itT) +1$ is possible. 

\begin{definition}\label{+1tree} Let $M$ be a solid premouse of degree $k$;
then a {\em 1-bounded plus tree on $M$} is a system
$\la \itT,\la E_\alpha^\itT \mid \alpha+1 < \lh(\itT) \ra, \degr^\itT \ra$
such that $\dom(\degr^\itT) = \lh(\itT)$, and there are
$M_\alpha$, $i_{\alpha,\beta}$, $D$ satisfying
\begin{itemize}
\item[(1)] If $\degr^{\itT}(0) = \degr(M)$, then $\itT$ is a plus tree on $M$
in the sense of \cite[4.4.3]{normalization_comparison}, with models, embeddings,
and dropset $M_\alpha$, $i_{\alpha,\beta}$, $D$. In this case, we also put
$0 \in D$, and say that $\itT$ has a {\em small drop} at 0.
\item[(2)] If $\degr^{\itT}(0) = \degr(M)+1 = k+1$, then
$M_\alpha$, $i_{\alpha,\beta}$, $D$ satisfy
all the properties  listed in \cite[4.4.3]{normalization_comparison}, except
that clause (3)(b) is modified so that if $T \tpred(\alpha+1)=\beta$,
$[0,\beta]_T \cap D = \emptyset$, and $\crit(E_\alpha^\itT) < \rho_k(M_\beta)$, then
\begin{itemize}
\item[(a)] if $\crit(E_\alpha^\itT) < \rho_{k+1}(M_\beta)$, then
$M_{\alpha+1} = \Ult_{k+1}(M_\beta,E_\alpha^\itT)$, and
\item[(b)] if $\rho_{k+1}(M_\beta) \le \crit(E_\alpha^\itT)$, then
$M_{\alpha+1} = \Ult_{k}(M_\beta,E_\alpha^\itT)$. 
\item[(c)] In  case (b), 
we put $\alpha+1 \in D^\itT$ and say that $\itT$ has a {\em small drop}
at $\alpha+1$.
\end{itemize}
\item[(3)] $\beta$ is a {\em +1-node} of $\itT$ iff $[0,\beta]_T \cap D^\itT
=\emptyset$. If $\beta$ is a +1-node, then we set
$\degr^\itT(\beta) = \degr(M_\beta)+1$. Otherwise, we set
$\degr^{\itT}(\beta)=\degr(M_\beta)$.
\item[(4)] $\itT$ is {\em 1-maximal} iff 0 is a +1-node of $\itT$.
If $\itT$ is not 1-maximal, then we say it is {\em 0-bounded}.
\end{itemize}
\end{definition}

If $\itT$ is 1-bounded, there is at most one small drop along any branch,
and if there is a small drop, it must be the first drop along that branch.
We call a drop $\alpha+1 \in D^\itT$ {\em large} iff it is not small.

\begin{remark}\label{largedropunsound}{\rm It is easy to see that
$\itM_\alpha^\itT$ is unsound (i.e. not $\degr(\itM_\alpha^\itT)$ sound)
iff $[0,\alpha]_T$ has a large drop.}
\end{remark}

In a 1-bounded tree, all the drops are forced, including the small ones,
except possibly a small drop
at 0. If there is such a drop, then $\itT$ is 0-bounded, and therefore it
is just a plus tree in
the sense of \cite{normalization_comparison}. In general, a 1-maximal
plus tree may not be a plus tree in the sense of \cite{normalization_comparison}, but it seems
better to expand the meaning of ``plus tree" than to invent a new term. In a 1-bounded tree,
after  a branch has dropped in
any way, the later
ultrapowers on the branch are always $n$-ultrapowers of $n$-sound premice.\footnote{We could relax
clause (2) by allowing $M_{\alpha+1} = \Ult_{k}(M_\beta,E_\alpha^\itT)$ when
$\crit(E_\alpha^\itT) < \rho_{k+1}(M_\beta)$. Only a few things become more complicated.}
If $M$ is not just solid, but sound, then a 1-maximal plus tree on $M$ is essentially
the same thing as a plus tree on $M^+$ in the sense of \cite{normalization_comparison}.

From Lemma \ref{ult1elementarity} we get
\begin{lemma}\label{+1tree.branchelem} Let $\itT$ be a 1-bounded plus tree on
the solid premouse $M$,
$\alpha <_T \beta$, and $D^\itT \cap [\alpha,\beta]_T = \emptyset$;
then $i_{\alpha,\beta}^\itT$ is elementary, and if $\itT$ is 1-maximal, then
it is exact.
If in addition  $\beta$ is a +1-node, then
$i_{\alpha,\beta}^\itT$ is +1-elementary.
\end{lemma}
Exactness uses that $\itT$ is 1-maximal and $M$ is stable, for otherwise
we might take a $\degr(M)$-ultrapower with critical point
$\eta_k^M$, producing thereby an inexact canonical embedding.

Along branches that have only a small drop, the embeddings are slightly
less elementary.

\begin{definition}\label{nearplusone.elementary} Let $M$ and $N$ be solid premice, $k = \degr(M)=\degr(N)$,
and $\pi \colon M \to N$. We say that $\pi$ is {\em nearly +1-elementary}
iff 
\begin{itemize}
\item[(i)] $\pi$ is elementary and exact,
\item[(ii)] $\pi``\rho_1^{M^k} \subseteq \rho_1^{N^k}$, $\pi(p_1(M^k)) = p_1(N^k)$, and
\item[(iii)] $\pi$ is {\em $k+1$-theory preserving}, in that
for all $\alpha < \rho_1^{M^k}$ and all $q \in M^k$,
$\pi(\textrm{Th}_1^{M^k}(\alpha \cup \lbrace q \rbrace)) =
\textrm{Th}_1^{N^k}(\pi(\alpha) \cup \lbrace \pi(q) \rbrace)$.
\end{itemize}
\end{definition}

Nearly +1-elementary $\pi \colon M \to N$ are called {\em $\deg(M)+1$-apt}
in \cite{finetame}. That paper proves a copying lemma and a Dodd-Jensen lemma
for such maps.\footnote{ See \cite[2.11, 2.12]{finetame}.}
 We shall use those results below.

Notice that if $\pi$ is nearly elementary as a map from $M$ to $N$, then it
is nearly (+1)-elementary as a map from $M^-$ to $N^-$. $M$ and $N$ may have either type,
but $M^-$ and $N^-$ have type 1.

Clause (i) of \ref{nearplusone.elementary} implies that
$\pi(\eta_k^M) = \eta_k^N$.\footnote{See \cite[4.3.4]{normalization_comparison}.}

\begin{lemma}\label{ult0elementarity}
Let $N$ be a solid premouse
and $\degr(N)=k$, and let $P=\Ult_{k}(N,E)$,
where $E$ is an extender over $N$ that is close to $N$
such that $\rho_{k+1}(N) \le \crit(E) < \rho_k(N)$. Suppose
$P$ is wellfounded; then $P$ is solid, and the canonical
embedding $i_E^N$ is nearly +1-elementary. Moreover,
$\rho_{k+1}(N) = \rho_{k+1}(P)$.
\end{lemma}
\begin{proof} This is proved in \cite[Lemma 2.4.12]{normalization_comparison},
except for the assertions that $P$ is projectum solid and stable, and that
$i_E^N$ is exact.

Note that $\crit(E) \notin \lbrace \rho_{k+1}(N), \eta_k^N \rbrace$
because $\crit(E)$ is measurable in $N$ by closeness, while
$N$ is projectum solid and stable. Since $\rho_{k+1}(N) < \crit(E)$,
$P$ is projectum solid, and since $\eta_k^N \neq \crit(E)$,
$i_E^N$ is continuous at $\rho_k(N)$, and therefore exact. Since
$i_E^N$ is continuous at $\rho_k(N)$, $i_E^N(\eta_k^N) = \eta_k^P$.
But $N$ is stable, so $P$ is stable.
\end{proof}

\begin{remark}{\rm With more work, one can weaken the hypotheses on $E$
by dropping closeness, and requiring of $\crit(E)$ only that
$\crit(E)<\rho_k(N)$ and $\crit(E) \notin \lbrace \rho_{k+1}(N),
\eta_k^N \rbrace$. In the case that $\crit(E) < \rho_{k+1}(N)$,
one must weaken the conclusion to $\rho_{k+1}(P) = \sup i``\rho_{k+1}(N)$,
for $i = i_E^N$, but otherwise the conclusions remain 
the same.\footnote{ That $\rho_{k+1}(P) = \sup i``\rho_{k+1}(N)$ under these
hypotheses is a result of Schlutzenberg; see \cite[9.6.1]{normalization_comparison}.}}
\end{remark}

Note also

\begin{lemma}\label{anticore.elem} Let $M$ be solid, and let
$\pi \colon \mfc(M)^- \to M$ be the anticore map; then $\pi$
is nearly +1-elementary.
\end{lemma}
\begin{proof} Let $k=\degr(M)$; then $\pi$ is cofinal and
$\Sigma_0$ elementary, and hence $\Sigma_1$ elementary, as a map from $\mfc(M)^k$ to $M^k$.
Because we are using pfs fine structure
$\rho_k(M) \in \ran(\pi)$, so $\pi$ is exact. The rest is easy to verify.
\end{proof}

From Lemma \ref{ult0elementarity} we get
\begin{lemma}\label{+1tree.branchelem} Let $\itT$ be a 1-maximal plus tree on
the solid premouse $M$,
$\alpha <_T \beta$, and suppose that if
$\gamma \in D^\itT \cap [0,\beta]_T$, then $\gamma$ is a small drop; then
$i_{\alpha,\beta}^\itT$ is nearly (+1)-elementary.
\end{lemma}

Nearly +1-elementary embeddings suffice for copying
1-maximal trees. 

\begin{definition}\label{shiftnotation} Let $M$ and $N$ be solid premice with degree $k$, and $E$ and $F$
be extenders over $M$ and $N$ respectively. We say
\[
\la \pi, \varphi\ra \colon (M,E) \stackrel{*,1}{\longrightarrow} (N,F)
\]
iff $\pi$ is nearly +1-elementary, and $\la \pi,\varphi \ra
\colon (M^k,E) \stackrel{*}{\longrightarrow}(N^k,F)$ in the sense of
\cite[2.5.17]{normalization_comparison}.
\end{definition}

\begin{lemma}\label{shiftlemma}[Shift Lemma] Suppose that
$\la \pi, \varphi\ra \colon (M,E) \stackrel{*,1}{\longrightarrow} (N,F)$, where
$M$ and $N$ are solid premice of degree $k$. Let
$m = k+1$ if $\crit(E) < \rho_{k+1}(M)$, and $m=k$ otherwise, and let
$n=k+1$ if $\crit(F)<\rho_{k+1}(N)$ and $n=k$ otherwise. Let
\[
\sigma \colon Ult_m(M,E) \to \Ult_n(N,F)
\]
be the completion of the map $\sigma_0([a,f]^{M^k}_E) =
[\varphi(a),\pi(f)]^{N^k}_F$. Then
\begin{itemize}
\item[(i)] $\sigma$ is nearly +1-elementary,
\item[(ii)] $\sigma \within \lh(E) = \varphi \within \lh(E)$, and
\item[(iii)] $\sigma \circ i_E^M = i_F^N \circ \pi$.
\end{itemize}
\end{lemma}
\begin{proof}(Sketch.) Since $\sup \pi``\rho_{k+1}(M) \le
\rho_{k+1}(N)$, $m=k+1$ implies $n=k+1$.\footnote{ It is possible that
$m=k$ and $n=k+1$. We could have copied with $n=k$ in this case,
but we are not going to do that, because we want to stay in the realm of
1-maximal trees.} Thus the definition
of $\sigma_0$ makes sense, in that $\pi$ can be applied to $f$
(even if $f$ is only $\Sigma_1^{M^k}$), and $\pi(f)$ is a function
that is used in the $N$-ultrapower. Note that both
$i_E^M$ and $i_F^N$ are exact, so $\sigma$ is exact by commutativity.
The remaining calculations are the usual ones.
\end{proof}

\begin{definition} Let $\pi \colon M \to N$ be nearly +1-elementary,
and let $\itT$ be a 1-bounded plus tree on $M$; then 
$\pi\itT$ is the 1-bounded copied tree on $N$.
\end{definition}
Of course we stop the construction of $\pi\itT$ if we reach an illfounded
model. Letting $\pi_\alpha \colon M_\alpha \to N_\alpha$ be the copy map, we prove
by induction that if $\beta = T\tpred(\alpha+1)$, then
\[
\la \pi_\beta \within M_{\alpha+1}^*,\pi_\alpha \ra \colon (M_{\alpha+1}^*,E_\alpha^\itT)
\stackrel{*}{\longrightarrow}(N_{\alpha+1}^*,E_\alpha^{\pi\itT}),
\]
and if $D^\itT \cap [0,\beta]_T = \emptyset$, then $\pi_\beta$ is
nearly +1-elementary. This can be done.\footnote{ See \cite[Section 4.5]{normalization_comparison}.}

\begin{definition}\label{1boundedstack} A {\em 1-bounded stack}  on $M$ is a stack $s$ of trees such that
\begin{itemize}
\item[(1)] each tree $\itT_i(s)$ in $s$ is 1-bounded, and
\item[(2)] for $i>0$, letting $N$ be the base model of $ \itT_i(s)$,
$\itT_i(s)$ is 1-maximal iff $M$-to-$N$ has no drops of any kind.
$s$ is {\em 1-maximal} iff $\itT_0(s)$ is 1-maximal; otherwise, $s$ is 0-bounded.
\end{itemize}
\end{definition}

Let $G^1(M,\lambda,\theta)$ be the variant
 of $G(M,\lambda,\theta)$ in which
the output is a 1-bounded  stack on $M$.

\begin{definition}\label{+1strategies} A {\em $(\lambda,\theta)^+$-iteration strategy} 
for $M$ is a winning
strategy for player \rm{II} in $G^1(M,\lambda,\theta)$. Let $\Sigma$ be such a strategy; then
\begin{itemize}
\item[(a)](Tail strategy.) For $s$ a 1-bounded stack by $\Sigma$ with last model $N$,
$\Sigma_{s,N}(t) = \Sigma(s^\frown t)$. If $M$-to-$N$ does not drop in $s$,
then $\Sigma_{s,N}$ is a $(\lambda-\lh(s), \theta)^+$-iteration strategy for $N$.
Otherwise, it is a $(\lambda-\lh(s),\theta)$-strategy.
\item[(b)](Pullback strategy.) If $\pi \colon N \to M$ is nearly +1-elementary,
then $\Sigma^\pi$ is the $(\lambda,\theta)^+$-strategy for $N$ given by:
$\Sigma^\pi(s) = \Sigma(\pi s)$.
\end{itemize}
\end{definition}

If $\Sigma$ is a $(\lambda,\theta)^+$-strategy for $M$, then we obtain
an ordinary  $(\lambda,\theta)$-strategy for $M$ by restriction $\Sigma$
to act on 0-bounded stacks $s$. We call this restricted strategy $\Sigma^-$.

\bigskip
\noindent
{\em Notation:} If $\itT$ is a tree on $(P,\Sigma)$ and $\alpha < \lh(\itT)$,
then we may write $\Sigma_\alpha^\itT$ for the tail strategy
$\Sigma_{\itT \within \alpha+1, \itM_\alpha^\itT}$.

\subsection{Background-induced +1-strategies}

We get $(\lambda,\theta)^+$-iteration strategies from the same
background constructions that gave us $(\lambda,\theta)$-strategies.
The constructions themselves do not change at all, because the strategy
predicate of an lpm still has only information about the action of
the strategy on 0-bounded, $\lambda$-separated trees.

Recall that a {\em conversion stage} is a tuple
$\la M,\psi,Q,\mathbb{C},R \ra$ such that $R$ is an appropriate
background universe, $\mathbb{C}$ is a maximal hod pair
construction in the sense of $R$, $Q$ is a level of $\mathbb{C}$,
and $\psi \colon M \to Q$ is nearly elementary.\footnote{See Section 4.8 and Section 9.4
of \cite{normalization_comparison}.}

\begin{definition}\label{plus1conversionstage} A {\em +1-conversion stage} is
a conversion stage $\la M,\psi, Q, \mathbb{D},R\ra$ such that $M$ and $Q$ are
solid, and $\psi$ is nearly +1-elementary. 
\end{definition}

Given a conversion stage $c = \la M, \psi, Q, \mathbb{C}, R\ra$
and a 0-bounded plus tree $\itT$ on $M$,  \cite[Section 4.8]{normalization_comparison} defines the
{\em conversion system}
\[
\lift(\itT,c) = \la \itT^*, \la c_\alpha \mid \alpha < \lh(\itT) \ra \ra.
\]
Here $c_0 = c$, and $c_\alpha = \la M_\alpha, \psi_\alpha, Q_\alpha, \mathbb{C}_\alpha, R_\alpha \ra$
is a conversion stage such that $M_\alpha = \itM_\alpha^\itT$ and
$R_\alpha = \itM_\alpha^{\itT^*}$. The same construction applies when
$\itT$ is 1-bounded and $c$ is a +1-conversion stage; we just
need to note that when $\alpha$ is a +1-node of $\itT$, then the
conversion stage $c_\alpha$ produced by our construction is
a +1-conversion stage. 

This comes down to adding a few lines to
the Shift Lemma for Conversions.\footnote{\cite[3.3.2, 4.8.2]{normalization_comparison}.} 
Suppose $\alpha+1$ is a +1-node of $\itT$, $T\tpred(\alpha+1) = \beta$,
and $E=E_\alpha^\itT$. Let 
\begin{align*}
\varphi &= \sigma_{Q_\alpha}^{\mathbb{C}_\alpha}[Q_\alpha|\lh(\psi_\alpha(E^-))] \circ \psi_\alpha,\\
E^* &= B^{\mathbb{C}_\alpha}(\varphi(E^-)),
\intertext{ and }
E_\alpha^{\itT^*} &=E^*
\end{align*}
So $E^*$ is the background extender for the resurrection of $\psi_\alpha(E^-)$.
There is enough agreement between $\varphi$ and $\psi_\beta$ that $E^*$
should be applied to $R_\beta = \itM_\beta^{\itT^*}$ in a normal
continuation of $\itT^*$.
Letting $k= \degr(M)$ and $i^* = i_{E^*}^{R_\beta}$, our next conversion
stage is then
\[
c_{\alpha+1} = \la \Ult_{k+1}(M_\beta,E), \pi, i^*(Q_\beta), i^*(\mathbb{C}_\beta),
\Ult(R_\beta,E^*)\ra.
\]
Since $\alpha+1$ is a +1 node, $M_{\alpha+1}^* = M_\beta$.
$\pi$ is (the completion of) the map $\pi \colon \Ult_1(M_\beta^k,E)
\to i^*(Q_\beta^k)$ given by
\[
\pi([a,f]^{M_\beta^k}_E) = [\varphi(a),\psi_\beta(f)]^{R_\beta}_{E^*},
\]
where $\psi_\beta$ moves $f$ by moving its $\Sigma_1^{M_\beta^k}$ definition.
Since $\psi_\beta$ is $\textrm{Th}_1$-preserving, $\pi$ is well defined. It is
easy to check that $\pi$ is $\Sigma_1$-elementary and $\text{Th}_1$-preserving
as a map from  $ \Ult_1(M_\beta^k,E)$ to $i^*(Q_\beta^k)$, using that
\[
\pi \circ i_E^{M_\beta} = i^* \circ \psi_\beta.
\]
Commutativity
and the fact that $i_E^{M_\beta}$, $i^* \within Q_\beta$, and $\psi_\beta$
are all exact also implies that $\pi$ is exact. Finally,
\begin{align*}
\sup \pi``\rho_{k+1}(M_{\alpha+1}) &= \sup \pi \circ i_E^{M_\beta} ``\rho_{k+1}(M_\beta)\\
                  &= \sup i^* \circ \psi_\beta `` \rho_{k+1}(M_\beta)\\
                  &\le i^*(\rho_{k+1}(Q_\beta)) = \rho_{k+1}(i^*(Q_\beta)).
\end{align*}
Collectively, these calculations show that $\pi$ is nearly +1-elementary.
It agrees with $\varphi$ on $\lh(E)$, which we need to keep the conversion
going.

We lift 1-bounded stacks in the natural way. For example,
$\lift(\la \itT,\itU\ra, c) = \la \lift(\itT,c), \lift(\itU,d)\ra$,
where $d$ is the last conversion stage in $\lift(\itT,c)$.
Letting $\itT^*=\lift(\itT,c)_0$ and $\itU^* = \lift(\itU,d)_0$,
$\la \itT^*,\itU^*\ra$ is the corresponding stack of nice trees on
the background universe in $c$. Abusing notation, we write
\[
\lift(\vec{\itT},c)_0 = \la \itT_i^* \mid i < \lh(\vec{\itT})\ra
\]
for the stack of lifted trees on the background universe in $c$.

\begin{definition}\label{induced.strategy} Let $((N^*,\in, w, \itF,\Psi),\Psi^*)$ be a 
coarse strategy pair,\footnote{ See  
\cite[9.4.14]{normalization_comparison}. Roughly, the requirements
are that
\begin{itemize} 
\item[(a)] $(N^*,\in,w,\itF,\Psi)$ is a {\em coarse strategy premouse},
i.e.  $N^*$ is countable and transitive, $N^* \models \zfc +
``(w,\itF)$ is a coherent pair" + ``$\Psi$ is a $(\delta^*,\delta^*,\itF)$-iteration
strategy for $V$ that quasi-normalizes well, has strong hull condensation, and
is pushforward consistent", and
\item[(b)] $\Psi^*$ is a complete $(\omega_1,\omega_1)$-iteration strategy
for $(N^*,\in,w,\itF,\Psi)$ that normalizes well and has strong hull
condensation, and
\item[(c)] if $i \colon N^* \to S$ is the iteration map associated to a stack
$t$ according to $\Psi^*$, then $i(\Psi) \subseteq \Psi^*_{t,S}$.
\end{itemize}
$(w,\itF)$ is a coherent pair iff $w$ is a wellorder
of $V_\delta$, where $\delta = \delta(w)$, and $\itF \subseteq V_\delta$ is
 a set of nice extenders such that for $E \in \itF$ and
$\nu = \lh(E)$, $i_E(w) \cap V_{\nu+1} = w \cap V_{\nu+1}$ and
$i_E(\itF) \cap V_{\nu+1} = \itF \cap V_{\nu+1}$. See \cite[2.9.6]{normalization_comparison}.} 
 and let $\mathbb{C}$ be
the maximal least branch $(w,\itF,\Psi)$-construction of  $(N^*,\in,w,\itF,\Psi)$\footnote{
Cf. \cite[9.4.10]{normalization_comparison}.},
 with models $M_{\nu,k}=M^{\mathbb{C}}_{\nu,k}$ and induced 
strategies $\Omega_{\nu,k}=\Omega^\mathbb{C}_{\nu,k}$.
\begin{itemize}
\item[(1)]  Let
$c$ be a +1-conversion stage of the form $\la M,\psi, Q,\mathbb{C},N^*\ra$;
then $\Omega^+(c,\Psi^*)$ is the $(\omega_1,\omega_1)^+$-iteration strategy for 
$M$ given by
\[
\vec{\itT} \text{ is by $\Omega^+(c,\Psi^*)$} \text{ iff } \lift(\vec{\itT}, c)_0 \text{ is by $\Psi^*$.}
\]
\item[(2)]$(\Omega^+_{\nu,k})^{\mathbb{C}} = \Omega^+(\la M_{\nu,k}$, \text{ id }, $M_{\nu,k}, \mathbb{C}, N^* \ra, \Psi^*)$.
\end{itemize}
\end{definition}

Clearly, $\Omega_{\nu,k} = (\Omega_{\nu,k}^+)^-$. As one would expect,
$\Omega_{\nu,k+1}$ is the pullback
of $\Omega_{\nu,k}^+$ under the anticore map.

\begin{lemma}\label{omeganulplus.pullback} Let $\mathbb{C}$ be the maximal least 
branch $(w,\itF,\Psi)$-construction of some coarse strategy pair, with associated
models $M_{\nu,k}$ and strategies $\Omega_{\nu,k}$ and $\Omega_{\nu,k}^+$.
Let $\pi \colon M_{\nu,k+1}^- \to M_{\nu,k}$ be the anticore map;
then $\pi$ is nearly +1-elementary, and
\[
\Omega_{\nu,k+1} = (\Omega_{\nu,k}^+)^\pi.
\]
\end{lemma}
\begin{proof} $\pi$ is nearly +1-elementary by \ref{anticore.elem}.

Note first that the identity makes sense. $\Omega_{\nu,k+1}$ is a strategy
acting on 0-bounded stacks on $M_{\nu,k+1}$, or equivalently, on
1-maximal stacks on $M_{\nu,k+1}^-$. If $s$ is a 1-maximal stack on
$M_{\nu,k+1}^-$, then $\pi s$ is a 1-maximal stack on $M_{\nu,k}$,
and so $\Omega_{\nu,k}^+$ acts on it. Thus the identity makes sense.

The identity is true because, letting
\begin{align*}
c &= \la M_{\nu,k+1}, \text{ id }, M_{\nu,k+1},\mathbb{C}, N^* \ra \\
\intertext{ and }
d &= \la M_{\nu,k}, \text{ id }, M_{\nu,k},\mathbb{C}, N^* \ra, \\
\intertext{ and letting $s$ be a 0-bounded stack on $M_{\nu,k+1}$,}
\lift(s, c)_0 &= \lift(\pi s, d)_0.
\end{align*}
Thus letting $\Psi^*$ be the strategy for $N^*$,
 \begin{align*}
\text{$s$ is by $\Omega_{\nu,k}$} & \text{ iff $\lift(s, c)_0$ is by $\Psi^*$} \\
       & \text{ iff
$\lift(\pi s, d)_0$ is by $\Psi^*$} \\
   &\text{ iff $\pi s$ is by
$\Omega_{\nu,k}^+$.}
\end{align*}

The proof that $\lift(s, c)_0 = \lift(\pi s, d)_0$ is a routine
induction, essentially identical to the proof of \cite[Lemma 5.4.2]{normalization_comparison}.
We omit further detail.
\end{proof}

\subsection{Regularity properties and comparison}

     The definitions and results of \cite{normalization_comparison} go over to
+1-strategies with almost no change. 

      For example, if $\itT$ and $\itU$ are 1-maximal trees on $M$,
then $\Phi$ is a {\em tree embedding} from $\itT$ to $\itU$ iff letting
$\Phi = \la u, v, \la s_\alpha \mid \alpha < \lh(\itT)\ra,
\la t_\alpha \mid \alpha+1 < \lh(\itT) \ra \ra$, $\Phi$ has all
the properties enumerated in \cite[Definition 6.4.1]{normalization_comparison}, and
in addition
\[
\alpha \text{ is a +1-node of $\itT$ } \Rightarrow s_\alpha \text{ is +1-elementary}.
\]
A +1-strategy $\Sigma$ has {\em strong hull condensation} iff whenever $\Phi \colon \itT \to \itU$
is a tree embedding and $\itU$ is by $\Sigma$, then $\itT$ is by $\Sigma$; moreover, whenever
$\alpha < \lh(\itT)$, $v(\alpha) \le_U \beta$,  $\pi = \hat{i}_{v(\alpha),\beta}^\itU \circ s_\alpha$, and
$Q \unlhd \dom(\pi)$, then we have $\Sigma_{\itT,Q} = (\Sigma_{\itU \within \beta+1,\pi(Q)})^\pi$.

If $s$ is a 1-maximal stack, then the quasi-normalization $V(s)$
and embedding normalization $W(s)$ are defined just as before, but so that
the trees they produce are 1-maximal. If the trees in $s$ are $\lambda$-separated,
then $V(s) = W(s)$. This is the only case we care about. A 1-strategy $\Sigma$
{\em quasi-normalizes well} iff whenever $s$ is a 1-maximal stack by $\Sigma$,
then $V(s)$ is by $\Sigma$; moreover if $Q$ and $R$ are the last models of $s$ and $V(s)$,
and $\sigma \colon Q \to R$ is the quasi-normalization map, then
$\Sigma_{s,Q} = (\Sigma_{V(s),R})^\sigma$.\footnote{ If $M$-to-$Q$ does not drop in $s$,
then $M$-to-$R$ does not drop in $V(s)$ and $\sigma$ is nearly +1-elementary,
so $(\Sigma_{V(s),R})^\sigma$ is indeed a strategy for 1-maximal stacks on $Q$.}

The definition of internal lift consistency does not change at all.

If $P$ is an lpm and $\Sigma$ is a $(\lambda,\theta)^+$-strategy for $P$,
then we say $(P,\Sigma)$ is {\em pushforward consistent} iff whenever
$s$ is a stack by $\Sigma$ with last model $Q$, $\dot{\Sigma}^Q \subseteq \Sigma_{s,Q}$. 

\begin{definition}\label{+1hodpair}($\adp$) A {+1-hod pair} with scope $H_{\omega_1}$ 
is a pair $(P,\Sigma)$
such that $P$ is a countable  lpm, $\Sigma$ is an $(\omega,\omega_1)^+$-iteration strategy
for $P$, and $(P,\Sigma)$ has strong hull condensation, normalizes well, and is internally
lift consistent and pushforward consistent.
\end{definition}

\begin{remarks}{\rm 
\begin{itemize}
\item[(a)] If we are not assuming $\adp$ we may want to consider uncountable $P$,
and $\Sigma$ that are defined on uncountable trees.
\item[(b)]  If $(P,\Sigma)$ is
a +1-hod pair, then $(P,\Sigma^-)$ is an lbr hod pair in the sense of
 \cite{normalization_comparison}.
\item[(c)] If $(P,\Sigma)$ is a +1-hod pair and $P$ is sound, then
$(P^+,\Sigma)$ is an lbr hod pair in the sense of \cite{normalization_comparison}.
\end{itemize}}
\end{remarks}

Since nearly +1-elementary maps suffice to copy 1-maximal trees, we
get a Dodd-Jensen lemma:

\begin{lemma}\label{plusonedoddjensen}[Dodd-Jensen] Let $(M,\Sigma)$ be a
+1-mouse pair, $s$ a 1-maximal stack on $(M,\Sigma)$ with last pair
$(N,\Lambda)$, and $\pi \colon (M,\Sigma) \to (Q,\Omega) \unlhd (N,\Lambda)$
be nearly +1-elementary; then $(Q,\Omega) = (N,\Lambda)$, the branch
$M$-to-$N$ of $s$ does not have a large  drop, and $i_s(\eta) \le \pi(\eta)$ for all
$\eta \in \OR^M$.
\end{lemma}

\begin{lemma}\label{omegaplus.regular} Let $((N^*,\in, w, \itF,\Psi),\Psi^*)$ be a 
coarse strategy pair,
and let $\mathbb{C}$ be
the maximal least branch $(w,\itF,\Psi)$-construction of  $(N^*,\in,w,\itF,\Psi)$
 with models $M_{\nu,l}$ and induced 
+1-strategies $\Omega_{\nu,l}^+$; then for all $\nu,l$,
$\Omega_{\nu,l}^+$ has strong hull condensation, quasi-normalizes well,
and is internally lift consistent and pushforward consistent.
\end{lemma}
\begin{proof}(Sketch.) The proof in \cite{normalization_comparison}
that $\Omega_{\nu,l}$ has these properties works also for
$\Omega_{\nu,l}^+$.
\end{proof}

We turn now the comparison theorem for +1-hod pairs.

\begin{definition}\label{iteratespastdef} Let $(P,\Sigma)$ be a +1-hod pair; then
\begin{itemize}
\item[(1)] $(P,\Sigma)$ {\em iterates strictly past $(Q,\Lambda)$} iff 
there is a 1-maximal, $\lambda$-separated tree $\itT$ on $(P,\Sigma)$ with last pair
$(R,\Psi)$ such that either $(Q,\Lambda) \lhd (R,\Psi)$, or
$P$-to-$R$ has a large drop and $(Q,\Lambda) = (R,\Psi)$.
\item[(2)] $(P,\Sigma)$ {\em iterates to $(Q,\Lambda)$ }iff
there is a 1-maximal, $\lambda$-separated tree $\itT$ on $(P,\Sigma)$
with last pair $(Q,\Lambda)$ and such that $P$-to-$Q$ does not 
have a large drop.
\end{itemize}
\end{definition}
In case (1), $(Q,\Lambda)$ must be an ordinary lbr hod pair, and in case
(2) it must be a 1-hod pair. In both cases,
the 1-maximal tree $\itT$ is uniquely determined
by $(P,\Sigma)$ and $(Q,\Lambda)$.

\begin{definition}\label{capturesSigma}
$(P,\Sigma)$ be a  +1-hod pair.
We say that a coarse strategy pair
$((N^*,\in, w, \itF,\Psi),\Psi^*)$  {\em captures
$(P,\Sigma)$} iff there is an inductive-like pointclass $\Gamma$ with the scale property
such that $\text{Code}(\Sigma) \in \Delta_\Gamma$, and
for $\delta^* = \delta(w)$,
\begin{itemize}
\item[(i)] $N^* \models ``\delta^*$ is Woodin", and
\item[(ii)] $P \in \text{HC}^{N^*}$, and there is a Coll$(\omega,\delta^*)$-term $\tau$ and
 a universal $\Gamma$-set $U$
such that
 if $i:N^*\rightarrow S$ is via $\Sigma^*$ and $g\subseteq 
  \textrm{Col}(\omega,i(\delta^*))$ is $S$-generic, 
 then $i(\tau)_g = U\cap S[g]$.
\end{itemize}
\end{definition}

\begin{theorem}\label{starpsigmathm}($\adp$) Let $(P_0,\Sigma_0)$ be a +1-hod pair, 
let $((N^*,\in, w, \itF,\Psi),\Psi^*)$ be a coarse strategy pair
that captures $(P_0,\Sigma_0)$, and let $\mathbb{C}$ be the maximal 
least branch $(w,\itF,\Psi)$-construction of $(N^*,\in,w,\itF,\Psi)$;
then
there is a $\la \theta,n\ra <_{\lex} \la \delta^*,0\ra$
such that
\begin{itemize}
\item[(1)] for all $\la \nu,k\ra <_{\lex} \la \theta,n\ra$, $(P_0,\Sigma_0)$ iterates
strictly past $(M_{\nu,k},\Omega_{\nu,k})$, and
\item[(2)] $(P_0,\Sigma_0)$ iterates to $(M_{\theta,n}, \Omega_{\theta,n}^+)$.
\end{itemize}
\end{theorem}
\begin{proof}(Sketch.) The proof is very close to that of
\cite[Theorem 9.5.2]{normalization_comparison}, so we shall just give an
outline.

Let $\la \theta,n \ra$ be lex least
such that $(P_0,\Sigma_0)$ does not iterate strictly past
$(M_{\theta,n},\Omega_{\theta,n})$.
For $\la \nu, k\ra <_{\lex} \la \theta, n\ra$, let
\[
\itW^*_{\nu,k} = \text{$\lambda$-separated tree whereby $(P_0,\Sigma_0)$ iterates strictly past $(M_{\nu,k},\Omega_{\nu,k})$.}
\]
Let
\[
(M,\Omega, \Omega^+)=(M_{\theta,n},\Omega_{\theta, n}, \Omega_{\theta,n}^+),
\]
and let $\itT$ be the tree on $(P_0,\Sigma_0)$ formed by iterating away least extender disagreements with $M$,
as follows. We assume by induction

\bigskip
\noindent
{\bf Induction hypothesis $(\dagger)_\alpha$.}

\medskip
\noindent
If the current last pair of $\itT$
is $(Q,\Lambda) = (\itM_\alpha^\itT,\Sigma_\alpha^\itT)$, then
\begin{itemize}
\item[(1)]($M$ is passive in extender disagreements) if
$E$ is on the $M$-sequence and $M||\lh(E) = Q||\lh(E)$, then
$M|\lh(E) = Q|\lh(E)$.
\item[(2)](No strategy disagreements) if 
$S \unlhd Q$ and $S \unlhd M$, then 
\begin{itemize}  
\item[(a)] if $S \lhd M$, then $S \lhd Q$ 
and $\Lambda_S = \Omega_S$, and
\item[(b)] if $S= M$, then $S=Q$, $\alpha$ is a +1-node of $\itT$, 
and $\Lambda = \Omega^+$.
\end{itemize}
\end{itemize}

\bigskip
\noindent
Note that (2)(a) implies that $Q$ is not a proper initial segment of $M$.

If $(\dagger)_\alpha$ is false, then we stop the construction of $\itT$,
and say that {\em $\itT$ fails at $\alpha$}.
Let us assume that $\itT$ never fails, and finish the proof. 

\begin{claim}\label{claim0}  If $(\dagger)_\alpha$ holds, then either
\begin{itemize}
\item[(i)] $\alpha$ is a +1-node of $\itT$, $Q=M$, 
and $\Lambda = \Omega^+$, or
\item[(ii)]
there is an extender $E$ on the $Q$-sequence
 such that $(Q, \Lambda)||\lh(E) = (M, \Omega)||\lh(E)$
but $M|\lh(E)$ is passive.
\end{itemize} 
\end{claim}
\begin{proof} Suppose that (ii) fails. 
Then by $(\dagger)_\alpha$(1), $Q \lhd M$ or $M \unlhd Q$. By (2)(a) then (with $S=Q$),  $M \unlhd Q$.
 But then (2)(b) with $S=M$
implies that (i) holds.
\end{proof}

If \ref{claim0}(i) holds, then $(P_0,\Sigma_0)$ has iterated to
$(M,\Omega^+)$, as desired. In this case we again stop the construction,
and say that {\em $\itT$ succeeds at $\alpha$}. Assume now that $\itT$ neither fails
nor succeeds at $\alpha$, and
let $E$ be the unique extender on the $Q$-sequence such that
$Q||\lh(E) = M||\lh(E)$ but $Q|\lh(E) \neq M|\lh(E)$. 
We set
\[
   E_\alpha^\itT = E^+
\]
and continue constructing $\itT$.  At limit steps we use $\Sigma_0$ to choose a branch, as
a tree on $(P_0,\Sigma_0)$ must do.

The lengths of the $E_\alpha^\itT$ are strictly increasing, so if $\itT$
never fails, then
eventually it must succeed, that is, we must reach $\alpha$ such that $(\dagger)_\alpha$ holds and
part (i)  of Claim \ref{claim0} holds. This means that $(P_0,\Sigma_0)$
has iterated to $(M,\Omega^+)$, as desired.


For $\la \nu, k\ra <_{\lex} \la \theta, n\ra$, let
\[
\itW^*_{\nu,k} = \text{1-maximal $\lambda$-separated tree whereby 
$(P_0,\Sigma_0)$ iterates strictly past $(M_{\nu,k},\Omega_{\nu,k})$.}
\] 
Let us check that the lemma of \cite[Section 8.3]{normalization_comparison} on realizing
resurrection embeddings as branch embeddings holds for the new $W^*_{\nu,k}$. 
Fix $\la \nu,k \ra <_{\lex} \la \theta,n\ra$
for a while, and suppose that $M_{\nu,k}$ is not sound. Let $\xi_1 = \lh(\itW^*_{\nu,k})$.

By definition, $\itM_{\xi_1}^{\itW^*_{\nu,k}} \unrhd M_{\nu,k}$,
so since $M_{\nu,k}$ is not sound, 
$\itM_{\xi_1}^{\itW^*_{\nu,k}} = M_{\nu,k}$. $(P_0,\Sigma_0)$ iterates strictly
past $(M_{\nu,k},\Omega_{\nu,k})$, so $[0,\xi_1]_{W^*_{\nu,k}}$ has a large drop.
Let
\begin{align*}
\eta_0 &= \text{ largest $\gamma$ in $[0,\xi_1]_{\itW^*_{\nu,k}} \cap D^{\itW^*_{\nu,k}}$},\\
\eta &= W^*_{\nu,k} \tpred(\eta_0),\\
i^* &= i_{\eta_0,\xi_1}^{\itW^*_{\nu,k}} \circ i_{\eta_0}^*,\\
\intertext{ and }
R &= \dom(i^*) = \itM_{\eta_0}^{*,\itW^*_{\nu,k}}.\\
\end{align*}
$R$ is
the proper initial segment of $\itM_\eta^{\itW^*_{\nu,k}}$ that lies on the branch
from $\eta$ to $\xi_1$. $R$ has degree $k$, and it is sound because the drop at $\eta_0$
is a large one.  Thus
\[
R^+ = \mfc_{k+1}(R) = \mfc_{k+1}(M_{\nu,k}) = M_{\nu,k+1},
\]
and $i^*$ is the anticore map.
Let $\rho = \rho(M_{\nu,k})$.

\begin{claim}\label{etaversusrho} $\rho < \lh(E_\eta^{\itW^*_{\nu,k}})$,
and for all $\tau < \eta$, $\lh(E_\tau^{\itW^*_{\nu,k}})\le \rho$.
\end{claim}
\begin{proof} We have $\rho \le \crit(i^*) < \lh(E_\eta^{\itW^*_{\nu,k}})$.
Suppose toward contradiction that $\tau < \eta$ and
$\rho < \lh(E_\tau^{\itW^*_{\nu,k}})$. We can then find such
$\tau$ with $\tau+1 <_{W^*_{\nu,k}} \eta$. But this implies
that whenever $S \lhd \itM_\eta^{\itW^*_{\nu,k}}$ with $o(S)\geq \lh(E_\tau^{\itW^*_{\nu,k}})$, then
$\lh(E_\tau^{\itW^*_{\nu,k}}) \le \rho(S)$. Since
$\rho(R) = \rho$, we have a contradiction.
\end{proof}

Sublemma 8.3.1.1 of \cite{normalization_comparison} goes over
verbatim when $\la \nu,k+1 \ra <_{\lex} \la \theta,n \ra$: 
 
\begin{lemma}\label{new8.3.1.1A}Let $\la \nu,k+1 \ra <_{\lex} \la \theta, n \ra$, 
and suppose that $M_{\nu,k}$
is not sound. Let $\pi \colon M_{\nu,k+1}^- \to M_{\nu,k}$ be the anticore map.
Let $\xi = \lh(\itW^*_{\nu,k})$, let $\eta, \eta_0, R$ be as above, 
and let $i^* \colon R \to \itM_{\xi}^{\itW^*_{\nu,k}} = M_{\nu,k}$ 
be the branch embedding of $\itW^*_{\nu,k}$ as above;
then
\begin{itemize}
\item[(a)] $R^+ = M_{\nu,k+1}$ and $i^* = \pi$.
\item[(b)] $\eta$ is the least $\gamma$ such that
$\lh(E_\gamma^{\itW^*_{\nu,k}}) > \rho(M_{\nu,k})$.
\item[(c)] $\itW^*_{\nu,k+1} = \itW^*_{\nu,k} \within \eta+1$,
\end{itemize}
\end{lemma}

\begin{proof} We have already shown (a) and (b), and (c) follows
at once.
\end{proof}

 As a result,
Lemma 8.3.1 goes over verbatim for resurrections from some $\la \nu,k\ra <_{\lex}
\la \theta,n\ra$:

\begin{corollary}\label{res.as.branch.1} Let $\la\nu,k\ra <_{\lex}
\la \theta,n\ra$, $P \unlhd M_{\nu,k}$,
\begin{align*}
\eta &= \text{ least $\xi$ such that $P \unlhd \itM_\xi^{\itW^*_{\nu,k}}$,}\\
\intertext{ and }
M_{\mu,j} &= \Res_{\nu,k}[P];\\
\intertext{ then}
\itW^*_{\nu,k} \within (\eta+1) &= \itW^*_{\mu,j} \within (\eta+1),\\
\intertext{ $\itW^*_{\mu,j}$ has last model $M_{\mu,j}$, and for
$\xi = \lh(\itW^*_{\mu,j})$, we have $\eta <_{W^*_{\mu,j}} \xi$, and }
\sigma_{\mu,j}[P] &= \hat{i}^{\itW^*_{\mu,j}}_{\eta,\xi}.\\
\end{align*}
\end{corollary}
\begin{proof} The proof of \cite[Lemma 8.3.1]{normalization_comparison} also
goes over verbatim.\footnote{ The fact that $\itW^*_{\nu,k}$ is $\lambda$-separated
plays a role in the proof.}
\end{proof}

We can now finish the proof of the theorem in the case $n>0$.

\begin{claim}\label{n>0case} If $n>0$, then $\itT$ succeeds at some $\alpha$.
\end{claim}
\begin{proof} Let $n=k+1$, and suppose toward contradiction that $\itT$ fails at $\alpha$.
Let 
\begin{align*}
(Q,\Lambda)&=(\itM_\alpha^\itT,\Sigma_{\alpha}^\itT),\\
\intertext{ and let }
(N,\Psi) &= (\itM_\xi^{W^*_{\theta,k}}, \Sigma_\xi^{\itW^*_{\theta,k}})\\
\end{align*} be the last
pair in $W^*_{\theta,k}$.

Suppose first that $M_{\theta,k}$ is sound. Since 
\[
M_{\theta,k+1} = M_{\theta,k}^+
\]
 (i.e. they are the same
as bare premice), $\itT \within \alpha+1 = \itW^*_{\nu,k} \within \alpha+1$. But $W^*_{\nu,k}$ never failed,
so clearly $\itT$ does not fail at $\alpha$ for the ``bad extender disagreement" reason; i.e. $(\dagger)_\alpha$(1)
holds. We verify $(\dagger)_\alpha$(2), i.e. that there is no strategy disagreement between $\Lambda$
and $\Omega_{\theta,k+1}$, and in fact $(Q,\Lambda)= (M_{\theta,k+1},\Omega_{\theta,k+1}^+)$.

  Let
$S \unlhd Q$ and $S \unlhd M_{\theta,k+1}$. 
If $S \lhd M_{\theta,k+1}$ then $S \unlhd M_{\theta,k}$ and
\[
(\Omega_{\theta,k})_S = (\Omega_{\theta,k+1})_S.
\]
Since $P_0$ iterated strictly past $M_{\theta,k}$,
either $S \lhd N$ or $S=N$ and $[0,\xi]_{W^*_{\theta,k}}$ has a large drop. In the latter case
$M_{\theta,k}$ is not sound, contradiction. So $S \lhd N$, and thus
\begin{align*}
     (\Omega_{\theta,k+1})_S &= (\Omega_{\theta,k})_S \\
                            &= \Psi_S \\
                            &= (\Sigma_{\alpha}^{W^*_{\theta,k}})_S\\
                            &=(\Sigma_{\alpha}^\itT)_S,
\end{align*}
which verifies $(\dagger)_\alpha$(2) when $S \lhd M_{\theta, k+1}$. Line 3 holds by strategy coherence.

Suppose next that $S=M_{\theta,k+1}$. 

\begin{subclaim}\label{subclaim.a} If $S= M_{\theta,k+1} \lhd Q$ or $[0,\alpha]_T$ has a large drop,
then $\Omega_{\theta,k+1} = \Lambda_S$.
\end{subclaim}
\begin{proof} The strategy comparison proof of \cite[Theorem 8.4.3]{normalization_comparison}
works here pretty much word-for-word.
\end{proof}

\begin{subclaim}\label{subclaim.b} If $S=M_{\theta,k+1}$, then
$S= Q$ and $[0,\alpha]_T$ has no large
drops.
\end{subclaim}
\begin{proof} Otherwise $(P_0,\Sigma_0)$ iterates strictly past $(M_{\theta,n},\Omega_{\theta,n})$
by Subclaim \ref{subclaim.a}.
\end{proof}

\begin{subclaim}\label{LambdaRclaim} If $S= M_{\theta, k+1}= Q$ and $[0,\alpha]_T$ has no large
drops, then $\Omega_{\theta,k+1}^+ = \Lambda$. 
\end{subclaim}
\begin{proof} The strategy comparison proof of \cite[Theorem 8.4.3]{normalization_comparison}
works pretty much word-for-word.
\end{proof}

Putting the subclaims together, we see that $(P_0,\Sigma_0)$ iterates to
$(M_{\theta,k+1},\Omega_{\theta,k+1}^+)$ in the case that $M_{\theta,k}$ is sound.

So we assume that $M_{\theta,k}$ is not sound.  We adopt our previous notation, with $\theta = \nu$. 
That is,
\begin{align*}
\eta_0 &= \text{ largest $\gamma$ in $[0,\xi_1]_{W^*_{\theta,k}} \cap D^{W^*_{\theta,k}}$},\\
\eta &= W^*_{\theta,k} \tpred(\eta_0),\\
i^* &= i_{\eta_0,\xi_1}^{\itW^*_{\theta,k}} \circ i_{\eta_0}^*,\\
\intertext{ and }
R &= \dom(i^*) = \itM_{\eta_0}^{*,W^*_{\theta,k}}.\\
\end{align*}
$R$ is
the proper initial segment of $\itM_\eta^{\itW^*_{\theta,k}}$ that lies on the branch
from $\eta$ to $\xi_1$. $R$ has degree $k$, and it is sound because the drop at $\eta_0$
is a large one.  Thus
\[
R^+ = \mfc_{k+1}(R) = \mfc_{k+1}(M_{\theta,k}) = M_{\theta,k+1},
\]
and $i^*$ is the anticore map.

Let $\rho = \rho_{k+1}(M)$. By \ref{etaversusrho}, $\eta$ is least such that
$\lh(E_\eta^{\itW^*_{\theta,k}})>\rho$, so
\[
\itT \within \eta+1 = \itW^*_{\theta,k} \within \eta+1.
\]
Note that 
\begin{align*}
(\Omega_{\theta,k+1})_R & = (\Omega_{\theta,k})^{i^*}\\
\intertext{ (by Lemma \ref{omeganulplus.pullback})}
&= ((\Sigma_{\xi_1}^{\itW^*_{\theta,k}})_{M_{\theta,k}})^{i^*}\\
\intertext{ ( because $\itW^*_{\theta,k}$ iterates $(P_0,\Sigma_0)$ strictly past
$(M_{\theta,k},\Omega_{\theta,k})$)}
&= (\Sigma_\eta^{\itW^*_{\theta,k}})_R\\
\intertext{ ( by pullback consistency of $\Sigma_0$)}
&= (\Sigma_\eta^\itT)_R.
\end{align*}

Now we proceed as we did when $M_{\theta,k}$ was sound.

\begin{subclaim}\label{subclaim.a1} If $ R^+ \lhd \itM_\eta^\itT$ or $[0,\eta]_T$ has a large drop,
then $\Omega_{\theta,k+1} = (\Sigma_\eta^\itT)_{R^+}$.
\end{subclaim}
\begin{proof} The strategy comparison proof of \cite[Theorem 8.4.3]{normalization_comparison}
works.
\end{proof}

\begin{subclaim}\label{subclaim.b1} $R^+ = \itM_\eta^\itT$ and
$[0,\eta]_T$ has no large
drops.
\end{subclaim}
\begin{proof} Otherwise $(P_0,\Sigma_0)$ iterates strictly past $(M_{\theta, k+1},\Omega_{\theta,k+1})$
by Subclaim \ref{subclaim.a1}.
\end{proof}

\begin{subclaim}\label{LambdaRclaim.1} If $R^+ = \itM_\eta^\itT$ and
$[0,\eta]_T$ has no large
drops, then $\Omega_{\theta,k+1}^+ = \Sigma_\eta^\itT$. 
\end{subclaim}
\begin{proof} The strategy comparison proof of \cite[Theorem 8.4.3]{normalization_comparison}
works.
\end{proof}

Putting these subclaims together, we see that $(P_0,\Sigma_0)$ iterates to
$(M_{\theta,k+1},\Omega_{\theta,k+1}^+)$ in the case that $M_{\theta,k}$ is unsound.

This completes the proof of Claim \ref{n>0case}.
\end{proof}

Now let us prove $(\dagger)$ for $\itT$ in the case $n=0$.

\begin{claim}\label{n=0case} If $n=0$, then $\itT$ succeeds at some $\alpha$.
\end{claim}
\begin{proof} 
We cannot have a strategy disagreement involving an
$S \lhd M_{\theta,-1}$. For if $\la \nu,k\ra <_{\lex} \la \theta,0\ra$ is sufficiently large,
then $S \lhd M_{\nu,k}$ and $(\Omega_{\theta,-1})_S = (\Omega_{\nu,k})_S$, so this would
mean that $(P_0,\Sigma_0)$ did not iterate strictly past $(M_{\nu,k},\Omega_{\nu,k})$.\footnote{
$M_{\theta,-1}$ is $M_{\theta,0}$ with its last extender, if nonempty, being removed. It is
the ``lim inf" of the $M_{\nu,k}$ for $\la \nu,k\ra <_{\lex} \la \theta,0\ra$. Its strategy
provided by $\mathbb{C}$ is $\Omega_{\theta,-1}$.}
Similarly, we cannot have a bad extender disagreement involving some $E$ on the $M_{\theta,0}$-sequence
other than its last extender.

Thus we can fix $\eta$ least such that $M_{\theta,-1} \unlhd \itM_\eta^\itT$, and we have
$(\Omega_{\theta,-1})_S = (\Sigma_\eta^\itT)_S$ for all $S \lhd M_{\theta,-1}$.

\begin{subclaim}\label{subclaim.a3} $\Omega_{\theta,-1} = (\Sigma_\eta^\itT)_{M_{\theta,-1}}$.
\end{subclaim}
\begin{proof} By the proof of \cite[8.4.3]{normalization_comparison}.
\end{proof}

\begin{subclaim}\label{subclaim.b3} If the last extender $E$ of $M_{\theta,0}$ is nonempty,
then $E$ is on the sequence of $\itM_\eta^\itT$.
\end{subclaim}
\begin{proof}
Let $E^*$ be the background
extender $B^{\mathbb{C}}(E)$. The usual proof\footnote{Cf. \cite[8.1.12]{normalization_comparison}.}
shows that $E$ is an initial segment of the branch extender of
$[\kappa_E,i_{E^*}(\kappa_E))$ in $i_{E^*}(\itT)$, so $E^+$ is used in $i_{E^*}(\itT)$,
so $E$ is on the $\itM_{\eta}^\itT$-sequence, a contradiction.\footnote{ This argument gives
a different proof that if $M_{\theta,0}$ is active, then $\Omega_{\theta,-1} = (\Sigma_\eta^\itT)_{M_{\theta,-1}}$.
We simply go to $i_{E^*}(V)$, use the strategy agreement we have there, and then pull it back to
$V$ by strategy coherence.}
\end{proof}

By \ref{subclaim.a3} and \ref{subclaim.b3}, $M_{\theta,0} \unlhd \itM_\eta^\itT$.

\begin{subclaim}\label{subclaim.c3} If $M_{\theta,0} \lhd \itM_\eta^\itT$ or $[0,\eta]_T$ has a large drop,
then $\Omega_{\theta,0} = (\Sigma_\eta^\itT)_{M_{\theta,0}}$.
\end{subclaim}
\begin{proof} The strategy comparison proof of \cite[Theorem 8.4.3]{normalization_comparison}
works.
\end{proof}

\begin{subclaim}\label{subclaim.d3} $M_{\theta,0} = \itM_\eta^\itT$ and
$[0,\eta]_T$ has no large
drops.
\end{subclaim}
\begin{proof} Otherwise $(P_0,\Sigma_0)$ iterates strictly past $(M_{\theta, 0},\Omega_{\theta,0})$
by Subclaim \ref{subclaim.c3}.
\end{proof}

\begin{subclaim}\label{subclaim.e3} If $M_{\theta,0} = \itM_\eta^\itT$ and
$[0,\eta]_T$ has no large
drops, then $\Omega_{\theta,0}^+ = \Sigma_\eta^\itT$. 
\end{subclaim}
\begin{proof} The strategy comparison proof of \cite[Theorem 8.4.3]{normalization_comparison}
works.
\end{proof}

These subclaims complete the proof of Claim \ref{n=0case}.
\end{proof}

This in turn completes the proof of Theorem \ref{starpsigmathm}.
\end{proof}

\subsection{The sound case}\label{stabletype1case}

      In the case that $\degr(P)>0$, Theorem \ref{starpsigmathm} implies the statement of $(*)(P,\Sigma)$ that was
proved in \cite{normalization_comparison}
in the case that $P$ has stable type 1.

For suppose $(P,\Sigma)$ has stable type 1 and $\degr(P)=k+1$.
Then $(P^-,\Sigma)$ is a +1-hod pair of degree $k$. Letting
$N^*,\mathbb{C},$ etc. be as in the statement of $(*)(P,\Sigma)$ in \cite{normalization_comparison},
Theorem \ref{starpsigmathm} gives us a pair $(M_{\nu,k},\Omega_{\nu,k})$
such that
\begin{itemize}
\item[(i)] $(P^-,\Sigma)$ iterates to $(M_{\nu,k},\Omega_{\nu,k}^+)$, and
\item[(ii)] $(P^-,\Sigma)$ iterates strictly past $(M_{\eta,j},\Omega_{\eta,j})$
whenever $\la \eta,j \ra <_{\lex} \la \nu,k\ra$.
\end{itemize}
The iterations here are by 1-maximal trees on $(P^-,\Sigma)$, which are the same
as ordinary (``0-maximal") trees on $(P,\Sigma)$. The iterations in (ii)
witness that $(P,\Sigma)$ iterates strictly past $(M_{\eta,j},\Omega_{\eta,j})$
whenever $\la \eta,j \ra <_{\lex} \la \nu,k\ra$.

Let $\itT$ be the 1-maximal tree that 
witnesses that  $(P^-,\Sigma)$ iterates to $(M_{\nu,k},\Omega_{\nu,k}^+)$,
$\xi+1$ be its length. If $M_{\nu,k}$ is sound, then
$(M_{\nu,k+1},\Omega_{\nu,k+1}) = (M_{\nu,k}^+, \Omega_{\nu,k}^+)$, so
$\itT$ (regarded as 0-maximal on $(P,\Sigma)$) witnesses that
$(P,\Sigma)$ iterates to $(M_{\nu,k+1},\Omega_{\nu,k+1})$.

So suppose that $M_{\nu,k}$ is not sound. Let $\beta <_T \xi$
be largest such that $\itM_\beta^\itT$ is $k+1$-sound. Since
$P$ has stable type 1, $\beta$ is the largest +1-node
on $[0,\xi]_T$.(This is where we use that fact.) Letting
$\beta_0 \le_T \xi$ be such that $T \tpred(\beta_0)=\beta$,
$\itT$ must have a small drop at $\beta_0$, and no further
drops in $(\beta_0,\xi]_T$. So  setting
\[
i^* = i_{\beta_0,\xi}^\itT \circ i_{0,\beta_0}^\itT
\]
and
\[
R=\dom(i^*)
\]
we get that $R^+ = M_{\nu,k+1}$ and $i^*$ is the anticore map.
Moreover 
\[
R^+=\itM_\beta^\itT
\]
because the drop at $\beta_0$ was small. We then have
\begin{align*}
\Omega_{\nu,k+1} &= (\Omega_{\nu,k}^+)^{i^*}\\
  & = (\Sigma_\xi^\itT)^{i^*}\\
  & = (\Sigma_\beta^\itT),
\end{align*} by Lemma \ref{omeganulplus.pullback} and pullback consistency.
This means that $(P,\Sigma)$ iterates to $(M_{\nu,k+1},\Omega_{\nu,k+1})$
in the sense required by $(*)(P,\Sigma)$.

\subsection{The almost sound case}\label{almostsoundcase}

   Now suppose $(P,\Sigma)$ is not of stable type 1, i.e., it is
either of type 2, or of type 1B and not projectum stable.
This implies $\degr(P)>0$; so say $\degr(P)=k+1$.
Again $(P^-,\Sigma)$ is a +1-hod pair.
Let $N^*, \mathbb{C},$ etc. be as in $(*)(P,\Sigma)$.
Let $(P^-,\Sigma)$ iterate to $(M_{\nu,k},\Omega_{\nu,k})$ via
the 1-maximal tree $\itT$ with length $\xi+1$, as above.

     If $P$ has type 2, then $P^-$ is not sound, so
$M_{\nu,k}$ is not sound. If $P$ has type 1B and
$\eta_{\k+1}^P$ is measurable in $P$, then some total
measure on the image of $\eta_{k+1}^P$ is used in
$[0,\xi)_T$, producing a type 2 ultrapower. So again
$M_{\nu,k}$ is unsound.  In both cases, $M_{\nu,k}$ has
a measurable cardinal $\kappa$ such that $\rho_{k+1}(M_{\nu,k}) < \kappa
< \rho_k(M_{\nu,k})$. Preimages of $\kappa$ were hit in $[0,\xi]_T$,
so $[0,\xi]_T$ has a small drop.

    Thus, letting $\beta$ be the largest +1-node of $\itT$ in
$[0,\xi]_T$, we have $\beta < \xi$. Let 
\begin{align*}
i &= i_{0,\beta}^\itT,\\
R&= \itM_\beta^\itT,
\intertext{ and }
i^* &= \hat{i}_{\beta,\xi}^\itT.
\end{align*}
Here $\degr(R) = k$.
Since $R = \cHull_1^{R^k}(\sup i``\rho_{1}(P^k) \cup \lbrace p_1(R), \rh_1(R) \rbrace)$
and $\sup i``\rho_1(P^k) = \rho_1(R^k)$, 
$R^+$ is of type 2.
$\Sigma_\beta^\itT$ is a +1-iteration strategy for $R$, or equivalently an
ordinary strategy for $R^+$, moreover
\[
\Sigma_\beta^\itT = (\Omega_{\nu,k}^+)^{i^*}.
\]
$(M_{\nu,k+1},\Omega_{\nu,k+1})$ is the {\em type 1 core} of the type 2 pair
$(R^+,\Sigma_\beta^\itT)$, in the following sense.

\begin{definition}\label{type1coredef} If $Q$ is a pfs premouse, then we set
\[
C(Q) = \begin{cases} Q & \text{ if $Q$ has type 1}\\
                     \mathfrak{C}_{\textrm{deg}(Q)}(Q) & \text{ if $Q$ has type 2.}
\end{cases}
\]
We call $C(Q)$ the {\em type 1 core of $Q$}.
$Q$ and $C(Q)$ have the same degree.
If $Q$ has type 2, then $C(Q)$ has type 1A, and we let
$D(Q)$ be the unique order zero measure on $\hat{\rho}(Q^-)$
such that
\[
Q^-= \textrm{Ult}(C(Q)^-,D(Q)).
\] 
The {\em type 1 core of $(Q,\Lambda)$} is $(C(Q),\Lambda^i)$,
where $i = \text{ id}$ if $C(Q)=Q$, and $i = i_{D(Q)}$ otherwise.
\end{definition}

\begin{remark}\label{type1corermk}{\rm 
Notice that if $(Q,\Lambda)$ is a type 2 lbr hod pair of degree $k$,
then for $i=i_{D(Q)}$, $(C(Q),\Lambda^i)$ is a type 1 pair of degree $k$.
$\Lambda^i$ is defined on ``level $k$" stacks because $i$
is nearly +1-elementary as a map from $C(Q)^-$ to $Q^-$.}
\end{remark}

\begin{claim}\label{type1coreclaim} $(M_{\nu,k+1},\Omega_{\nu,k+1})$ is the
type 1 core of $(R^+,\Sigma_\beta^\itT)$.
\end{claim}
\begin{proof} We have that $k=\degr(R)$, and $i^*$ is nearly +1-elementary
as a map from $R$ to $M_{\nu,k}$. It follows that 
\[
\mfc_{k+1}(R) =
\mfc_{k+1}(M_{\nu,k}) = M_{\nu,k+1}.
\]
Moreover, letting $D=D(R)$ and $\pi \colon M_{\nu,k+1} \to M_{\nu,k}$
be the anticore map,  we have the diagram
\begin{center}
\begin{tikzpicture}
    \matrix (m) [matrix of math nodes, row sep=2em,
    column sep=3.0em, text height=1.5ex, text depth=0.25ex]
{ R^+ & M_{\nu,k}\\
M_{\nu,k+1} &\\};
     \path[->,font=\scriptsize]
(m-1-1) edge node[above]{$i^*$}(m-1-2)
(m-2-1) edge node [left]{$i_D$}(m-1-1)
(m-2-1) edge node [below]{$\pi$}(m-1-2);
\end{tikzpicture}
\end{center}
But then
\begin{align*}
\Omega_{\nu,k+1} &= (\Omega_{\nu,k}^+)^\pi\\
                 &= (\Omega_{\nu,k}^+)^{i^* \circ i_D}\\
                 &= (\Sigma_\beta^\itT)^{i_D},
\end{align*}
as required.
\end{proof}

We extend the mouse order to type 2 pairs by regarding
$(Q,\Lambda)$ as being equivalent to its type 1 core.
That is,

\begin{definition}\label{mouseorder2} Let $(Q,\Lambda)$ and $(R,\Omega)$
be mouse pairs with type 1 cores $(Q_0,\Lambda_0)$ and $(R_0,\Omega_0)$
respectively; then $(Q,\Lambda) \le^* (R,\Omega)$ iff
$(Q_0,\Lambda_0) \le^* (R_0,\Lambda_0)$.
\end{definition}

\subsection{Generating type 2 pairs}\label{type2generation}

       We show that any type 2 pair $(Q,\Lambda)$ can be recovered from
its type 1 core $(C(Q),\Omega)$. Of course, $Q$ can be recovered from $C(Q)$ by taking an
ultrapower, the question is how to recover $\Lambda$ from $\Omega$. The tail strategy
$\Omega_{\la D(Q)\ra}$ is a strategy for $Q^-$, not $Q$. But we can apply
the strategy extension method\footnote{See \cite{siskindsteel} and \cite[7.3.11]{normalization_comparison}.}
to $\Omega$ to obtain a strategy for $Q$, and this strategy works out to be $\Lambda$.

        We aren't going to use the result in this subsection later, so we shall
omit some details.

\begin{definition}\label{type2inducedpair} 
Suppose that $(M,\Omega)$ is an lbr hod pair of type 1A
and degree $k+1$, and $D$ is the order zero measure of $M$
on $\kappa$, where $\rho_{k+1}(M) < \kappa < \rho_{k}(M)$.
Suppose also that $\cof_{k+1}^M(\kappa)<\rho_{k+1}(M)$. Let
\[
Q^+ = (\Ult(M^-,D),k+1),
\]
 and for $\itU$ a $\lambda$-separated tree
on $Q^+$ of limit length, let 
\[
\Omega_{\la D \ra}^+(\itU) = b \text{ iff $V(\la D \ra,\itU^\frown b)$ is by $\Omega$};
\]
then we call $(Q^+,\Omega_{\la D \ra}^+)$ the {\em type 2 pair generated by $(M,\Omega)$ and $D$.}
\end{definition}

\begin{lemma}\label{type1root} Let $(Q,\Lambda)$ be a type 2 pair,
and let $(M,\Omega)$ be its type 1 core; then $(Q, \Lambda)$ is the
type 2 pair generated by $(M,\Omega)$ and $D(Q)$.
\end{lemma}

\begin{proof} Let $k+1 = \degr(Q)$, and let
$Q^k = \Ult_0(M^k, D)$ where $D = D(Q)$
 is the order zero measure of
$M^k$ on $\rh_k(Q)$. Let
\[ 
\pi = i_D^{M^-}.
\]
In order to see that $(Q,\Lambda)$ is the type 2 pair generated by
$(M,\Omega)$ and $D$, we must show that $\Lambda = \Omega_{\la D \ra}^+$.
Letting $\itU$ be a plus tree on $Q$, the relevant diagram is

\begin{center}
\begin{tikzpicture}
    \matrix (m) [matrix of math nodes, row sep=3em,
    column sep=3.0em, text height=1.5ex, text depth=0.25ex]
{ Q & R & S\\
M & Q & N\\};
     \path[->,font=\scriptsize]
(m-1-2) edge node[above]{$\tau\itU$}(m-1-3)
(m-1-1) edge node[above]{$\tau$}(m-1-2)
(m-2-1) edge node [left]{$\pi$}(m-1-1)
(m-2-2) edge node [left]{$\tau$}(m-1-2)
(m-2-3) edge node [auto]{}(m-1-3)
(m-2-1) edge node [below]{$\pi$}(m-2-2)
(m-2-2) edge node [below]{$\itU$}(m-2-3);
\end{tikzpicture}
\end{center}
Here $\tau$ is the copy map, or what is the same,
the canonical embedding from $Q$ to $\Ult(Q^-,\pi(D))$.
Since $\tau$ is nearly +1-elementary
as a map on $Q$ it can be used to lift $\itU$.
We then calculate
\begin{align*}
 \itU \text{ is by }\Omega_{\la D \ra}^+(\itU) & \text{ iff } V(\la D \ra, \itU) \text{ is by } \Omega\\
             & \text{ iff } \pi V(\la D \ra,\itU) \text{ is by }\Lambda\\
              & \text{ iff } V(\la \pi(D) \ra , \tau\itU ) \text{ is by }\Lambda\\
          & \text{ iff } \la \pi(D)\ra,\tau\itU \ra \text{ is by } \Lambda\\
           & \text{ iff } \itU \text{ is by }\Lambda.
\end{align*}
Line 3 holds because quasi-normalizing commutes with lifting,
line 4 holds because $\Lambda$ quasi-normalizes well, and
line 5 holds because $\Lambda$ is pullback consistent.
\end{proof}

\section{CONDENSATION LEMMA}\label{sec:condensation}

The main theorem of this section is Theorem \ref{thm:cond_lem}. This theorem will be used 
in the $\square$-construction, but it is
more general than is necessary for that application. Its full generality
is used in \cite{siskindsteel}.

 Our theorem extends
Theorem 9.3.2 of  \cite{Zeman}, which deals with condensation
under $\pi \colon H \to M$ for pure extender mice $H$ and $M$.
That theorem breaks naturally into two cases: either (1)  $H \notin M$,
in which case $H$ is the $\crit(\pi)$-core of $M$, or (2) $H \in M$,
in which case $H$ is a proper initial segment of either $M$ or
an ultrapower of $M$. The proof in case (1) works for
least branch hod mice without much change, so we begin with 
that case.

\begin{definition} \label{alphacore} Let $M$ be an lpm or a pure extender premouse, and $n \le \textrm{deg}(M)$; then
\begin{itemize}
\item[(a)]  $\tilde{h}^{n+1}_M$ is the completion of $h^1_{\hat{M}^n}$, the $\Sigma_1$-Skolem function 
of $\hat{M}^n$. We write $\tilde{h}_M$ for $\tilde{h}^{\textrm{deg}(M)+1}_M$.
 So dom$(\tilde{h}_M) \subseteq M$.\footnote{See \cite[Definition 2.3.9]{normalization_comparison}.}
\item[(b)] Let $\rho(M) \le \alpha$ and $r = p(M)-\alpha$, and suppose that $r$ is solid.
Let $\pi \colon H \to M$ with $H$ transitive be such that
$\ran(\pi) = \tilde{h}_M `` (\alpha \cup r)$, and suppose that $\pi^{-1}(r)$
is solid over $H$. Then we call $H$ the {\em $\alpha$-core} of $M$, and write
$H = \mbox{core}_{\alpha}(M)$. In addition, if $(M,\Sigma)$ is a mouse pair, then the $\alpha$-core of $(M,\Sigma)$ is $(H,\Lambda)$, where $H =$ core$_\alpha(M)$ and $\Lambda = \Sigma^\pi$ and  $\pi$ is the anticore map.
\item[(c)] $M$ is {\em $\alpha$-sound} iff $M = \mbox{core}_{\alpha}(M)$.
\end{itemize}
\end{definition}

We note that core$(M) = \mathfrak{C}_{\textrm{deg}(M)+1}(M)$ is the $\rho(M)+1$-core of $M$. According
 to this definition, if $M$ is $\alpha$-sound, then
$\rho(M)\le \alpha$. So $M$ could be sound, but not $\alpha$-sound because $\alpha < \rho(M)$,
which might be confusing at first.

\begin{remark} {\rm Let $H$ be the $\alpha$-core of $M$, as witnessed by $\pi$.
We have $p(M) \subseteq \ran(\pi)$, so the new $\Sigma_{\textrm{deg}(M)+1}$
subset of $\rho(M)$ is $\Sigma_{\textrm{deg}(M)+1}$ over $H$. Thus $\rho(H) = \rho(M)$ and
$\pi(p(H)) = p(M)$, and $H \notin M$. 

One also gets that if both $H, M$ are of type 1, then $H$ is of type $1A$ iff $M$ is of type $1A$.

One might guess that
$P(\alpha)^M \subseteq H$, but this need not be the case, as the following example
shows. Let $N$ be sound, and let $M = \Ult(N,E)$, where $\rho(N) \le \kappa = \crt(E)$,
and $E$ has one additional generator $\alpha$. Let $H = \Ult(N,E \within \alpha)$,
and let $\pi \colon H \to M$ be the factor map. Clearly, $\pi$ witnesses that
$H$ is the $\alpha$-core of $M$. But $\alpha = (\kappa^{++})^H < (\kappa^{++})^M$,
so $H$ doesn't even have all the bounded subsets of $\alpha$ that are in $M$.} \qedhere
\end{remark}

\begin{theorem}[$\sf{AD}^+$]\label{alphacorecase}
Suppose $(M,\Sigma)$ is a lbr hod pair with scope HC.
Suppose $H$ and $M$ are both type 1 premice, $\pi:H \rightarrow M$ is nontrivial\footnote{$\pi$ is trivial iff $H=M$ and $\pi$
is the identity.}, and letting $n=\textrm{deg}(M) = \textrm{deg}(H)$ and
 $\alpha = crt(\pi)$,\footnote{Here we allow $\alpha$ to be $o(H)$ and $\pi$ to be the identity. }
$\alpha < \rho_n(M)$. Suppose also
\begin{itemize}
\item[(1)] $H$ is $\alpha$-sound, 
\item[(2)] $\pi$ is nearly elementary,
 and
\item[(3)] $H$ is an lpm of the same kind as $M$\footnote{\label{same kind}This means:
 $H$ is passive if and only if $M$ is passive; $H$ is $B$-active if and
  only if $M$ is $B$-active; and $H$ is $E$-active 
 if and only if $M$ is $E$-active; in the third case, $\dot{F}^{H}$ 
  is of type A (B, C) if and only if $\dot{F}^M$
  is of type A (B, C respectively). All but the last clause are implicit in (2).}, and
\item[(4)] $H \notin M$.

Then $H$ is the $\alpha$-core of $M$.
\end{itemize}
 \end{theorem}
\begin{proof} Let $r = p(H) - \alpha$.
\begin{center}
$T = \mbox{Th}^{H^n}_{1}(\alpha \cup r),$
\end{center}
so that $T$ codes $H$.  $T $ is sometimes denoted by Th$^H_{n+1}(\alpha\cup r)$.

Suppose first that $\pi\restriction H^n$ is not cofinal. We have that $T$ is $\Sigma_1$ over $H^n$, and hence $T$ is $\Sigma_1$
over some proper initial segment of $M^n$, so that $T \in M^n$.
If $n>0$, then $M|\rho_n(M) \models {\sf KP}$ and
$T \in M|\rho_n(M)$, so  $H \in M$. If $n =0$ and $H$ is
fully passive, then we have
$\pi \colon H \to M|\eta$ for some $\eta < o(M)$, and $\ran(\pi)$
in $M$. Any premouse is closed under transitive collapse, so we
again get $H \in M$. If $n=0$ and $H$ is extender-active, then letting
$H^- = H||o(H)$, we get $H^- \in M$ by the argument just given.
However, $\dot{F}^H$ can be computed from the fragment $\dot{F}^M \restriction
\sup \pi``o(H)$ and $\pi$ inside $M$, so $H \in M$. The case that $n=0$
and $H$ is branch-active can be handled similarly, noting that the
proper initial
segments of $b^M$ are in $M$.

So we may assume $\pi\restriction H^n$ is cofinal $\Sigma_0$, and hence $\Sigma_1$ elementary. We then get that $\pi(\eta^H_n) = \eta^M_n$.  
We claim that $\rho(M) \le \alpha$. For if not, $T$ is 
a bounded subset of $\rho(M)$ that is $\Sigma_1$ over $M^n$. Thus
$T \in M|\rho(M)$, so $H \in M$.
 
Suppose $r = \emptyset$. If $\gamma \in (p(M) - \alpha)$,
then $T$ can be computed easily from the solidity witness
$W_{r,\gamma}^M$,\footnote{Here recall that $W_{r,\gamma}$ is the transitive collapse of Hull$_1^{M^n}(\gamma\cup \{r-\{\gamma\}\})$.} so $T$ in $M$, and with a bit more work,
$H \in M$. So we have $p(M)-\alpha = \emptyset$, which
implies that $H$ is the $\alpha$-core of $M$,
as witnessed by $\pi$.

Suppose next that $r = \langle \beta_0,...,\beta_l \rangle$, and
$p(M) - \alpha = \langle \gamma_0,...,\gamma_m
\rangle$,
where $\beta_i > \beta_{i+1}$ and $\gamma_i > \gamma_{i+1}$
for all $i$. We show by induction on $i \le l$ that $i \le m$
and $\pi(\beta_i) =
\gamma_i$. Suppose we know it
for $i \le k <l$. Let $W=W^H_{r,\beta_{k+1}}$ be the solidity witness
for $\beta_{k+1}$ in $H$. Since $\pi\restriction H^n$ is $\Sigma_1$
elementary, $\pi(W)$ can be used to compute
$\mbox{Th}^M_{n+1}(\pi(\beta_{k+1}) \cup \lbrace \gamma_0,...,\gamma_k \rbrace)$
inside $M$. But $\rho(M) < \pi(\beta_{k+1})$, so we must have
$k < m$. Similarly, $\gamma_{k+1} < \pi(\beta_{k+1})$ is impossible, as otherwise
$\pi(W)$ could be used in $M$ to compute the $\Sigma_{n+1}$ theory of
$p(M) \cup \rho(M)$. On the other hand, if
$\pi(\beta_{k+1}) < \gamma_{k+1}$, then using the solidity witness
$W_{p(M),\gamma_{k+1}}^M$ for $\gamma_{k+1}$ in $M$, we get $H \in M$.

It follows that $\pi(r) = p(M)-\alpha$, and thus $H$ is the
$\alpha_0$-core of $M$.

\end{proof}

\begin{remark}
In the case $H$ is the core of $M$, we can also get agreement of $\Sigma$ and $\Sigma^\pi$ 
up to $(\rho^+)^H=(\rho^+)^M$. See \cite[Corollary 9.6.6]{normalization_comparison}. It may be possible
to prove strategy condensation in the other cases, but we have not tried to do that.

An analog of the above theorem can be stated and proved for type 2 premice. \qedhere
\end{remark}

Next we deal with condensation under $\pi \colon H \to M$ in the case
$H \in M$.\footnote{ If $\pi \colon H \to M$ is elementary, $\alpha =
\crit(\pi)$, $H$ is 
$\alpha$-sound, and $\alpha < \rho(M)$, then $H \in M$. This is the case
with the coarser condensation results of \cite[5.55]{normalization_comparison}
and \cite[8.2]{fsit}, where $\alpha = \rho(H)$ and $\pi(\alpha)=\rho(M)$.}
 We shall actually prove a stronger result, one that
includes condensation for iteration strategies as well as condensation
for the mice themselves.

The following is an easy case of condensation for pairs.

\begin{lemma}[$\adp$]\label{identityembedding} Let $(M,\Sigma)$
be a mouse pair with scope $\hc$, and let $\pi \colon (H,\Psi)
\to (M,\Sigma)$ be elementary, with $\pi =$ identity; then
either $(H,\Psi) \unlhd (M,\Sigma)$, or
$(H,\Psi) \lhd \Ult((M,\Sigma),E_\alpha^M)$, where $\alpha
= o(H)$.
\end{lemma}

\begin{proof} Suppose first $H$ is extender-active.
Let $F = \dot{F}^H$ and $G = \dot{F}^M$, and let
$\kappa = \crit(F)$. So $\kappa^{+,H} = \kappa^{+,M} < o(H)$,
and $i_F^H `` \kappa^{+,H} = i_G^M `` \kappa^{+,M}$.
Thus $\ran(\pi)$ is cofinal in $o(M)$, which implies
$(H,\Psi) = (M,\Sigma)$.

  Next, suppose that $H$ is branch active.\footnote{Of course,
this only applies when $M$ is an lpm. In general, our proofs
for pure extender pairs are special cases of the proofs
for lbr hod pairs, so it doesn’t hurt to assume
our mouse pair is an lbr hod pair.} Since $\pi$ is
the identity, $\eta = \eta^H = \eta^M$ and $s = s^H = s^M$.
Let $\T = \T(s)$, and let $\nu = \nu^H$,
so that $o(H) = \eta + \nu$. Because $\pi$ preserves
$\dot{B}^H$, $b^H = b^M \cap \nu$. But 
$b^M \cap \nu = b^{M|o(H)}$ because $M$ is an lpm,
so $H = M$. We get $\Sigma^\pi = \Sigma_{l(H)}$
from the internal-lift consistency of $(M,\Sigma)$, so
$(H,\Psi) \unlhd (M,\Sigma)$.

    Finally, suppose that $H$ is fully passive. 
Clearly, $M||\ohat(H)$ is branch-passive, and thus
$M||\ohat(H) = H$. Using internal-lift consistency for
$(M,\Sigma)$, we get 
$(H,\Psi) \unlhd (M,\Sigma)$, unless $M|\hat{o}(H)$ is extender-active.
In that case we get $(H,\Psi) \lhd 
Ult(M,E_\alpha^M)$, where
$\alpha = \hat{o}(H)$, using internal-lift consistency
and strategy coherence.
\end{proof}

 
Our main condensation theorem for mouse pairs is:

\begin{theorem}[$\sf{AD}^+$]\label{thm:cond_lem}
Suppose $(M,\Sigma)$ is a mouse pair with scope HC.
Suppose  $\pi:(H,\Psi) \to (M,\Sigma)$ is nearly elementary, and not
the identity. Let $\alpha = \crit(\pi)$, and suppose
\begin{itemize}
\item[(1)] $\rho(H) \le \alpha < \rho^-(H)$, and $H$ is $\alpha$-sound,  and
\item[(2)] $H$ is a premouse of the same kind as $M$, $\textrm{deg}(H)=\textrm{deg}(M)$, both $H$ and $M$ are solid pfs premice, and
\item[(3)] $H\in M$.
\end{itemize}
Let $(H_0,\Psi_0)$ be the type 1 core of $(H,\Psi)$; then
exactly one of the following holds.
\begin{itemize}
\item[(a)] $(H_0, \Psi_0) \lhd (M,\Sigma)$, 
\item[(b)] $(H_0, \Psi_0) \lhd \Ult_0((M,\Sigma), \dot{E}^M_\alpha)$.
\end{itemize}
\end{theorem}

See Footnote \ref{same kind} for the definition of ``same kind". It does not mean ``same type"; we are
leaving open the possibility that one of $H$ and $M$ has type 1, while the other has type
2. Note that if $\rho(H)<\alpha$, then conclusion (b) is impossible.

 When one applies Theorem \ref{thm:cond_lem} in the proof of $\square_\kappa$ in pfs mice, 
then $M$ has type 1, and one can arrange that $H$ has type 1 as well, so that $(H_0,\Psi_0) = (H,\Psi)$.
One also has
that $H \in M$. In that proof, $\rho(H) = \rho(M) = \kappa$,
and $\alpha =(\kappa^+)^H$, and both $H$ and $M$ are sound. Since $\rho(H) < \alpha$, (b) does not hold. So one gets conclusion (a), that $(H,\Psi) \lhd (M,\Sigma)$.
In the square proof, what matters then is just that $H \unlhd M$; the
full external strategy agreement given by $\Sigma^\pi = \Sigma_{H}$
is not used.

In the applications of \ref{thm:cond_lem} to full normalization
and positionality in \cite{siskindsteel}, one must consider the case that $H$ and $M$ have different types,
and both alternatives in the conclusion are needed.

\begin{remark}\label{rem:consequences}{\rm It follows from the theorem that the hypothesis
 $\alpha < \rho_{\textrm{deg}(H)}(H)$
can be dropped, if one omits condensation of the external
strategy from its conclusion. See \ref{condensationtheorem} below.}
\end{remark}

By condition (3) and Theorem \ref{alphacorecase} (and the remark after), $H$ is not an $\alpha$-core of $M$. If $H$ is $\alpha$-sound and 
is not the $\alpha$-core of $M$, then by \ref{alphacorecase}, $H \in M$. 

Notice that one of the alternatives in the conclusion of 
 \cite[Theorem 9.3.2]{Zeman} does not occur here. The alternative that $(H,\Psi) \lhd \Ult_0((M|\xi,\Sigma), E)$ where $\rho(M|\xi)=\mu$ is the predecessor of $\alpha$ in $M$ and $E$ is an extender on the sequence of $M$ with critical point $\mu$ cannot occur here because
$M$ is projectum solid. 

A relatively coarse special case of Theorem \ref{thm:cond_lem} is sketched in
 \cite[Theorem 5.55]{normalization_comparison}. In that case, $\pi$ 
 is assumed to be fully elementary and crt$(\pi) = \rho(H)$.

\begin{proof}[Proof of Theorem \ref{thm:cond_lem}]

    Let $\pi \colon (H,\Psi) \to (M,\Sigma)$ be nearly elementary,
and let $\alpha_0 = \alpha = \crit(\pi)$ and $\textrm{deg}(H) = \textrm{deg}(M) = k_0$. For definiteness, 
let us assume that $H$ and $M$ are
least branch premice. The proof in the case that they are pure extender mice
 is similar.\footnote{Even
in the pure extender case, one cannot
simply quote 9.3.2 of \cite{Zeman}, because
we are demanding strategy condensation.}\footnote{Under $\adp$, every countable
$\omega_1$-iterable pure extender mouse $M$ has an complete iteration strategy $\Sigma$
such $(M,\Sigma)$ is a pure extender pair. Thus our theorems
\ref{alphacorecase},\ref{identityembedding}, and \ref{thm:cond_lem}
together imply the version of 9.3.2 of \cite{Zeman} for pfs mice, modulo some details about
where the strategies live, and how elementary $\pi$ is. See also Remark \ref{rem:consequences}.}

\begin{definition}\label{dfn:problematic_tuple} A tuple $\la (N,\Phi), (G,\Lambda), \sigma, \nu \ra$ is {\em problematic}
iff  
\begin{itemize}
\item[(1)] $(N,\Phi)$ and $(G,\Lambda)$ are of the same kind, 
with scope $\hc$, and
$G \in N$, 
\item[(2)] 
$\sigma \colon (G,\Lambda) \to (N,\Phi)$ is nearly elementary, with
$\crit(\sigma) = \nu$,
\item[(3)] $\rho(G)\leq \nu < \rho_{\textrm{deg}(G)}(G)$ and $G$ is $\nu$-sound, $\textrm{deg}(N)=\textrm{deg}(G)$, $G, N$ are both solid pfs premice,
and
\item[(4)] letting $(G_0, \Lambda_0)$ be the type 1 core of $(G,\Lambda)$, both conclusions (a) 
and (b) of \ref{thm:cond_lem}
 fail for the pair $(G_0, \Lambda_0), (N,\Phi)$; that is,
it is not the case that $(G_0,\Lambda_0) \lhd (N,\Phi)$, and it is not the case that 
$(G_0,\Lambda_0) \lhd \Ult_0((N,\Phi), \dot{E}^N_\nu)$.
\end{itemize}
\end{definition}


\medskip
\noindent
{\em Claim 1.} 
Let $\la (N,\Phi), (G,\Lambda), \sigma, \nu \ra$ be a 
problematic tuple, and $k = \textrm{deg}(G)$; then
$\sigma(\nu)$ is a cardinal of $N$, $\sigma(\nu) < \rho_k(N)$,
and $G \in N||\sigma(\nu)$.

\medskip
\noindent
{\em Proof.}
By (3), $\rho_{k+1}(G) \le \nu < \rho_k(G)$, so
$\sigma(\nu) < \rho_{k}(N)$. Since $\nu = \crit(\sigma)$, it is a cardinal
of $G$, so $\sigma(\nu)$ is a cardinal of $N$. But $G$ is $\nu$-sound, so it is
coded by a subset of $\nu$ in $N$, so $G \in N||\sigma(\nu)$.

$\hfill    \square$

We must show that $\la (M,\Sigma),(H,\Psi),\pi,\alpha_0\ra$ is not problematic.
Assume toward contradiction that it is, and that $(M,\Sigma)$ is minimal in the
mouse order such that $(M,\Sigma)$ is the first term in some problematic tuple,
and let $k_0 = \degr(M)$.  


We obtain a contradiction by comparing
the phalanx $(M, H, \alpha_0)$ with $M$, as usual. However, since we are
comparing strategies,
this must be done indirectly, by iterating both
into some sufficiently strong background construction $\mathbb{C}$.
 It can happen that at some point, the two sides
  agree with each other (but not with $\mathbb{C}$). This leads to a problem 
  in the argument that the end model on the phalanx side 
  can't be above $M$. The solution employed in \cite{normalization_comparison} is
   to modify how the phalanx is iterated, moving the whole phalanx 
   (including its exchange ordinal) up at certain stages. Our main
new problem here is that because of the restricted elementarity of our
maps, if we move up naively, the new phalanx and associated embedding
 may not be problematic. This forces us to drop to a new
problematic phalanx on occasion.

 \bigskip
 \noindent
 {\em Claim 2.} Let $\la (N,\Phi), (G,\Lambda),\sigma,\nu \ra$ be a problematic tuple,
 and $k = \textrm{deg}(G)$; then we can decompose $\sigma \restriction \hat{G}^k$ as
 \[
 \sigma \restriction \hat{G}^k  = \bigcup_{\eta < \rho_k(G)} \sigma^\eta,
 \]
 where each $\sigma^\eta$ belongs to $\hat{N}^k$.
 
 \medskip
 \noindent
 {\em Proof.}
     Assume first $k=0$ (so $G,N$ are of type 1), and that $\ohat{(G)}$ is a limit ordinal.
 For $\eta < \hat{o}(G)$, let
$G^\eta$ be $G||\eta$, expanded by $I^\eta$, where
$I^\eta$ is the appropriate 
fragment of $\dot{F}^G$ if $G$ is extender active,
the appropriate initial segment of $\dot{B}^G$ if $G$ is
branch active, and $I^\eta = \emptyset$ otherwise	. Let $N^\eta$ be $N||\sigma(\eta)$,
expanded by $\sigma(I^\eta)$. Let $s = p_1(G)\backslash \nu$ and $\sigma^\eta$ be the fragment of
$\sigma$ given by
\[
\dom(\sigma^\eta) = h^1_{G^\eta}``(\nu \cup s),
\]
and
\[
\sigma^\eta(h^1_{G^\eta}(\delta, s))= h^1_{N^{\sigma(\eta)}}(\delta,\sigma(s)),
\]
for $\delta < \nu$. We have that $\sigma^\eta \in N$, and
$\sigma = \bigcup_{\eta < \hat{o}(G)} \sigma^\eta$. If
$\ohat{(G)}$ is a successor ordinal, we can ramify using the $S$-hierarchy.

       The case $k>0$ is similar.
We have $\hat{G}^k = (G||\rho_k(G),A)$ where $A = \hat{A}^k_G$. For $\eta \le \rho_k(G)$, let 
       
\[
G^{\eta} = (G||\eta, A \cap G||\eta) = \hat{G}^k||\eta.
 \]
 Let $s = p_1({\hat{G}^k}) \setminus \nu$, and let
 $h^1_{\hat{G}^k}$ be the $\Sigma_1$ Skolem function, so that
 $\hat{G}^k = h^1_{\hat{G}^k}``(\nu \cup s)$. For $\eta <
 \rho_k(G)$, 
 \[
 \dom(\sigma^\eta) = h^1_{G^{\eta}} `` (\nu \cup s),
 \]
 and for $\gamma < \nu$ in $\dom(\sigma^\eta)$,
 \[
 \sigma^\eta(h^1_{G^{\eta}}(\gamma, s))= h^1_{\hat{N}^k||\sigma(\eta)}(\gamma,\sigma(s)).
 \]
 It is easy to see that this works.
 $\hfill \square$
 
 We call $\la (\sigma^\eta,G^\eta) \mid \nu \le \eta < \rho_k(G) \ra$ 
 as above the
 {\em natural decomposition} of $\sigma \restriction \hat{G}^k$.

     Using claim 2, we can move a problematic
tuple $\la (N,\Phi), (G,\Lambda), \sigma, \nu \ra$ up
via an iteration map
that is continuous at $\rho_k(G)$. When the iteration map
is discontinuous at $\rho_k(G)$, we may have to drop.

\begin{definition} Let $\Phi = \la (N,\Phi), (G,\Lambda), \sigma, \nu \ra$ be a
problematic tuple; then
$\Phi$ is {\em extender-active} iff
$E_\nu^N \neq \emptyset$.
\end{definition}

When we move up extender-active tuples, 
the new exchange ordinal
is always the image of the old one, so the new tuple is
still extender-active.

 \bigskip
 \noindent
 {\em Claim 3.} Let $\la (N,\Phi),(G,\Lambda),\sigma,\nu \ra$ be problematic,
and suppose that $(N,\Phi) \le^* (M,\Sigma)$;
 then there is no proper initial segment $(Q,\Omega)$ of $(G,\Lambda)$
 such that $\nu = \rho(Q)$ and either
 \begin{itemize}
 \item[(i)] $E_\nu^N = \emptyset$, and $(Q,\Omega)$ is not an initial
 segment of $(N,\Phi)$, or
 \item[(ii)] $E_\nu^N \neq \emptyset$, and $(Q,\Omega)$ is not a
 proper initial segment of $\Ult((N,\Phi),E_\nu^N)$.
 \end{itemize}

 \medskip
 \noindent
 {\em Proof.} This follows from the minimality of $(M,\Sigma)$ in the
mouse order. For if $(Q,\Omega)$ is a counterexample,
 then letting $(R,\Gamma) \lhd (N,\Phi)$ be such that
$R = \sigma(Q)$, we have that $(R,\Gamma) <^* (M,\Sigma)$, and
$\la (R,\Gamma), (Q,\Lambda_Q), \sigma \restriction Q, \nu \ra$
is problematic. We note that since $Q\lhd G$, $Q$ is of type 1 and hence $C(Q) = Q$.
 
 $\hfill   \square$

So under the hypotheses of claim 3, $(N,\Phi)$ agrees with $(G,\Lambda)$
strictly below $\nu^{+,G}$.

 We are ready now to enter
 the phalanx comparison argument of \cite{normalization_comparison}.
Fix a coarse strategy pair $((N^*,\in, w, \itF,\Psi),\Psi^*)$
that captures
$\Sigma$, and let $\mathbb{C}$ be the maximal $(w,\itF)$
construction, 
 with models $M_{\nu,l}$ and induced 
strategies $\Omega_{\nu,l}$. Let $\delta^* = \delta(w)$.
By Theorem \ref{starpsigmathm}, $(*)(M,\Sigma)$ holds,
so we can fix $\la \eta_0,k_0\ra$ lex least such that
either
\begin{itemize}
\item[(i)] 
$(M,\Sigma)$ iterates to $(M_{\eta_0,k_0},\Omega_{\eta_0,k_0})$, or
\item[(ii)] $(M,\Sigma)$ iterates to some type 2
pair generated by $(M_{\eta_0,k_0}, \Omega_{\eta_0,k_0})$.
\end{itemize}
As we showed in \S \ref{sec:type2comparison}, alternative (i) occurs
if $M$ has stable type 1, and alternative (ii) occurs otherwise.
Let $\itU_{\eta_0,k_0}$ be the $\lambda$-separated tree on $(P_0,\Sigma_0)$ that witnesses this.
For $\la \nu,l\ra <_{\lex} \la \nu_0,k_0\ra$ 
let $\mathcal{U}_{\nu,l}$ be the unique $\lambda$-separated tree on $M$ witnessing 
that $(M,\Sigma)$ 
  iterates strictly past $(M_{\mu,l},\Omega_{\nu,l})$. \footnote{ We can work in $N^*$ from now on, and interpret 
  these statements there. But in fact, the strategies $\Omega_{\nu,l}$ are induced
  by $\Sigma^*$ in a way that guarantees they extend to $\Sigma^*$-induced
  strategies $\Omega_{\nu,l}^{*}$ defined on all of HC. $\mathcal{U}_{\nu,l}$
  iterates $(M,\Sigma)$ past $(M_{\nu,l},\Omega_{\nu,l}^{*})$ in $V$.} 

We define $\lambda$-separated trees $\itS_{\nu,l}$ on $(M,H,\alpha_0)$ for certain $(\nu,l)\leq (\eta_0,k_0)$. 
Fix $(\nu,l)\leq(\eta_0,k_0)$
 for now, and assume $\itS_{\nu',l'}$ is defined whenever $(\nu',l')<(\nu,l)$. 
 Let $\U = \U_{\nu,l}$, and for $\tau < \lh(\U)$, let
\[
\Sigma_\tau^\U = \Sigma_{\U \restriction (\tau+1)}
\]
be the tail strategy for $\M_\tau^\U$ induced by $\Sigma$. 
 We proceed to define $ \mathcal{S} = \itS_{\nu,l}$,
by comparing the phalanx $(M, H,\alpha_0)$ with $M_{\nu,l}$. As we define
 $\mathcal{S}$, we 
  lift $\mathcal{S}$ to a padded tree $\mathcal{T}$ on $\M$, by copying. Let us
write
\[
\Sigma_\theta^\T =\Sigma_{\T \restriction (\theta+1)}
\]
for the tail strategy for $\M_\theta^\T$ induced by $\Sigma$.
For $\theta < \lh(\mathcal{S})$, we will have copy map
\begin{center}
$\pi_\theta: \M^\itS_\theta \rightarrow \M^\T_\theta$.
\end{center} 
The map $\pi_\theta$ is nearly elementary. We attach the complete strategy
\[
 \Lambda_\theta = (\Sigma_\theta^\T)^{\pi_\theta}
\]
  to $\M^\itS_\theta$. We also
  define a non-decreasing sequence of ordinals $\lambda_\theta = \lambda^\itS_\theta$ 
  that measure
   agreement between models of $\itS$, and tell us which 
   model we should apply the next
    extender to.

The following claim will be useful in pushing up problematic tuples.

\medskip
\noindent
{\em Claim 4.} Suppose $\xi <_S \theta$ and $(\xi,\theta]_S$ does not drop;
then $\Lambda_\xi = \Lambda_\theta^{i_{\xi,\theta}^\itS}$.

\medskip
\noindent
{\em Proof.} Because $\Sigma$ is pullback consistent, we have
$\Sigma_\xi^\T = (\Sigma_\theta^\T)^{i_{\xi,\theta}^\T}$. 
But then
\begin{align*}
\Lambda_\xi & = (\Sigma_\xi^\T)^{\pi_\xi} \\
            & = (\Sigma_\theta^\T)^{i_{\xi,\theta}^\T \circ \pi_\xi}\\
            & = (\Sigma_\theta^\T)^{\pi_\theta \circ i_{\xi,\theta}^\itS}\\
            & = \Lambda_\theta^{i_{\xi,\theta}^\itS},\\
\end{align*}
as desired.
$\hfill    \square$

We start with 
\begin{center}
$\M^\itS_0 = M, \M_1^\itS = H, \lambda_0 = \alpha_0,$
\end{center}
and
\begin{center}
$\M^\T_0 = \M^\T_1 = M,  \pi_0 = id,  \pi_1 = \pi$,
\end{center}
and
\begin{center}
$\Lambda_0 = \Sigma, \mbox{ } \Lambda_1 = \Sigma^{\pi_1}.$
\end{center}

We say that $0, 1$ are distinct roots of $\itS$. We say that $0$ is
unstable, and $1$ is stable\footnote{This is different from the notion of ``stable pfs premice". We shall
 let context dictate the meaning of the term.}. As we proceed, we shall declare
additional nodes $\theta$ of $\itS$ to be unstable.
We do so because $(\M_\theta^\itS,\Lambda_\theta) = 
(\M_\gamma^\U, \Sigma_\gamma^\U)$ for some
$\gamma$\footnote{ In the first version of
\cite{normalization_comparison} the external
strategy agreement was not required for $\theta$ to be declared
unstable, but it is important to do so here.}, and when we do so, we shall
immediately define $\M_{\theta+1}^\itS$, as well as 
$\sigma_\theta$ and $\alpha_\theta$
such that 
\begin{center}
$ \Phi_\theta =_{\mbox{df}} \la (\M_\theta^\itS, \Lambda_\theta),
(\M_{\theta +1}^\itS,\Lambda_{\theta+1}),\sigma_\theta,
 \alpha_\theta \ra$
\end{center}
is a problematic tuple. Here $\Lambda_{\theta+1} =
\Lambda_\theta^{\sigma_\theta}$.
In this case, $[0,\theta]_S$ does not drop, and
all $\xi \le_S \theta$ are also unstable. We regard
$\theta+1$ as a new root of $\itS$. This is the only way new roots are
constructed.

Let us also write
\begin{center}
$ \Phi_\theta^- =_{\mbox{df}} \la \M_\theta^\itS,\M_{\theta +1}^\itS,
\sigma_\theta,
 \alpha_\theta \ra$
\end{center}
for the part of $\Phi_\theta$ that is definable over $\M_\theta^\itS$.
We say {\em $\Phi_\theta^-$ is problematic} iff $\Phi_\theta$ is problematic
for reason that do not involve the external strategies; that is, iff
both statements ``$C(\M^\itS_{\theta+1})\lhd \M^\itS_\theta$" and ``$C(\M^\itS_{\theta+1}) \lhd
 \textrm{Ult}_0(\M^\itS_\theta, E^{\M^\itS_\theta}_{\alpha_\theta})$" are false.

If $\theta$ is unstable, then we define
\[
\beta_\theta = (\alpha_\theta^+)^{\M_{\theta+1}^{\itS}}.
\]
If $\xi <_S, \theta$, then we shall have
$ \beta_\theta \le i_{\xi,\theta}^{\itS}(\beta_\xi)$, and 
\[
\beta_\theta = i_{\xi,\theta}^{\itS}(\beta_\xi) \Rightarrow 
\Phi_\theta = i_{\xi,\theta}^{\itS}(\Phi_\xi),
\]
in the appropriate sense.
In this connection: it will turn out
that $i_{\xi,\theta}(\beta_\xi) = \beta_\theta$ implies
$i_{\xi,\theta}^\itS$ is continuous at
$\rho_k(\M_{\xi +1}^\itS)$, where $k = k_0 =\textrm{deg}(\M_{\xi +1}^\itS)$.
So we can set
\[
i_{\xi,\theta}^\itS(\sigma_\xi) = \text{ upward extension of }
\bigcup_{\eta < \rho_k(\M_{\xi+1}^\itS)} i_{\xi,\theta}^\itS(\sigma_\xi^\eta),
\]
 where $\la \sigma_\xi^\eta \mid \eta < \rho_k(\M_{\xi+1}^\itS)\ra$
is the natural decomposition of $\sigma_\xi$. This enables us to
make sense of $i_{\xi,\theta}^\itS(\Phi_\xi^-)$.

The construction of $\itS$ takes place in rounds in which we either
add one stable $\theta$, or one unstable $\theta$ and its
stable successor $\theta+1$. Thus the current last model is always
stable, and all extenders used in $\itS$ are plus extenders taken from stable models.
If $\gamma$ is stable, then  
\[
\lambda_\gamma =
 \hat{\lambda}(E_\gamma^{\itS}).
\]
In the case $\gamma$ is 
unstable and $\gamma+1 < \lh(\itS)$, we first define 
$\lambda_{\gamma+1}$ and set $$\lambda_\gamma = 
\textrm{inf}(\lambda_{\gamma+1},\alpha_\gamma).$$ See below for more detail.

In sum, we are maintaining by induction that the last node $\gamma$
of our current $\S$  is stable, and

\bigskip
\noindent \textbf{Induction hypotheses $(\dagger)_\gamma$.} If $\theta < \gamma$ 
and $\theta$ is unstable, then 
\begin{itemize}
\item[(1)] $0\leq_S\theta$ and $[0,\theta]_S$ does not drop (in model or degree),
and every $\xi \leq_S \theta$ is unstable,
\item[(2)] there is a $\gamma$ such that $(\M_\theta^\itS, \Lambda_\theta) = 
(\M_\gamma^\U, \Sigma_\gamma^\U)$,
\item[(3)] $\Phi_\theta = \la (\M_\theta^\itS,\Lambda_\theta), (M_{\theta+1}^\itS, \Lambda_{\theta+1}), 
\sigma_\theta, \alpha_\theta \ra$
 is a  problematic tuple,
 \item[(4)] $\Phi_\theta$ is extender-active iff $\Phi_0$ is extender-active,
 and if $\Phi_\theta$ is extender-active, then $i_{0,\theta}^\itS(\alpha_0)
 = \alpha_\theta$,
\item[(5)] if $\xi <_S \theta$, then $\alpha_\theta \le i^\itS_{\xi,\theta}(\alpha_\xi)$ and
$\beta_\theta \le i^{\itS}_{\xi,\theta}(\beta_\xi)$,
\item[(6)]  if
$\la \alpha_\theta, \beta_\theta \ra = i_{\xi,\theta}^{\itS}(\la \alpha_\xi,
 \beta_\xi \ra)$,
then 
\begin{itemize}
\item[(a)] $\Phi_\theta^- = i^{\itS}_{\xi,\theta}(\Phi_\xi^-)$, and
\item[(b)]$i_{\xi,\theta}^\itS$ is continuous at $\rho_{k_0}(\M_{\xi +1}^\itS)$, where
$k_0 = \textrm{deg}(\M_{\xi+1}^\itS)$,
\end{itemize}
 \item[(7)] $\M_{\theta +1}^\T = \M_\theta^\T$, and 
$\pi_{\theta+1} = \pi_\theta \circ \sigma_\theta$.
\end{itemize}

\medskip

Setting $\sigma_0 = \bar{\pi}$, we have $(\dagger)_1$.

For a node $\gamma$ of $\itS$, we write $\itS$-pred$(\gamma)$ for the immediate 
$\leq_S$-predecessor of $\gamma$. For $\gamma$ a node in $\itS$, we set 
\begin{center}
st$(\gamma) = $ the least stable $\theta$ such that $\theta\leq_S\gamma$,
\end{center}
and
\begin{center}
rt$(\gamma)=\begin{cases}
 \text{$\itS$-pred(st$(\gamma)$)} &  \ \text{if $\itS$-pred(st$(\gamma)$) exists} \\ 
\text{st$(\gamma)$} &  \ \text{otherwise}.
\end{cases}$
\end{center}

The construction of $\itS$ ends when we reach a stable $\theta$ 
such that

\begin{enumerate}[(I)]
\item $(M_{\nu,l},\Omega_{\nu,l}) \lhd (\M^\itS_\theta, \Lambda_\theta)$, or $(M_{\nu,l},\Omega_{\nu,l}) = (\M^\itS_\theta, \Lambda_\theta)$ and $[$rt$(\theta),\theta]_S$ drops, or
\item $[$rt$(\theta),\theta]_S$ does not drop in model
 or degree and either $(\M^\itS_\theta, \Lambda_\theta) \unlhd (M_{\nu,l}, \Omega_{\nu,l})$, 
 or $(M_{\nu,l},\Omega_{\nu,l})$ is the type 1 core
of $(\M^\itS_\theta, \Lambda_\theta)$. 
 
\end{enumerate}
We need the second alternative in (II) because models on the phalanx side might
have type 2. The comparison
arguments of \S 3 and \cite{normalization_comparison} show
that if $\la \nu,l\ra$ is
least such that (II) holds, and
$(\M^\itS_\theta, \Lambda_\theta) \unlhd (M_{\nu,l}, \Omega_{\nu,l})$, 
then 
$(\M^\itS_\theta, \Lambda_\theta) = (M_{\nu,l}, \Omega_{\nu,l})$,

If case (I) occurs, then we go on to define $\itS_{\nu,l+1}$. If case (II) occurs, 
we stop the construction. We will sometimes write ``case (II)(a)" for the first case 
of (II) and ``case (II)(b)" for the second case. 

We now describe how to extend $\itS$ one more step. First we assume $\itS$ has successor
 length $\gamma+1$ and let $\M^\itS_\gamma$ be the current last model, so that $\gamma$ is
  stable. Suppose $(\dagger)_\gamma$ holds. Suppose (I) and (II) above do not hold 
  for $\gamma$, so that we have a least
   disagreement between $\M^\itS_\gamma$ and $M_{\nu,l}$.  By the results of \S 3, i.e. essentially the proof of 
   \cite[Lemma 9.6.5]{normalization_comparison}, the least
disagreement involves only an extender $E$ on the sequence of $\M^\itS_\gamma$ as 
long as $\crit(E) < \rho_l(M_{\nu,l})$.\footnote{The only extender that has critical
 point $\geq \rho_l(M_{\nu,l})$, if exists, is the order zero measure
 $F = D(\itM_\gamma^\itS)$  witnessing case (II)(b). This happens precisely 
when some model on the main branch $[$rt$(\theta),\theta]_S$ has type 2. In this 
case, $F$ is the last extender used in the comparison and is the only extender 
on the $M_{\nu,l}$-sequence used. } In this case, 
letting $\tau = \lh(E)$, we have
\begin{itemize}
\item $M_{\nu,l}|(\tau,0) = \M^\itS_\gamma|(\tau,-1)$,\footnote{Recall 
$\M^\itS_\gamma|(\tau,-1)$ is the structure obtained from $\M^\itS_\gamma|\tau$
 by removing $E$.} and
\item ${(\Omega_{\nu,l})}_{(\tau,0)}=(\Lambda_\gamma)_{(\tau,-1)}$.
\end{itemize}  

Set 
\[
\lambda^\itS_\gamma = \lambda_E,
\]
 and let $\xi$ be least such that 
$\crit(E)<\lambda^\itS_\xi$. We let $\itS$-pred$(\gamma+1)=\xi$. Let $(\beta,k)$ be 
lex least such that either $\rho(\M^\itS_\xi|(\beta,k))\leq  \crit(E)$ or 
$(\beta,k)=(\hat{o}(\M^\itS_\xi),\textrm{deg}(\M^\itS_\xi))$. Set $R =
 \M^\itS_\xi|(\beta,k)$ and set 
\[
\M^\itS_{\gamma+1} = \textrm{Ult}(R,E^+),
\]
and let $\hat{i}^\itS_{\xi,\gamma+1}$ be the canonical embedding.  
Let 
\begin{center}
$\M^\T_{\gamma+1}=$Ult$(\M^\T_\xi|(\pi_\xi(\beta),k),\pi_\gamma(E^+))$,
\end{center}
and let $\pi_{\gamma+1}$ be given by the Shift Lemma . This determines
$\Lambda_{\gamma +1}$. We write $E^\itS_\gamma$ for $E^+$ and similarly we write $E^\T_\gamma$ for $\pi_\gamma(E^+)$. Similar notations apply 
to extenders $E^\U_\gamma$ on the tree $\U$.

If $\xi$ is stable or
 $(\beta,k)< (\hat{o}(\M^\itS_\xi),\textrm{deg}(\M^\itS_\xi))$, then we declare $\gamma+1$ to be stable.
 $(\dagger)_{\gamma+1}$ follows vacuously from $(\dagger)_\gamma$.\footnote{It is possible 
  that  $\xi$ is unstable, $\lambda_\xi = \alpha_\xi$, $\itS$-pred$(\gamma+1)=\xi$, and 
  crt$(E^\itS_\gamma)=\lambda_F$ where $F$ is the last extender of $\M^\itS_\xi|\alpha_\xi$. 
  In this case, $(\beta,k) = (\lh(F),0)$. The problem then is that
   $\M^\itS_{\gamma+1}$ is not an lpm, because its last extender
   $i_{\xi,\gamma +1}(F)$ has a missing whole initial segment,
   namely $F$. Schindler and Zeman found a way to deal with this anomaly
   in \cite{deconstructing}. Their method works in the hod mouse context as well.
   Here we shall not go into the details of this case. The anomaly cannot occur
   when $\xi$ is stable, because then $\lambda_\xi = \hat{\lambda}(E_\xi^{\itS})$ is inaccessible
   in $\M_\gamma^{\itS}$.}
   
If $\xi$ is unstable, $(\beta,k)= (\hat{o}(\M^\itS_\xi),\textrm{deg}(\M^\itS_\xi))$, and 
$(\M^\itS_{\gamma+1},\Lambda_{\gamma+1})$ is not a model of $\U$, then again we declare $\gamma+1$ stable. 
Again, $(\dagger)_{\gamma+1}$ follows vacuously from $(\dagger)_\gamma$.

 Finally, suppose 
 $\xi$ is unstable, $(\beta,k)= (\hat{o}(\M^\itS_\xi),\textrm{deg}(\M^\itS_\xi))$,
 and for some $\tau$,
\[
(\M^\itS_{\gamma+1},\Lambda_{\gamma+1}) = (\M_\tau^\U, \Sigma_\tau^\U).
\]
 We then we declare $\gamma+1$ to 
 be unstable
  and $\gamma+2$ stable. 
  We must define the problematic tuple needed for $(\dagger)_{\gamma+2}$.
  Let $i = i_{\xi,\gamma+1}^\itS$, and
  \begin{center}
  $\la (N,\Psi),(G,\Phi),\sigma,\alpha \ra = \la (\M_\xi^\itS,\Lambda_\xi),
(\M_{\xi+1}^\itS,\Lambda_{\xi+1}), \sigma_\xi,\alpha_\xi \ra.$
  \end{center}
  We have that $\la (N,\Psi),(G,\Phi),\sigma,\alpha \ra$ is problematic. We must define $\Phi_{\gamma+1}$ and verify $(\dag)_{\gamma+2}$. We break into 2 cases. Let $k =
  \textrm{deg}(G)$. (So $k = k_0 = \textrm{deg}(N) = \textrm{deg}(M)$.)
  
  \bigskip
  \noindent
  {\textbf{ Case 1.} } $i$ is continuous at $\rho_k(G)$.
  
  \medskip

      In this case, we simply let
      \begin{center}
$\la \M_{\gamma+2}^\itS, \alpha_{\gamma+1} \ra = 
\la i(G), i(\alpha) \ra.$
\end{center}
We must define $\sigma_{\gamma+1}$. Note that by our case hypothesis,
\[
i(\hat{G}^k) = \hat{i(G)}^k.
\]
 Let
$\la \sigma^\eta \mid
\eta < \rho_k(G) \ra$ be the natural decomposition of
$\sigma \restriction \hat{G}^k$, and set
\[
i(\sigma \restriction \hat{G}^k) = \bigcup_{\eta < \rho_k(G)}i(\sigma^\eta).
\]
Using the continuity of $i$ at $\rho_k(G)$,
it is easy to see that $i(\sigma \restriction \hat{G}^k)$ is $\Sigma_0$-elementary and cardinal preserving
from $\hat{i(G)}^k$ to $\hat{i(N)}^k$. We set
 \[
\sigma_{\gamma +1} = \text{ completion
of $i(\sigma \restriction \hat{G}^k)$ via upward extension of embeddings,}
\]
and
\[
 \Lambda_{\gamma+2}  = \Lambda_{\gamma+1}^{\sigma_{\gamma+1}}.
 \]
 This defines $\Phi_{\gamma+1}$, that is 
 \begin{center}
 $\Phi_{\gamma+1} =  \la (i(N),\Lambda_{\gamma+1}),(i(G),\Lambda_{\gamma+2},i(\sigma),i(\alpha) \ra$. 
 \end{center}
 Abusing notation a bit, let us write
 \[
 \Phi_{\gamma+1}^- = \la i(N),i(G),i(\sigma),i(\alpha) \ra.
 \]
     
       We must see that $\Phi_{\gamma+1}$ is problematic.  First, it satisfies
the hypotheses of the Condensation Theorem \ref{thm:cond_lem}. For $G$ is
$\alpha$-sound, so $i(G)$ is $i(\alpha)$-sound. (Note $G \in N$, so $i \within G$
is $\Sigma_\omega$-elementary.) 
By downward extension of embeddings
(cf. \cite[Lemma 3.3]{schindler2010fine}), $i(\sigma \restriction \hat{G}^k)$
extends to a unique embedding from some $K$ into $i(N)$, and it is easy
to see that $K = i(G)$, because $i(G)$ is $k$-sound, and that the
embedding in question is what we have called $i(\sigma)$.  $i(\sigma)$ is nearly
elementary: it maps parameters correctly, and $i(\sigma)\restriction \hat{i(G)}^k$ is $\Sigma_0$-elementary 
and cardinal preserving by
construction.

    Finally, $\crit(i(\sigma)) = i(\alpha)$, because
for all sufficiently large $\eta < \rho_k(G)$,
$\alpha+1 \subseteq \dom(\sigma^\eta)$ and $\crit(\sigma^\eta)=\alpha$, so
$\crit(i(\sigma^\eta))=i(\alpha)$.

So we must see that one of the conclusions of \ref{thm:cond_lem} fails. We show that the
conclusion that failed for $\Phi_\xi$ fails for $\Phi_{\gamma+1}$.

       Suppose
first that $\Phi_\xi^-$ is problematic. We break into cases.
If $E_\alpha^N = \emptyset$ and $G$ is of type 1, then $\neg G \lhd N$.
But then $E_{i(\alpha)}^{i(N)} = \emptyset$, $i(G)$ is of type 1, and
$\neg i(G) \lhd  i(N)$, so $\Phi_{\gamma+1}^-$ is problematic.
If $G$ is of type 1 and $E_\alpha^N \neq \emptyset$, then
$\neg G \lhd \Ult(N,E_\alpha^N)$. But then $i(G)$ is of type 1,
$E_{i(\alpha)}^{i(N)} \neq \emptyset$, and $\neg i(G) \lhd
\Ult(i(N),E_{i(\alpha)}^{i(N)})$, so $\Phi_{\gamma+1}^-$
is problematic.\footnote{In both cases, $i$ has enough elementarity to preserve these facts. 
For instance, if $\neg (G\lhd \textrm{Ult}(N,E^N_\alpha))$, then this is equivalent 
to $\neg (G\lhd  \textrm{Ult}(N|(\alpha^+)^N, E^N_\alpha))$, and $i$ preserves this fact.} 
If $G$ is of type 2, since $\Phi_\xi^-$ is problematic, no $L\lhd N$ (or if 
$E^N_\alpha\neq \emptyset$, no $L\lhd \textrm{Ult}_0(N, E^N_\alpha)$) is such
 that $L = C(G)$.
 By elementarity of $i$ the same fact holds for $i(G)$ and therefore, $\Phi_{\gamma+1}^-$
is problematic.

     So we may assume $\Phi_\xi^-$ is not problematic, and hence
also $\Phi_{\gamma+1}^-$ is not problematic. Suppose first
$G \lhd N$, so $i(G) \lhd i(N)$. We must show
$\Lambda_{\gamma+1}^{i(\sigma)} \neq (\Lambda_{\gamma+1})_{i(G)}$, so suppose
otherwise. Using Claim 4 we get $\Lambda_\xi = \Lambda_{\gamma+1}^i$, so

\begin{align*}
\Lambda_\xi^\sigma & = (\Lambda_{\gamma+1})^{i \circ \sigma}\\
                 & = (\Lambda_{\gamma+1})^{i(\sigma) \circ i}\\
                 & =(\Lambda_{\gamma+1}^{i(\sigma)})^i \\
                & = ((\Lambda_{\gamma+1})_{i(G)})^i\\
                & = (\Lambda_\xi)_G,\\
\end{align*}
a contradiction. Equation 2 holds because $i \circ \sigma =
i(\sigma) \circ i$, and equation 5 comes from equation 4 using
claim 4 again. Thus $\Phi_\xi$ is not problematic, contradiction.

Suppose next $G\lhd \textrm{Ult}(N,E_\alpha)$. We want to show 
$(\Lambda_\xi)_{\langle E_\alpha,G \rangle} = \Lambda_\xi^\sigma$. This 
comes from the following calculations (see Figure \ref{fig:type1_ultrapower}).

\begin{align*}
\Lambda_\xi^\sigma & = (\Lambda_{\gamma+1})^{i \circ \sigma}\\
                  &= (\Lambda_{\gamma+1})^{i(\sigma) \circ i}\\
                 & = (({\Lambda_{\gamma+1}})_{\langle i(E_\alpha),i(G)\rangle})^{i}\\
                & = (\Lambda_\xi)_{\langle E_\alpha,G\rangle}.\\
\end{align*}
This first equality follows from Claim 4. The third equality follows from the 
assumption that $\Phi_{\gamma+1}$ is not problematic. The last equality 
follows from the fact that $\Lambda_\xi = (\Lambda_{\gamma+1})^i$, so the tail 
strategy $(\Lambda_{\gamma+1})_{\langle i(E_\alpha),i(N)\rangle}$ is pulled back
by the Shift lemma map from Ult$(N,E_\alpha)$ to Ult$(i(N),i(E_\alpha))$, which is just $i$.

\begin{figure}
\centering
\begin{tikzpicture}[node distance=2cm, auto]
 \node (A) {$G$};
\node (B) [left of=A, node distance=0.75cm] {$\rhd$};
\node (C) [left of=B, node distance=1.5cm] {Ult$(N,E_\alpha)$};
\node (D) [above of=A]{$i(G)$};
\node (E) [above of=B] {$\rhd$};
\node (F) [above of=C] {Ult$(i(N),i(E_\alpha))$};
\node (G) [right of=A]{$N$};
\node (H) [above of=G]{$i(N)$};
\draw[->] (A) to node {$i$} (D);
  
\draw[->] (C) to node {$i$} (F);
\draw[->] (G) to node {$i$} (H);
\draw[->] (A) to node {$\sigma$} (G);
\draw[->] (D) to node {$i(\sigma)$} (H);
\end{tikzpicture}
\caption{Diagram illustrating that $(\Lambda_\xi)_{\la E_\alpha,G\ra} = \Lambda_\xi^\sigma$.}
\label{fig:type1_ultrapower}
\end{figure}

Suppose next that $G$ is of type 2. Suppose $L\lhd N$ is such that $C(G)=L$. (Then case that $L$ is an ultrapower away
from $N$ is similar.) Let $j \colon L \to G$ be the ultrapower map,
and let $l = i(j)$, so that $l \colon i(L) \to i(G) = \Ult_0(i(L),i(D))$.
$i(G)$ has type 2, so since $\Phi_{\gamma+1}$ is not problematic, $i(L) \lhd i(N)$
and $(\Lambda_{\gamma+1})_{i(L)} = \Lambda_{\gamma+1}^{i(\sigma) \circ l}$.
We wish to show that $(\Lambda_\xi)_L = \Lambda_\xi^{\sigma \circ j}$.\footnote{ As we
remarked in \ref{type1corermk}, $\Lambda_{\gamma+1}^{i(\sigma) \circ l}$
and $\Lambda_\xi^{\sigma \circ j}$ are indeed level $k$ strategies.}  But we can calculate
\begin{align*}
\Lambda_\xi^{\sigma \circ j} & = \Lambda_{\gamma+1}^{i \circ \sigma \circ j}\\
                             & = \Lambda_{\gamma+1}^{i(\sigma)\circ l \circ i} = (\Lambda_{\gamma+1}^{i(\sigma) \circ l})^i\\
                             & = ((\Lambda_{\gamma+1})_{i(L)})^i\\
                             & = (\Lambda_\xi)_L.
\end{align*}
Line 1 uses Claim 4, as does the step from line 3 to line 4. Line 2 uses the
commutativity of the diagram in Figure \ref{Pullback}. As we noted, line 3 holds because $\Phi_{\gamma+1}$
is not problematic.

\begin{figure}
\centering
\begin{tikzpicture}[node distance=2cm, auto]
 \node (A) {$G$};
\node (B) [left of=A, node distance=2cm] {$L$};
\node (C) [right of=A, node distance=2cm] {$N$};
\node (D) [above of=A]{$i(G)$};
\node (E) [above of=B] {$i(L)$};
\node (F) [above of=C] {$i(N)$};
\draw[->] (A) to node {$i$} (D);
  
\draw[->] (B) to node {$i$} (E);
\draw[->] (C) to node {$i$} (F);
\draw[->] (B) to node {$j$} (A);
\draw[->] (E) to node {$l$} (D);
\draw[->] (A) to node {$\sigma$} (C);
\draw[->] (D) to node {$i(\sigma)$} (F);
\end{tikzpicture}
\caption{Diagram illustrating that $\Lambda_{\xi}^{\sigma \circ j}= (\Lambda_{\xi})_L$.}
\label{Pullback}
\end{figure}

         Thus $\la \M_{\gamma +1}^\itS, \M_{\gamma+2}^\itS,
\sigma_{\gamma+1},\alpha_{\gamma+1} \ra$ is problematic. Setting 
\begin{center}
$ \M_{\gamma+2}^\T = \M_{\gamma+1}^\T \text{ and } \pi_{\gamma+2} = 
\pi_{\gamma +1} \circ \sigma_{\gamma+1},$
\end{center}
the rest of $(\dagger)_{\gamma+2}$ is clear.
         
         \bigskip
         \noindent
         {\textbf{ Case 2.} } $i$ is discontinuous at $\rho_k(G)$.
         
         \medskip
       Set $\kappa = \crt(E_\gamma^\itS)$. In case 2, $\rho_k(G)$ has cofinality
       $\kappa$ in $N$. Since $\rho(G) \le \alpha$ and $G$ is $\alpha$-sound, we have
       a $\Sigma_1$ over $\hat{G}^k$ map of $\alpha$ onto $(\alpha^+)^G$. Ramifying
       this map, we see that $(\alpha^+)^G$ also has cofinality $\kappa$ in $N$. Here recall that $\rho(G) = \rho_1(\hat{G}^k)$ and $p(G) = p_1(\hat{G}^k)$.
       
        Let
       $\la (\sigma^\eta,G^\eta) \mid \alpha \le \eta < \rho_k(G) \ra$
         be the 
natural decomposition of $\sigma \restriction \hat{G}^k$.\footnote{We encourage the
reader to focus on the case $k=0$, which has the main ideas.}
 Let $s = p(G) -\alpha$, so that
      \[
      \dom(\sigma^\eta) = h^1_{G^\eta}``(\alpha \cup s).
      \]
      Let
       \[
       \bar{\tau} = \bigcup_{\eta < \rho_k(G)} i(\sigma^\eta).
       \]
        The domain of $\bar{\tau}$ is no longer all of $i(\hat{G}^k)$, instead
       \[
       \dom(\bar{\tau}) = \bigcup_{\eta < \rho_k(G)} h^1_{i(G^\eta)}``(i(\alpha) \cup i(s)).
       \]
        But set
       \[
       J = \Ult(G,E_i \restriction \sup i``\alpha),
       \]
       and let $t \colon G \to J$ be the canonical embedding,
       and $v \colon J \to i(G)$ be the factor map. 
       This is a $\Sigma_k$ ultrapower, by what may
       be a long extender. That is, $J$ is the decoding of
       $\hat{J}^k = \Ult_0(\hat{G}^k,E_i \restriction \sup i``\alpha)$. Note
       that $t$ is continuous at $\alpha$, because $\alpha$ is regular
       in $G$ (because $\alpha = \crit(\sigma)$), and
       $\alpha < \rho_k(G)$.
       
             We claim that 
       $\ran(v \restriction \hat{J}^k) \subseteq \dom(\bar{\tau})$.
       For let $f \in \hat{G}^k$ and $b \subseteq \sup i``\alpha$ be finite, so that
       $t(f)(b)$ is a typical element of $\hat{J}^k$, and $v(t(f)(b)) =
       i(f)(b)$. We can find $\eta < \rho_k(G)$ such that $f \in
       \dom(\sigma^\eta)$ and $\eta > \alpha$, so that $i(f) \in 
       \dom(i(\sigma^\eta))$ and $b \subseteq i(\eta)$.  Since 
       $f ``(\alpha) 
       \subseteq \dom(\sigma^\eta)$,
       $i(f)``i(\alpha) \subseteq \dom(i(\sigma^\eta))$, so $i(f)(b)
       \in \dom(\bar{\tau})$, as desired.
       
       Let $\tau$ be the extension of $\bar{\tau}$ given by: for
       $a \subseteq \sup i``\rho_k(G)$ finite,
       \[
       \tau(\tilde{h}^{k+1}_{i(G)}(a,p(i(G))) = \tilde{h}^{k+1}_{i(N)}(\bar{\tau}(a), p(i(N))).
       \]
       It is easy to check that $\ran(v) \subseteq \dom(\tau)$.
       
       This gives us the diagram in Figure \ref{Lift_up}.
       
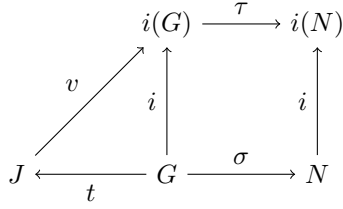
\begin{figure}
\centering
\begin{tikzpicture}[node distance=2cm, auto]
 \node (A) {$J$};
\node (B) [right of=A] {$G$};
\draw[->] (B) to node {$t$} (A);
  \node (C) [right of=B] {$N$};
  \draw[->] (B) to node {$\sigma$} (C);
 \node (E) [above of=B]{$i(G)$};
\node (F)[right of=E]{$i(N)$};
\draw[->] (E) to node {$\tau$} (F);
\draw[->] (A) to node {$v$} (E);
\draw[->] (C) to node {$i$} (F);
\draw[->] (B) to node {$i$} (E);
\end{tikzpicture}
\caption{Lift up of $(N,G,\sigma,\alpha)$ in the case $i$ is discontinuous at $\rho_k(G)$}
\label{Lift_up}
\end{figure}
        
        The map $\tau$ here is only partial on $i(G)$, but
        $\tau \circ v \colon J \to i(N)$ is total. Also,
        $i``G \subseteq \dom(\tau)$, so $\tau \circ i$ is total
        on $G$. For each
        $\eta < \rho_k(G)$, and $x \in \dom(\sigma^\eta)$,
        \[
        i \circ \sigma^\eta(x)  = i(\sigma^\eta)(i(x)),
        \]
        so $\tau \circ i$ agrees on $\hat{G}^k$ with $i \circ \sigma$. Since both
        map $p_k(G)$ to $p_k(i(N))$,
        \[
        \tau \circ i = i \circ \sigma.
        \]
        Clearly $i \restriction G = v \circ t$, so the diagram commutes.
       
             Since $E_i \restriction \sup i``\alpha$ is a
             long extender, we need some care to
             show that $J$ is a premouse. The worry is that it could be a protomouse,
              in the
       the case that $G$ is extender active and $k=0$. So suppose $k=0$ and
       $\mu = \crt(\dot{F}^G)$; it is enough to see that
       $t$ is continuous at $\mu^{+,G}$. If not, we have $f \in G$
       and $b \subseteq \sup i``\alpha$ finite such that 
       \[
       \sup i`` \mu^{+,G} < t(f)(b) < i(\mu^{+,G}).
       \]
       We may assume $\dom(f) = \gamma^{|b|}$, where $\gamma < \alpha$,
        and by Los,
       $\ran(f)$ is unbounded in $\mu^{+,G}$. It follows that
       $\mu^{+,G} < \alpha$, so $\mu^{+,G} = \sigma(\mu^{+,G}) = \mu^{+,N}$.
 But $\cof(\mu^{+,G}) = \cof(\hat{o}(G)) =
       \kappa$ in $N$, so $\mu^{+,G}$ is not a cardinal in $N$, contradiction.
       
       Thus $J$ is a premouse. We claim that $\tau \circ v$ is nearly
       elementary. First, $\tau$ is a partial $\Sigma_0$ map
       from $i(\hat{G}^k)$ to $i(\hat{N}^k)$, so $\tau \circ v \restriction \hat{J}^k$
       is $\Sigma_0$ from $\hat{J}^k$ to $i(\hat{N}^k)$. It is also easy to see that $\tau\circ v\restriction \hat{J}^k$ is cardinal preserving.


       
        For $\eta < \rho_k(G)$, we have that $\sigma^\eta$ is the identity
        on $\alpha \cap \eta$, so
        \[
        \sup i``\alpha = \sup t``\alpha \le \crit(\tau \circ v) .
        \]
        But $\alpha < \sigma(\alpha)$, so $i(\alpha)< i \circ \sigma(\alpha)
        = \tau \circ v \circ t (\alpha)$. Also, $t(\alpha) \le i(\alpha)$,
        so $t(\alpha)< \tau \circ v(t(\alpha))$. Thus
        $\crit(\tau \circ v) \le t(\alpha)$, and since $t(\alpha) = \sup t``\alpha$,
        we get
        \[
         \crit(\tau \circ v) = \sup i``\alpha = t(\alpha).
         \]

\begin{remark}{\rm It is possible that $i$ is continuous at $\alpha$, even though
it is discontinuous at $\rho_k(G)$ and hence $J \neq i(G)$. In this case
$\crt(v)>i(\alpha) = \sup i``\alpha$ and $\crt(\tau) = i(\alpha)= t(\alpha)$.}
\end{remark}
        
        We set
       \begin{align*}
       \M_{\gamma+1}^\itS &= i(N),\\
       \M_{\gamma+2}^\itS &= J,\\
        \sigma_{\gamma+1} &= \tau \circ v, \text{ and}\\
        \alpha_{\gamma +1} &= \crt(\tau \circ v).
        \end{align*}
        We define then
        \begin{center}
         \begin{center}
 $\Phi_{\gamma+1} =  \la (i(N),\Lambda_{\gamma+1}),(J,\Lambda^{\sigma_{\gamma+1}}_{\gamma+1}),\sigma_{\gamma+1}, \alpha_{\gamma+1} \ra$. 
 \end{center}
        \end{center}
        
%
%
\medskip
\noindent
{\em Claim 5.} \label{claim:problematic_propagates}
$\Phi_{\gamma+1}$ is problematic.

\medskip
\noindent
{\em Proof.} We begin with proving a subclaim.

\medskip
\noindent
{\em Subclaim 5a.}
\begin{itemize}
\item[(a)]If $\Phi_\xi^-$ is problematic, then $\Phi_{\gamma+1}^-$
is problematic.
\item[(b)] $\Phi_\xi$ is extender-active iff $\Phi_{\gamma+1}$ is extender-active;
moreover, if $\Phi_\xi$ is extender-active, then $i_{\xi,\gamma+1}^\itS(\alpha_\xi)
=\alpha_{\gamma+1}$.
\end{itemize}

\medskip
\noindent
{\em Proof.}
        Recall that we write
\[
 \Phi_{\gamma+1}^- = \la i(N),J,\tau \circ v, \crt(\tau \circ v) \ra.
\]
It is easy to see that the tuple obeys the hypotheses of
\ref{thm:cond_lem}: clearly $J$ is $\Sigma_{k+1}$ generated by
  $\sup i``\alpha \cup t(s)$, and as we showed above, $\crt(\tau \circ v) =
 \sup i``\alpha$.  $t$ preserves the solidity of $s$, so
 $t(s) = p(J) - \sup i``\alpha$. We have shown
 that  $\tau \circ v$ is nearly elementary. Since $i(G) \in i(N)$,
 we have $J \in i(N)$, so $J$ is not the $\crit(\tau \circ v)$-core
of $i(N)$.

    So what we need to see is that
neither of the conclusions (a)-(b) of \ref{thm:cond_lem}
hold for $\langle i(N),J,\tau \circ v, \crt(\tau \circ v) \rangle$,
and that $\Phi_{\gamma+1}^-$ is extender-active iff $\Phi_\xi^-$
is extender-active. 
We break into two cases.
       
    \bigskip
           \noindent
  
       {\em Case A.} $\Phi_\xi$ is not extender-active.
       \bigskip

Let us show that $\Phi_{\gamma+1}^-$ is not extender active.
For this, we must see that $\crit(\tau \circ v)$ is not an index in $i(N)$.
There are two cases. If $i$ is continuous at $\alpha$, then
$\crit(v) > i(\alpha)$ and $\crit(\tau) = i(\alpha)$, so
$\crit(\tau \circ v) =i(\alpha)$, which is not an index in $i(N)$.
Otherwise, $\crit(\tau \circ v) = \crit(v) = \sup i``\alpha$. But
then $\sup i``\alpha$ has cofinality $\kappa$ in $i(N)$, and since $\kappa$
is a limit cardinal in $i(N)$, it is not the cofinality of the index of a
total extender in $i(N)$.

Now we show that $\Phi_{\gamma+1}^-$ is problematic.
Suppose toward contradiction that it is not; that is, that
\[
C(J) \lhd i(N).
\]

We have $E_\alpha^N = \emptyset.$ Since $(N,\Psi) \le^* 
          (\M_\xi^\T,\Sigma_\xi^\T) \le^* (M,\Sigma)$, Claim 3 gives
\[
G|(\alpha^+)^G = N||(\alpha^+)^G.
\] 
 Since $G \in N$,
         there is a first level $P$ of $N$ such that $P||(\alpha^+)^G =
         G|(\alpha^+)^G$ and $\rho(P) \le \alpha$. Because $\Phi_\xi^-$
         is problematic, $P \neq G$ and if $G$ is of type 2, $G\neq \textrm{Ult}(P,D)$ where $D$ is 
the Mitchell order 0 measure on $\hat{\rho}^-(G)$. Letting $n = \textrm{deg}(P)$, we get by the argument
         above that in $N$, $\rho_n(P)$ has the same cofinality as
         $(\alpha^+)^P = (\alpha^+)^G$, namely $\kappa$. 
         
         We set
          \[
          Q = \Ult(P,E_i \restriction \sup i ``\alpha),
          \]
           and  let $t_0 \colon P \to Q$ be the canonical embedding,
           and $v_0 \colon Q \to i(P)$ be the factor map.\footnote{More precisely, $Q$ is decoded from $\textrm{Ult}_0(\hat{P}^n,E_i\restriction \textrm{sup} i``\alpha)$.} We wish to show that $C(Q) \lhd i(N)$.
Let us assume that $Q \neq i(P)$, as otherwise this is trivial. This implies that
$v_0 \neq \text{id}$ and $\crt(v_0) < \rho_n(Q)$.

 Again, we must see
           that $Q$ is not a protomouse, in the case $P$ is extender
           active with $\crt(\dot{F}^P) = \mu$, and $n=0$. The proof is the same
as it was for $J$. Namely, if 
 $\alpha \le \mu^{+,P}$ then $t_0$ is continuous at $\mu^{+,P}$ because any $f \in P$
such that $f \colon [\gamma]^{|b|} \to \mu^{+,P}$ for $\gamma < \alpha$ has bounded range.
If $\mu^{+,P} < \alpha$, then because $P|\alpha^{+,P} =
G|\alpha^{+,G}$, $\mu^{+,P}$ is a cardinal of $G$, and hence of $N$,
so $i$ is continuous at $\mu^{+,P}$, so $t_0$ is continuous at $\mu^{+,P}$.

           It is easy to check that the hypotheses of \ref{thm:cond_lem}
           hold for $\la (i(P),\Omega),(Q, \Omega^{v_0}),v_0,\crt(v_0)\ra$,
where $\Omega = (\Lambda_{\gamma+1})_{i(P)}$. Note here that
           \[
           \sup i``\alpha =t_0(\alpha) \le \crt(v_0).
           \]
$t_0$ is continuous at
$\alpha$ and $\alpha^{+,P}$, and $i$ is discontinuous at $\alpha^{+,P}$ because
$\cf^N(\alpha^{+,P}) = \kappa$. Thus
\[
\crt(v_0) = \begin{cases} t_0(\alpha) & \text{ if $i$ is discontinuous at $\alpha$,}\\
                       t_0(\alpha)^{+,Q} & \text{ otherwise.}\\
\end{cases}
\]
           Let $s_0 = p(P) \setminus \alpha$; then
           $P = h^{n+1}_P `` \alpha \cup s_0$ because $P$ is sound and
           $\rho(P) \le \alpha$. Thus 
           \[
           Q = h^{n+1}_Q ``(\sup i``\alpha \cup t_0(s_0)).
           \]
           Moreover, $t_0$ maps the solidity witnesses for $s_0$ to
           solidity witnesses for $t_0(s_0)$, so 
           \begin{center}
           $Q$ is $\crit(v_0)$-sound,
           \end{center}
           with parameter $t_0(s) \setminus \crit(v_0)$. Also,
           \[
           \rho(Q) \le t_0(\alpha) \le \crt(v_0) \le t_0(\alpha^{+,P}) \le \rho_n(Q),
           \]
           where the last inequality comes from $\rho_n(Q) = \sup t_0 `` \rho_n(P)$
           and the fact that $t_0$ is continuous at $\alpha^{+,P}$. It is easy to verify that $v_0$ is
           nearly elementary. Finally,
           $i(P)$ is sound, so $Q$ cannot be its $\crt(v_0)$-core.
           
           Thus the hypotheses of \ref{thm:cond_lem}
           hold for $\la (i(P),\Omega),(Q, \Omega^{v_0}),v_0,\crt(v_0)\ra$. But note
           \[
           (i(P),\Omega)) \lhd (i(N),\Lambda_{\gamma+1}) \le^* (\mathcal{M}_{\gamma+1}^\T,
           \Sigma_{\gamma+1}^\T) \le^* (M,\Sigma).
           \]
            So because $(M,\Sigma)$ is minimal, one of the conclusions of
            \ref{thm:cond_lem} holds for $Q,i(P)$. However, $t_0(\alpha)$ is not an index
in $i(N)$ because $\alpha$ is not an index in $N$, and $t_0(\alpha)^{+,Q}$ is not an index
in $i(N)$ because it has cofinality
$\kappa$ in $i(N)$, and $\kappa$ is a limit cardinal in $i(N)$. So
\[
C(Q) \lhd i(P) \lhd i(N).
\]

So $C(J)$ and $C(Q)$ are each the first level of $i(N)$ collapsing
$t(\alpha)^{+,J} = t_0(\alpha)^{+,Q}$ to $t(\alpha) = t_0(\alpha)$, so
\[
C(J) = C(Q).
\]
It follows that $n=k$. We get a contradiction by pulling back to
$G$ and $P$. 
           
           \medskip
           \noindent
           {\em Case (i).} $J, Q$ are both of type 1.
           
           Then $J=Q$, $t \within \alpha = t_0 \within \alpha$, and $t(s) = t_0(s_0) =
p(J)-t(\alpha)$. It follows that
$G = P$, so $G\lhd N$, contradicting our assumption that $\Phi_\xi^-$ is problematic.

\bigskip
\noindent
There are some bothersome details in the remaining cases.  Some arise
from the possibility
that $t$ and $t_0$ do not preserve $\rho_k$. In that connection, the following
lemma is useful. It is implicit (very nearly explicit) in the proof of
solidity and universality for pfs mice in \cite[Section 4.10]{normalization_comparison}.

\begin{lemma}\label{zerosoundnesslemma} ($\adp$) Let $(R,\Sigma)$ be a sound
lbr hod pair of degree $k$, and suppose that either $R$ has type 1A, or
$R$ has type 1B and is not projectum stable; then $R^k_0$ is sound, in the sense
that
\[
R^k_0 = \Hull_1^{R^k_0}(\rho_1(R^k_0)+1 \cup p_1(R^k_0)).
\]
\end{lemma}
\begin{proof}(Sketch.) This is implicit in the proof of
Theorem 4.10.9 in \cite{normalization_comparison}, and in fact
it is pointed out in the last 3 lines of that proof. The problem
is that $R^k_0$ does not have a name for $\rho_k(R)$ in its language,
whereas $R^k$ does.

 Let $\rho_k = \rho_k(R)$,
$\rho = \rho_{k+1}(R)$, and if $R$ has type 1B, let $\bfc_k(R)$ be its
strong core and
$i \colon \bfc_k(R) \to R = \Ult(\bfc_k(R),D)$ be the canonical embedding.
Let 
\[
\varepsilon = \begin{cases} \text{least $\gamma$ s.t. } \rho_k= h^1_{R^{k-1}}(\gamma,p_k(R)) & \text{ if $R$ has type 1A},\\
        \text{least $\gamma$ s.t. } i(D) = h^1_{R^{k-1}}(\gamma,p_k(R)) & \text{ if $R$ has type 1B}.
\end{cases}
\]

Our hypotheses on $R$ imply that $\varepsilon$ exists\footnote{See \cite[4.10.8]{normalization_comparison}.}, and $\varepsilon \in \Hull_1(\rho \cup r)$,
where $r = p_1(R^k_0)$.\footnote{See \cite[4.10.9, Claim 1]{normalization_comparison}.}  Letting $p=p_1(R^k)$,
this enables us to show that $p \subseteq r$,
and to translate $\Sigma_1^{R^k}$ definitions using parameters in $\rho+1 \cup p$ into
$\Sigma_1^{R^k_0}$ definitions using parameters in $\rho+1 \cup r$.\footnote{See Claims 2 and 3
in the proof of \cite[4.10.9]{normalization_comparison}.} It follows that
\[
\Hull_1^{R^k_0}(\rho+1 \cup r) = \Hull_1^{R^k}(\rho+1 \cup p).
\]
Since we have assumed that $R^k$ is 1-sound, this is what we need.
\end{proof}

  Let us return to Subclaim 5a.         
           
\bigskip         
 \noindent {\em Case (ii).} $Q$ has type 1 and $J$ has type 2. 

\medskip
\noindent
In this case, $G$ must have type 1B or type 2. Suppose first it has type 1B.
The relevant diagram is
\begin{center}
\begin{tikzpicture}[node distance=2.0cm, auto]
 \node (A) {$J$};
\node (B) [right of=A] {$C(J)$};
\draw[->] (B) to node {$t_1(i_D)$} (A);
  \node (C) [right of=B, node distance=1cm] {$=$};
 \node (D) [right of=C,  node distance=1cm]{$Q$};
 
 \node (G) [below of=A]{$G$};
\node (H)[right of=G]{$\bfc_k(G)$};
\node (J)[right of=H]{$P$};
\draw[->] (G) to node {$t$} (A);
\draw[->] (H) to node {$t_1$} (B);
\draw[->] (H) to node {$i_D$} (G);
\draw[->] (J) to node {$t_0$} (D);

\end{tikzpicture}
\end{center}

Here $i_D \colon \bfc_k(G) \to G$ is the ultrapower by the order zero measure $D$,
and $t_1$ is the canonical embedding associated to $\Ult(\bfc_k(G),E_i \within \sup i``\alpha)$.
$t_1$ agrees with $t$ and $t_0$ on $\alpha^{+,P}$.

Since $C(J)$ is type 1A, $Q$ is type 1A, so $P$ is type 1A. Thus $P^k_0$ is sound by Lemma \ref{zerosoundnesslemma}.
Since $G$ is type 1B and $J$ is type 2, $G$ is not projectum stable, so
$G^k_0$ is $\alpha$-sound by the lemma. Letting $r = p_1(G^k_0)-\alpha$ and $w = p_1(P^k_0)-\alpha$, we get that $t_1(r)$ is the top
part of $p_1(C(J)^k_0)$, hence the top part of $p_1(Q^k_0) = t_0(w)$, so $t_1(r)=t_0(w)$. Since
$t_0$ and $t_1$ agree on $\alpha^{+,P}$, we get that $\bfc_k(G) =P$. This is a contradiction,
because $\rho_k(\bfc_k(G))$ is measurable in $\bfc_k(G)$ , so $\bfc_k(G)$ is not a pfs premouse.

Suppose next that $G$ has type 2. Now the relevant diagram is
\begin{center}
\begin{tikzpicture}[node distance=2cm, auto]
 \node (A) {$J$};
\node (B) [right of=A] {$C(J)$};
\draw[->] (B) to node {$t_1(i_D)$} (A);
  \node (C) [right of=B, node distance=1cm] {$=$};
 \node (D) [right of=C,  node distance=1cm]{$Q$};
 
 \node (G) [below of=A]{$G$};
\node (H)[right of=G]{$C(G)$};
\node (J)[right of=H]{$P$};
\draw[->] (G) to node {$t$} (A);
\draw[->] (H) to node {$t_1$} (B);
\draw[->] (H) to node {$i_D$} (G);
\draw[->] (J) to node {$t_0$} (D);

\end{tikzpicture}
\end{center}

The proof in the type 1B case now tells us that $C(G) = P$. That means
that $\Phi_\xi^-$ is not problematic, again a contradiction.

\bigskip
\noindent
{\em Case (iii).} $Q$ has type 2.

\medskip
\noindent
In this case, $P$ has type 1B and is not projectum stable.
By the Lemma, $P^k_0$ is sound. We take subcases on the type of $G$.

If $G$ has type 1A, the relevant diagram is
\begin{center}
\begin{tikzpicture}[node distance=2cm, auto]
 \node (A) {$J$};
\node (B) [right of=A,node distance=1cm] {$=$};
  \node (C) [right of=B, node distance=1cm] {$C(Q)$};
 \node (D) [right of=C]{$Q$};
 
 \draw[->] (C) to node {$t_0(i_D)$} (D);
 \node (G) [below of=A]{$G$};
\node (H)[right of=G]{$\bfc_k(P)$};
\node (J)[right of=H]{$P$};
\draw[->] (G) to node {$t$} (A);
\draw[->] (H) to node {$t_1$} (C);
\draw[->] (J) to node {$t_0$} (D);
\draw[->] (H) to node {$i_D$} (J);

\end{tikzpicture}
\end{center}

As above, the agreement between $t$ and $t_0$ and the fact
that $J=C(Q)$ implies that $G=\bfc_k(P)$. This is impossible because
$\bfc_k(P)$ is not a pfs premouse.

If $G$ has type 1B, then $J$ must have type 2, since otherwise
$J=C(J) =C(Q)$, so $J$ has type 1A, so $G$ has type 1A. So our
diagram is
\begin{center}
\begin{tikzpicture}[node distance=2cm, auto]
 \node (A) {$J$};
\node (B) [right of=A] {$C(J)$};
\draw[->] (B) to node {$t_1(i_U)$} (A);
  \node (C) [right of=B, node distance=1cm] {$=$};
 \node (D) [right of=C, node distance=1cm]{$C(Q)$};
 \node (E) [right of=D]{$Q$};
 
 \draw[->] (D) to node {$t_0(i_D)$} (E);

 \node (G) [below of=A]{$G$};
\node (H)[right of=G]{$\bfc_k(G)$};
\node (J)[right of=H]{$\bfc_k(P)$};
\node (K)[right of=J]{$P$};

\draw[->] (G) to node {$t$} (A);
\draw[->] (H) to node {$i_U$} (G);
\draw[->] (H) to node {$t_1$} (B);
\draw[->] (J) to node {$i_D$} (K);
\draw[->] (J) to node {$t_2$} (D);
\draw[->] (K) to node {$t_0$} (E);

\end{tikzpicture}
\end{center}

$G$ is not projectum stable, so $G^k_0$ is $\alpha$-sound.
The agreement of $t_1$ and $t_2$ on $\alpha^{+,P}$ then implies
that $G^k_0 = P^k_0$, hence $\bfc_k(G) = \bfc_k(P)$, and hence
$G=P$. So $\Phi_\xi^-$ is not problematic, contradiction.

If $G$ has type 2, the argument above shows that $C(G) = \bfc_k(P)$.
But this is impossible, because $C(G)$ is a pfs premouse and
$\bfc_k(P)$ is not.

Cases (i)-(iii) exhaust the possibilities, so we have proved Subclaim 5a
in the case that $\Phi_\xi$ is not extender active.

           \bigskip
           \noindent
           {\em Case B.} $\Phi_\xi$ is extender-active.
           
           \medskip
           \noindent
           In this case $\sup  i``\alpha = i(\alpha)$, because  $\alpha$
           has cofinality $\crt(E_\alpha^N)^{+,N}$ in $N$. So $i(\alpha)$
           is an index in $i(N)$, say of $F$. Moreover, $i(\alpha) =
           \crit(\tau \circ v)$, so we have (b) of the subclaim.
           
           Let $R = \Ult(N,E_\alpha^N)$.  $G|\alpha^{+,G}$ is
           an initial segment of $R$ by $(*)(N)$. If $\alpha^{+,R}
           = \alpha^{+,G}$, then $i(\alpha)^{+,J} =
           i(\alpha)^{+,\Ult(i(N),F)}$, where $F = i(E^N_\alpha)$, so $\la i(N),J,\tau \circ v, i(\alpha)\ra$
           is problematic. (Since $\crt(v) > i(\alpha)$,
        $i(\alpha) = \crt(\tau)$.) Otherwise,
           we have a first initial segment $P$ of $R$ past
           $\alpha^{+,G}$ that projects to or below $\alpha$.
           We can now use $P$ just as we did in Case A, thereby proving (a)
           of the subclaim. 
           
           This finishes the proof of Subclaim 5a.
$\hfill   \square$

      We are trying to show $\Phi_{\gamma+1}$ is problematic, so by Subclaim 5a, we may assume that $\Phi_\xi^-$ and $\Phi_{\gamma +1}^-$
are not problematic. Suppose for example that $G \lhd N$ and that
$J \lhd i(N)$ (so both $G,J$ are of type 1). 
The relevant diagram is

\begin{center}
\begin{tikzpicture}
    \matrix (m) [matrix of math nodes, row sep=3em,
    column sep=3.0em, text height=1.5ex, text depth=0.25ex]
{ J & i(G)& i(N)\\
& G & N \\};
     \path[->,font=\scriptsize]
(m-1-1) edge node[above]{$v$}(m-1-2)
(m-1-2) edge
             node [above] {$\tau$} (m-1-3)
(m-2-2) edge node[left]{$t$}(m-1-1)
(m-2-2) edge node[right]{$i$}(m-1-2)
(m-2-3) edge node[right]{$i$}(m-1-3)
(m-2-2) edge node[below]{$\sigma$}(m-2-3);
\end{tikzpicture}
\end{center}

Since $\Phi_\xi$ is problematic,
$\Lambda_\xi^\sigma \neq (\Lambda_\xi)_G$. If
$\Phi_{\gamma +1}$ is not problematic, then 
\begin{equation}\label{eqn:one}
(\Lambda_{\gamma+1})_J = 
\Lambda_{\gamma+1}^{\tau \circ v}.
\end{equation}
Because $(i(G), (\Lambda_{\gamma+1})_{i(G)}) <^* (M,\Sigma)$, we can apply our
induction hypothesis that Theorem \ref{thm:cond_lem} is true
to $v \colon J \to i(G)$. Let
\[
 \Psi = (\Lambda_{\gamma+1})_{i(G)}.
\]
 Since
$v$ is nearly elementary and $J$ is a premouse, Lemma 9.2.3 of
\cite{normalization_comparison} implies that $(J,\Psi^v)$ is an
lbr hod pair and $v$ is a nearly elementary map of it
into $(i(G),\Psi)$ with $\crit(v) \ge \rho(J)$.  The arguments above show
that the other hypotheses of \ref{thm:cond_lem} are satisfied, so we have
\[
\Psi_J = \Psi^v,
\]
in other words
\begin{equation}\label{eqn:two}
((\Lambda_{\gamma+1})_{i(G)})^v = ((\Lambda_{\gamma +1})_{i(G)})_J\\
                             = (\Lambda_{\gamma+1})_J.
\end{equation}
The second line here follows from the fact that $\Lambda_{\gamma+1}$
is mildly positional.\footnote{See \cite[Definition 3.6.1]{normalization_comparison} for the definition of mild positionality. \cite[Lemma 4.6.10]{normalization_comparison} shows that background-induced strategies are mildy positional.}

We now have:

\begin{align*}
\Lambda_\xi^\sigma & = (\Lambda_{\gamma+1})^{i \circ \sigma}\\
                 & = (\Lambda_{\gamma+1})^{\tau \circ v \circ t}\\
                 & =(\Lambda_{\gamma+1}^{\tau \circ v})^t \\
                & = ((\Lambda_{\gamma+1})_J)^t\\
                & = ((\Lambda_{\gamma+1})_{i(G)}^v)^t\\
                & = (\Lambda_{\gamma+1})_{i(G)}^i\\
                & = (\Lambda_\xi)_G.\\
\end{align*}
The first equality follows from Claim 4. The second equality follows from the fact that $i\circ \sigma = \tau\circ v \circ t$.
 The fourth equality follows from Equation \ref{eqn:one}. The fifth equality follows from Equation \ref{eqn:two}. 
We have shown $\Lambda_\xi^\sigma = (\Lambda_\xi)_G$. This contradicts our assumption that $\Phi_\xi$ is problematic,
and finishes the proof of Claim 5 in the case that $G \lhd N$ and $J \lhd i(N)$. 


The following subclaim will help deal with the type 2 case.

\medskip
\noindent
{\em Subclaim 5b.}  Suppose $\la (N,\Phi), (G,\Lambda), \sigma, \nu \ra$ is a tuple 
satisfying (1)--(3) of Definition \ref{dfn:problematic_tuple}. Suppose $G$ is of 
type 2. Let $n = \textrm{deg}(G)$, $D$ be the measure of Mitchell order 0 on 
$\hat{\rho}_n(G)$ such that $G = \textrm{Ult}(C(G),D)$, and let $i_D$ be the 
ultrapower map. Suppose $C(G)\lhd N$ and $\Phi_{C(G)} = \Phi^{\sigma\circ i_D}$. 
Then $$(\Phi_{C(G)})_{\langle D \rangle,G} = \Lambda.$$

\noindent
{\em Proof.} Consider the diagram in Figure \ref{fig:pullback_ult}, where we 
let $H$ be Ult$(G,i_D(D))$ and $j$ be the corresponding ultrapower map, and $Y$ 
be the image of $H$ under the copy construction. We let $k:N\rightarrow Y$ be the ultrapower map by $\sigma(i_D(D))$.

Let $\Sigma = \Phi_{\langle \sigma(i_D(D)) \rangle, Y }$ be the tail of $\Sigma$ on $Y$. We have
\begin{align*}
(\Phi_{C(G)})_{\langle D \rangle,G } & = (\Sigma)^{\sigma\circ j} \\
            & = \Sigma^{k\circ \sigma}\\
            & = \Lambda \\
\end{align*}

The first equality follows from the definition of pullback strategies on stacks. 
The second equality follows from commutativity, i.e. $\sigma\circ j = k\circ \sigma$. 
The third equality follows from our assumption that $\Lambda = \Phi^\sigma$ and the definition of $\Sigma$.

$\hfill \square$

\begin{figure}
\centering
\begin{tikzpicture}[node distance=2cm, auto]
 \node (A) {$C(G)$};
\node (B) [right of=A] {$G$};
\draw[->] (A) to node {$i_D$} (B);
  \node (C) [above of=A, node distance=1cm] {$G$};
  \draw[->] (A) to node {$i_D$}(C);
 \node (D) [right of=C]{$H$};
 \draw[->] (C) to node {$j$} (D);
 \draw[->] (B) to node {$j$}(D);
\node (E) [above of=C, node distance=1cm] {$N$};
\node (F) [right of=E]{$Y$};
\draw[->] (E) to node {$k$} (F);
\draw[->] (D) to node {$\sigma$} (F);
\draw[->] (C) to node {$\sigma$} (E);
\end{tikzpicture}

\caption{Diagram illustrating that $(\Phi_{C(G)})_{\langle D \rangle,G} = \Lambda.$}
\label{fig:pullback_ult}
\end{figure}

Returning to the proof of Claim 5, i.e. that $\Phi_{\gamma+1}$ is problematic,
let us consider the case that $G\lhd N$ and $J$ has type 2, so that
$J = \textrm{Ult}(\bar{J},D)$, 
where $\bar{J}= C(J)\lhd i(N)$ and $D$ is the measure of Mitchell order $0$ on $\hat{\rho}^k_J$.
The relevant diagram is

\begin{center}
\begin{tikzpicture}
    \matrix (m) [matrix of math nodes, row sep=3em,
    column sep=3.0em, text height=1.5ex, text depth=0.25ex]
{ \bar{J} & J & i(G)& i(N)\\
& & G & N \\};
     \path[->,font=\scriptsize]
(m-1-1) edge node[above]{$i_D$}(m-1-2)
(m-1-2) edge node [above]{$v$} (m-1-3)
(m-1-3) edge
             node [above] {$\tau$} (m-1-4)
(m-2-3) edge node[left]{$t$}(m-1-2)
(m-2-3) edge node[right]{$i$}(m-1-3)
(m-2-4) edge node[right]{$i$}(m-1-4)
(m-2-3) edge node[below]{$\sigma$}(m-2-4);
\end{tikzpicture}
\end{center}

Again we assume toward contradiction that $\Phi_{\gamma+1}$ is
not problematic, so that
$(\Lambda_{\gamma+1})_{\bar{J}} = \Lambda_{\gamma+1}^{\tau \circ v \circ i_D}$. Combined with Subclaim 5b,
applied with $\Phi = \Lambda_{\gamma+1}$ and $\sigma = \tau \circ v$, this yields

\begin{equation}\label{eqn:three}
((\Lambda_{\gamma+1})_{\bar{J}})_{\la D, J \ra} = 
\Lambda_{\gamma+1}^{\tau \circ v}.
\end{equation}

Since $(i(G), (\Lambda_{\gamma+1})_{i(G)}) <^* (M,\Sigma)$ 
and $v \colon (J, ((\Lambda_{\gamma+1})_{i(G)})^v) \to (i(G),
(\Lambda_{\gamma+1})_{i(G)})$ satisfies the hypotheses 
of Theorem \ref{thm:cond_lem}, we get from 5b and the mild positionality
of $\Lambda_{\gamma+1}$:

\begin{equation}\label{eqn:four}
((\Lambda_{\gamma+1})_{\bar{J}})_{\la D, J \ra} = 
(\Lambda_{\gamma+1})_{i(G)})^{v}.
\end{equation}

Now we calculate
\begin{align*}
\Lambda_\xi^\sigma &= (\Lambda_{\gamma+1})^{i \circ \sigma}\\
                   &= (\Lambda_{\gamma+1})^{\tau \circ v \circ t}\\
                   &= ((\Lambda_{\gamma+1})^{\tau \circ v})^t\\
                   &= (((\Lambda_{\gamma+1})_{\bar{J}})_{\la D, J \ra})^t\\
                   &= ((\Lambda_{\gamma+1})_{i(G)})^{v})^t\\
                   &= (\Lambda_{\gamma+1})_{i(G)})^{i}\\
                   &= (\Lambda_\xi)_G.
\end{align*}
Line 1 uses Claim 4. Line 2 uses commutativity. Line 4 uses
equation \ref{eqn:three} and line 5 uses equation
\ref{eqn:four}. Line 7 uses Claim 4 again.
Since $\Lambda_\xi^\sigma = (\Lambda_\xi)_G$,
$\Phi_\xi$ is not problematic, contradiction.

Let us consider the case that $G$ is of type 2. Let
$(\bar{G}, \Lambda_\xi^{j \circ \sigma})$ be the type 1 core of
$(G,\Lambda_\xi^\sigma)$, where $j = i_D^{\bar{G}}$ for $D$ the order zero
measure on $\rh_k(G)$. Since $\Phi_\xi^-$ is not problematic,
$\bar{G} \lhd N$ and $i(\bar{G}) \lhd i(N)$.  
$J$ is also of type 2, and setting
\[
H = \Ult(\bar{G}, E_i \within \sup i``\alpha)
\]
and letting $t^* \colon \bar{G} \to H$ be the canonical embedding, we have
\[
J = \textrm{Ult}(H, t^*(D)).
\]
So $J$ has type 2, and $H = C(J)$, and $t^*(D) = D(J)$.  Let 
$m:H\rightarrow J$ be the $t^*(D)$-ultrapower map, and $l: i(\bar{G}) \rightarrow i(G)$ be the $i(D)$-ultrapower map. Let $v^*$ 
be the factor map from $H$ to $i(\bar{G})$, so that $v^* \circ t^* = i \restriction \bar{G}$. See
Figure \ref{fig:type2case} for a diagram of the situation.

\begin{figure}
\centering
\begin{tikzpicture}[node distance=3cm, auto]
 \node (A) {$G$};
\node (B) [right of=A] {$N$};
\draw[->] (A) to node {$\sigma$} (B);
  \node (C) [left of=A, node distance = 1.5cm] {$\bar{G}$};
  \draw[->] (C) to node {$j$} (A);
 \node (E) [above of=A, node distance=1.5cm]{$J$};
 \node (L) [left of=E, node distance=1.5cm]{$H$};
\node (F)[above of=E, node distance = 1.5cm]{$i(G)$};
\node (G)[above of=B]{$i(N)$};
\node (H)[left of=F, node distance=1.5cm]{$i(\bar{G})$};
\draw[->] (A) to node {$t$} (E);
\draw[->] (E) to node {$v$} (F);
\draw[->] (F) to node {$\tau$} (G);
\draw[->] (B) to node {$i$} (G);
\draw[->] (H) to node {$l$} (F);
\draw[->] (L) to node {$m$} (E);
\draw[->] (C) to node {$t^*$} (L);
\draw[->] (L) to node {$v^*$} (H);
\end{tikzpicture}
\caption{Diagram for the case $G$ is of type 2.}
\label{fig:type2case}
\end{figure}

Assuming toward contradiction that $\Phi_{\gamma+1}$ is not problematic, we have
\begin{equation}\label{eqn:five}
(\Lambda_{\gamma+1})_{H} = 
(\Lambda_{\gamma+1})^{\tau \circ v \circ m}.
\end{equation}
As above, the fact that $(i(\bar{G}),(\Lambda_{\gamma+1})_{i(\bar{G})})$ is strictly below
$(M,\Sigma)$ in the mouse order gives
\begin{equation}\label{eqn:six}
(\Lambda_{\gamma+1})_{H} = ((\Lambda_{\gamma+1})_{i(\bar{G})})_{H} = 
((\Lambda_{\gamma+1})_{i(\bar{G})})^{v^*}.
\end{equation}
We then calculate
\begin{align*}
\Lambda_\xi^{\sigma \circ j} &=  \Lambda_{\gamma+1}^{\tau \circ l \circ v^* \circ t^*}\\
                             &= \Lambda_{\gamma+1}^{\tau \circ v \circ m \circ t^*}\\
                             &= (\Lambda_{\gamma+1})_H^{t^*}\\
\intertext{ (by equation \ref{eqn:five})}
                             &= (\Lambda_{\gamma+1})_{i(\bar{G})}^{v^* \circ t^*}\\
\intertext{ (by equation \ref{eqn:six})}
                              &= (\Lambda_{\gamma+1})_{i(\bar{G})}^i\\
                              &= (\Lambda_\xi)_{\bar{G}}.
\end{align*}
This shows that $\Phi_\xi$ is not problematic, contradiction.

The cases where $G\lhd \textrm{Ult}(N, E^N_\alpha)$ or 
$C(G)\lhd \textrm{Ult}(N, E^N_\alpha)$ are handled similarly. 
For example, let $E=E_\alpha^N$, and suppose
 $G\lhd \textrm{Ult}(N, E)$ and $J$ is
 of type 1. In this case, note that $i$ is continuous at $\alpha$, so
$\crit(\tau \circ v )= i(\alpha)$. The relevant diagram is

\begin{center}
\begin{tikzpicture}[node distance=2cm, auto]
 \node (A) {$J$};
\node (B) [right of=A] {$i(G)$};
\draw[->] (B) to node {$v$} (A);
  \node (C) [right of=B] {$i(N)$};
 \node (D) [right of=C]{$Q$};
 \node (E) [right of=D, node distance=1cm]{$\rhd$};
 \node (F) [right of=E, node distance=1cm]{$i(G)$};
 \draw[->] (B) to node {$\tau$} (C);
 \draw[->] (C) to node {$i(E)$} (D);
 \node (G) [below of=B]{$G$};
\node (H)[right of=G]{$N$};
\node (J)[right of=H]{$P$};
\node (K)[right of=J, node distance=1cm]{$\rhd$};
\node (L)[right of=K, node distance=1cm]{$G$};
\draw[->] (G) to node {$t$} (A);
\draw[->] (G) to node {$i$} (B);
\draw[->] (G) to node {$\sigma$} (H);
\draw[->] (H) to node {$E$} (J);
\draw[->] (H) to node {$i$} (C);
\draw[->] (J) to node {$i$} (D);
\draw[->] (L) to node {$i$} (F);
\end{tikzpicture}
\end{center}

Here $P=\Ult(N,E)$ and $Q=\Ult(i(N),i(E))$.  

We assume toward contradiction that $\Phi_{\gamma+1}$ is
not problematic. This implies
\begin{equation}\label{eqn:seven}
((\Lambda_{\gamma+1})_{\la i(E)\ra})_J = 
(\Lambda_{\gamma+1})^{\tau \circ v}.
\end{equation}
As above, the fact that $(i(G), ((\Lambda_{\gamma+1})_{\la i(E)\ra})_{i(G)})$ is strictly below
$(M,\Sigma)$ in the mouse order gives
\begin{equation}\label{eqn:eight}
((\Lambda_{\gamma+1})_{\la i(E) \ra})_J = (((\Lambda_{\gamma+1})_{\la i(E) \ra})_{i(G)})_{J} = 
((\Lambda_{\gamma+1})_{\la i(E)\ra})_{i(G)}^{v}.
\end{equation}
We then calculate

\begin{align*}
\Lambda_\xi^\sigma & =(\Lambda_{\gamma+1}^{\tau\circ v})^t\\
                   &= (((\Lambda_{\gamma+1})_{\la i(E) \ra})_J)^t\\
                   &= (((\Lambda_{\gamma+1})_{\la i(E) \ra})_{i(G)}^v)^t\\
                  &= ((\Lambda_{\gamma+1})_{\la i(E) \ra})_{i(G)}^i\\
                   &= ((\Lambda_\xi)_{\la E \ra})_G.
\end{align*}
Thus $\Phi_\xi$ is not problematic, contradiction.

So we conclude that $\Phi_{\gamma+1}$ is problematic
in all cases. This finishes the proof of Claim 5. $\hfill   \square$

           \medskip
           \noindent

           Claim 5 finishes the definition of $\Phi_{\gamma+1}$,
            and the proof that it is a problematic tuple. We have also
            proved (4) of $(\dagger)_{\gamma+2}$.
            We now verify the rest of $(\dagger)_{\gamma+2}$. For the first part of (5), note that
            if $i$ is discontinuous at $\alpha$, then $\sup i``\alpha =
            \crt(v) = \crt(\tau \circ v) < i(\alpha)$, and if
            $i$ is continuous at $\alpha$, then $\crt(\tau) = i(\alpha) =
            \crt(\tau \circ v)$. Thus in either case, $\alpha_{\gamma+1} \le
            i_{\xi,\gamma+1}^{\itS}(\alpha_\xi)$.
            
            For the rest of (5) and (6), it is enough to see that
            $\beta_{\gamma+1} < i(\beta)$, where  $ \beta = \beta_\xi = (\alpha^+)^G$.
(Recall that we are in the case that $i$ is discontinuous at $\rho_k(G)$.)
            If $i$ is discontinuous at $\alpha$, then $\alpha$ is a limit cardinal
            of $G$, and $\beta_{\gamma+1} =
            (\sup i``\alpha)^{+,J} < i(\alpha) < i(\beta)$, as desired. If $i$ is
            continuous at $\alpha$, then since 
            $(\alpha^+)^G$ has cofinality $\kappa$ in $N$, we get
            \[
              (\alpha_{\gamma+1})^{+,J} = i(\alpha)^{+,J} = \sup i``\beta < i(\beta),
              \]
              as desired.
              
              (7) of $(\dagger)_{\gamma+2}$ is obvious from our definitions. 
              
    \begin{remark}\label{rmk:case2}{\rm If Case 2 occurs in the passage from $\Phi^-_\xi =
    \la N,G,\sigma,\alpha \ra$ to 
    to $\Phi^-_{\gamma+1} = \la i(N),J, \tau \circ v, \crit(\tau \circ v) \ra$,
    then $\rho_k(J) = \sup t`` \rho_k(G)$ has cofinality $\kappa$ in $i(N)$,
    where $\kappa = \crit(E_\gamma^\itS)$. Along branches of $\itS$ containing
    $\gamma+1$, $\kappa$ can no longer be a critical point. It follows that
    along any given branch, Case 2 can occur at most once.} 
    \end{remark}
            
            If (I) or (II) hold at $\gamma+2$, then the construction of $\itS$
            is over. Otherwise, we let $E_{\gamma+2}^{\itS}$ be the least disagreement
            between $\M_{\gamma+2}^{\itS}$ and $M_{\nu,l}$, and we set 
            \[
            \lambda_{\gamma+1} = \inf(\alpha_{\gamma+1}, \hat{\lambda}(E_{\gamma+2}^{\itS})).
            \]
            This completes the successor step in the construction of $\itS$.

Now suppose we are given $\itS \restriction \theta$, where
$\theta$ is a limit ordinal. Let $b=\Sigma(\T\restriction\theta)$.
\\
\\
\noindent \textbf{Case 1.} There is a largest $\eta\in b$ such that $\eta$ is unstable.

\medskip

Fix $\eta$. There are two subcases.
\begin{enumerate}[(A)]
\item for all $\gamma\in b-(\eta+1)$, rt$(\gamma)=\eta+1$.
 In this case, $b-(\eta+1)$ is a branch of $\itS$. Let $\itS$ choose this branch,
\begin{center}
$[\eta+1,\theta)_S=b-(\eta+1)$,
\end{center}
and let $\M^\itS_\theta$ be the direct limit of the $\M^\itS_\gamma$ for sufficiently 
large $\gamma\in b-(\eta+1)$. We define 
the branch embedding $i^\itS_{\gamma,\theta}$ a usual and 
$\pi_\theta:\M^\itS_\theta\rightarrow \M^\T_\theta$ is given by 
the fact that the copy maps commute with the branch embeddings.
 We declare $\theta$ to be stable.

\item for all $\gamma\in b-(\eta+1)$, rt$(\gamma)=\eta$. Let $\itS$ choose
\begin{center}
$[0,\theta)_S = (b-\eta)\cup[0,\eta]_S$,
\end{center}
and let $\M^\itS_\theta$ be the direct limit of the $\M^\itS_\gamma$ for 
sufficiently large $\gamma\in b$.
 Branch embeddings $i^\itS_{\gamma,\theta}$ for $\gamma\geq \eta$ are defined as usual.
  $\pi_\theta:\M^\itS_\theta\rightarrow \M^\T_\theta$ is given by the fact that copy
   maps commute 
  with branch embeddings. We declare $\theta$ to be stable.
\end{enumerate}
Since $\theta$ is stable,  $(\dagger)_\theta$ follows at once from
$\forall \gamma < \theta (\dagger)_\gamma$.

\noindent \textbf{Case 2.} There are boundedly many unstable ordinals in $b$ but no largest one.

We let $\eta$ be the sup of the unstable ordinals in $b$. Let $\itS$ choose 
\begin{center}
$[0,\theta)_S = (b-\eta)\cup[0,\eta]_S$,
\end{center}
and define the corresponding objects as in case 1(B). We declare
 $\theta$ stable, and again $(\dagger)_\theta$ is immediate.

\noindent \textbf{Case 3.} There are arbitrarily large unstable ordinals in $b$. 
In this case, $b$ is a disjoint union of pairs $\{\gamma,\gamma+1\}$ such that
 $\gamma$ is unstable and $\gamma+1$ is stable. We set
\begin{center}
$[0,\theta)_S=\{\xi\in b \ | \ \xi \textrm{ is unstable}\},$
\end{center}
and let $\M^\itS_\theta$ be the direct limit of the $\M^\itS_\xi$'s for $\xi\in b$ unstable.
 There is no dropping in model or degree along $[0,\theta)_S$. We define maps
  $i^\itS_{\xi,\theta}, \pi_\theta$ as usual. If $(\M^\itS_\theta, \Lambda_\theta)$ is not a pair of 
the form $(\M_\tau^\U, \Sigma_\tau^\U)$, then
 we declare $\theta$ stable and $(\dagger)_\theta$ is immediate. 
 
     Suppose that $(\M_\theta^\itS, \Lambda_\theta)$ is a pair of $\U$. We declare
$\theta$ unstable. Note that by clauses (4) and (5) of $(\dagger)$, there
is a $\xi <_S \theta$ such that for all $\gamma$ with $\xi <_S \gamma <_S \theta$,
$i^\itS_{\xi,\gamma}(\la \alpha_\xi,\beta_\xi\ra) = \la \alpha_\gamma, \beta_\gamma \ra$.
So we can set
\[
\alpha_\theta = \text{ common value of $i_{\gamma,\theta}^\itS(\alpha_\gamma)$, for $\gamma <_S 
\theta$ sufficiently large.}
\]
By clause (5), we can set
\[
\M_{\theta+1}^\itS= \text{ common value of $i_{\gamma,\theta}^\itS(\M_{\gamma+1}^\itS)$, for $\gamma <_S 
\theta$ sufficiently large.}
\]
We also let
\[
\sigma_\theta= \text{ common value of $i_{\gamma,\theta}^\itS(\sigma_\gamma)$, for $\gamma <_S 
\theta$ sufficiently large.}
\]
Here
\[
i_{\gamma,\theta}^\itS(\sigma_\gamma) = \text{ upward extension of }
\bigcup_{\eta < \rho} i_{\gamma,\theta}^\itS(\sigma_\gamma^\eta),
\]
where $\rho = \rho_k(\M_{\gamma+1}^\itS)$ for $k = \textrm{deg}(\M_{\gamma+1}^\itS)$,
and the $\sigma_\gamma^\eta$ are the terms in the natural decomposition
of $\sigma_\gamma$. By (6) of $(\dagger)$ and Remark \ref{rmk:case2},
$i_{\gamma,\theta}^\itS$ is continuous at $\rho_k(\M_{\gamma+1})$ for
$\gamma <_S \theta$ sufficiently large, so $\sigma_\theta$ is defined on
all of $\M_{\theta+1}^\itS$. It is easy then to see that
$\Phi_\theta = \la (\M_\theta^\itS, \Lambda_\theta), (\M_{\theta+1}^\itS, \Lambda_{\theta+1}),
 \sigma_\theta, \alpha_\theta \ra$
is a problematic tuple. 

If (I) holds, then we stop the construction of $\itS = \itS_{\nu,l}$ and move on
to $\itS_{\nu,l+1}$. If (II) holds, we stop the construction of $\itS$ and do not
move on. If neither holds, we let $E_{\theta+1}^\itS$ be the extender on
the $\M_{\theta+1}^\itS$ sequence that represents its first disagreement
with $M_{\nu,l}$, and set
\[
\lambda_{\theta+1} = \hat{\lambda}(E_{\theta+1}^\itS),
\]
and
\[
\lambda_\theta = \inf(\lambda_{\theta+1}, \alpha_\theta).
\]
It then is routine to verify $(\dagger)_{\theta+1}$.

This finishes our construction of $\itS=\itS_{\nu,l}$ and $\T$. Note that every extender used 
in $\itS$ is taken from a stable node, and every stable node except the last model of $\itS$ 
contributes exactly one extender to $\itS$. The last model of $\itS$ is stable.

\medskip
\noindent
{\em Claim 6.} The construction of $\itS_{\nu,l}$ stops for one of the reasons (I) and (II).

\medskip
\noindent
{\em Proof sketch.} This comes from the proof of Theorem
\ref{starpsigmathm} and the analysis of the type 2 case in
subsection \ref{almostsoundcase}, together
with the method for adapting such results to the
comparison of phalanxes with background constructions
used in \cite{normalization_comparison}.
$\hfill    \square$

\medskip
\noindent
{\em Claim 7.}
For some $(\nu,l)\leq (\eta_0,k_0)$, the construction of $\itS_{\nu,l}$ stops for reason (II).

\medskip
\noindent
{\em Proof.}
If not, then the construction of $\itS=\itS_{\eta_0,k_0}$ must reach
 some stable $\theta$ such that $(M_{\eta_0,k_0},\Omega_{\eta_0,k_0})
\unlhd (\M^\itS_\theta, \Lambda_\theta)$, and either
$(M_{\eta_0,k_0},\Omega_{\eta_0,k_0})
\lhd (\M^\itS_\theta, \Lambda_\theta)$ or the branch of $\itS$ leading to $\theta$ has
a drop. Let
\[
Q = \pi_\theta(M_{\eta_0,k_0});
\]
then either $(Q,(\Sigma_\theta^\itT)_Q) \lhd (M_\theta^\itT,\Sigma_\theta^\itT)$
or the branch $[0,\theta]_T$ has a drop.

Suppose first that $M_{\eta_0,k_0}$ is a nondropping iterate of $M$ and let
 $j:M \rightarrow M_{\eta_0,k_0}$ be the iteration map given
by $\U_{\eta_0,k_0}$. Letting $\T=\T_{\eta_0,k_0}$, we have $\pi_\theta:\M^\itS_\theta\rightarrow 
\M^\T_{\theta}$ from the copying 
construction.  Note that
\[ \pi_\theta\circ j \colon (M, \Sigma) \to (Q,(\Sigma_\theta^\T)_Q)
\]
is an elementary map, because
\[
\Sigma = \Omega_{\eta_0,k_0}^j = ((\Lambda_\theta)_{M_{\eta_0,k_0}})^j =
 ((\Sigma_\theta^\T)_Q)^{\pi_\theta \circ j}.
\]
Thus $\pi_\theta \circ j$ maps $(M,\Sigma)$ into an 
initial segment of $(\M_\theta^\T,\Sigma_\theta^\T)$ which is either
proper or reached along a branch that drops. This is contrary
to the Dodd-Jensen Theorem.   

Suppose next that $j: M \rightarrow N = \textrm{Ult}(M_{\eta_0,k_0},F)$ is the iteration map from $\itU_{\eta_0,k_0}$.
Let $R = \Ult(Q,\pi_\theta(F))$. We have the diagram

\begin{center}
\begin{tikzpicture}
    \matrix (m) [matrix of math nodes, row sep=3em,
    column sep=3.0em, text height=1.5ex, text depth=0.25ex]
{  M& Q & R\\
 & M_{\eta_0,k_0} & N\\
     & & M \\};
     \path[->,font=\scriptsize]
(m-1-1) edge node[above]{$i_{0,\theta}^\itT$}(m-1-2)
(m-1-2) edge node [above]{$\pi_\theta(F)$} (m-1-3)
(m-2-2) edge node [above]{$F$} (m-2-3)
(m-2-2) edge node [right]{$\pi_\theta$}(m-1-2)
(m-2-3) edge node [right]{$\sigma$}(m-1-3)
(m-3-3) edge node[right]{$j$}(m-2-3);
\end{tikzpicture}
\end{center}
Here we have drawn the case $Q=\M_\theta^\itT$ for simplicity.
$\sigma$ is the copy map.

We need to replace $M$ by $M^-$ in calibrating elementarity,
because $i_{\pi_\theta(F)} \circ i_{0,\theta}^\itT$ is not
elementary as a map on $M$, because of the
drop at the end. But $i_{\pi_\theta(F)} \circ i_{0,\theta}^\itT$ is
an iteration map associated to a +1-iteration tree on $M^-$,
and $\sigma \circ j$ is +1-elementary as a map on $M^-$.
Thus we can apply Dodd-Jensen in the category of +1-iteration maps,
and we have a contradiction as above.\footnote{Dodd-Jensen applies because
we can copy +1 trees on using nearly +1-elementary copy maps. See
\S 3. The more general statement of this fact
belongs to Jensen's $\Sigma^*$ theory.}

$\hfill    \square$

\medskip

By Claim 7, we may fix  $(\nu,l)\leq (\eta_0,k_0)$ such that the construction of $\itS_{\nu,l}$ 
terminates for reason (II). Let
$\itS = \itS_{\nu,l}$ and $\lh(\itS) =
\theta+1$. Thus $\theta$ is stable,
$[\text{rt}(\theta),\theta]_S$ does not drop in model or
degree (so $l=k_0$), and either
\begin{enumerate}[(a)]
\item $(\M^\itS_\theta,\Lambda_\theta) = (M_{\nu,l},\Omega_{\nu,l})$, or\footnote{  The comparison
arguments of \S 3 and \cite{normalization_comparison} show
that if $\la \nu,l\ra$ is
least such that (II) holds, and
$(\M^\itS_\theta, \Lambda_\theta) \unlhd (M_{\nu,l}, \Omega_{\nu,l})$, 
then 
$(\M^\itS_\theta, \Lambda_\theta) = (M_{\nu,l}, \Omega_{\nu,l})$,}
\item $(\M^\itS_\theta,\Lambda_\theta)$ has type 2, and $(M_{\nu,l},\Omega_{\nu,l})$ is its
type 1 core. 
\end{enumerate}
Let $\U = \U_{\nu,l}$, and $\gamma+1 = \lh(\itU)$. The result
of comparison via $\itU$ is that
either
\begin{itemize}
\item[(a')] $(M_{\nu,l},\Omega_{\nu,l})
\unlhd (\M_\gamma^\itU, \Sigma_\gamma^\itU)$, or
\item[(b')]$(\M_\gamma^\itU,\Sigma_\gamma^\itU)$ has type 2, and $(M_{\nu,l},\Omega_{\nu,l})$ is its
type 1 core. 
\end{itemize}

In the usual $(M,H,\alpha)$ vs. $M$ comparisons, one shows that the phalanx side
does not terminate on a branch above $M$. The next claim adapts that argument
to our current situation, in which the phalanx has been lifted along various
branches of $\itS$.

\medskip
\noindent
{\em Claim 8.} 
For some unstable $\xi$, $\text{rt}(\theta) = \xi+1$. 

\medskip
\noindent
{\em Proof.}
If not, then $0\leq_S \theta$ and the branch $[0,\theta]_S$ does not drop
in model or degree.
We take cases on how $\itS$ and $\itU$ end. Let
$i=i_{0,\theta}^\itS$ and $i^* = i_{0,\theta}^\itT$.
Let $\lh(\itU) = \gamma+1$ and $j=i_{0,\gamma}^\itU$.

\medskip
\noindent
{\em Case 1.} $\M_\theta^\itS = M_{\nu,l}$ and $M_{\nu,l} \unlhd \M_\gamma^\itU$.

\medskip
\noindent
{\em Proof.}
We claim that  $(\M_\theta^\itS,\Lambda_\theta) = (\M_\gamma^\U, \Sigma_\gamma^\U)$. For
otherwise $i$ maps
$(M,\Sigma)$ to a proper initial segment of a $\Sigma$-iterate of $(M,\Sigma)$,
 contrary to the Dodd-Jensen Theorem.
(Note here that $\Sigma$ is the pullback of $\Lambda_\theta$ under
$i$, by Claim 4.) For the same reason, $[0,\gamma]_U$ does not drop,
and thus $(\nu,l) = (\eta_0,k_0)$.
The relevant diagram is now

\begin{center}
\begin{tikzpicture}
    \matrix (m) [matrix of math nodes, row sep=3em,
    column sep=3.0em, text height=1.5ex, text depth=0.25ex]
{ \M_\theta^\itT& M_{\nu,l} & M_{\nu,l}\\
M & M & M\\ };
     \path[->,font=\scriptsize]
(m-1-2) edge node[above]{$\pi_\theta$}(m-1-1)
(m-1-3) edge node [above]{$\text{id}$} (m-1-2)
(m-2-1) edge node [left]{$i^*$} (m-1-1)
(m-2-3) edge
             node [left] {$j$} (m-1-3)
(m-2-2) edge node [left]{$i$}(m-1-2);
\end{tikzpicture}
\end{center}
Using Dodd-Jensen we get that
\begin{center}
$i=j$.
\end{center}
To see this, note that both are elementary maps from
$(M,\Sigma)$ to $(\M_\theta^\itS,\Lambda_\theta) = (\M_\gamma^\U, \Sigma_\gamma^\U)
= (M_{\nu,l},\Omega_{\nu,l})$.
Since $j$ is an iteration map, for all $\xi$
\[
j(\xi) \le i(\xi).
\]
Since $i^*$ is also an iteration map, for all $\xi$
\[
i^*(\xi) =\pi_{\theta}\circ i(\xi) \le
\pi_\theta \circ j(\xi).
\]
 Multiplying by $\pi_\theta^{-1}$, we get
that $i(\xi) \le j(\xi)$ for all $\xi$.
 So $i = j$.

Since we can recover branch extenders from branch embeddings, we get then that
\begin{center}
$e^\itS_\theta = e^\U_\gamma$.\footnote{$e^\itS_\theta$ is the sequence of extenders used along the 
branch $[0,\theta]_S$ and similarly for $e^\U_\gamma$.}
\end{center}
Let $\eta\leq_S \theta$ be least such that $\eta$ is stable. Then $e^\itS_\eta = e^\itS_\theta \restriction\delta=
e^\U_\gamma\restriction\delta$ for some $\delta$. But there is $\tau$ such that $e^\U_\tau = e^\U_\gamma\restriction\delta$.
 Thus $\M^\itS_\eta=\M^\U_\tau$. We have also
\[
\Lambda_\eta = \Lambda_\theta^{i_{\eta,\theta}^\itS} = (\Sigma_\gamma^\U)^{i_{\tau,\gamma}^\U} = \Sigma_\tau^\U,
\]
by pullback consistency, since $i_{\eta,\theta}^\itS = i_{\tau,\gamma}^\U$.

If $\eta$ is a limit ordinal, then by the rules at limit stages of $\itS$ above, we declare $\eta$ unstable.
 This contradicts our assumption. If $\itS$-pred$(\eta)=\mu$, then $\mu$ is unstable by our minimality assumption 
on $\eta$; but then we declare $\eta$ unstable by our rules at successor stages. 
Again, we reach a contradiction. This finishes Case 1.
$\hfill    \square$

\bigskip
\noindent
{\em Case 2.} $\M_\theta^\itS$ has type 2, and $M_{\nu,l} \unlhd \M_\gamma^\itU$.

\medskip
\noindent
{\em Proof.} Let $(\M_\theta^\itS,\Lambda_\theta) = \Ult((M_{\nu,l},\Omega_{\nu,l}), F)$, where $F$ is the order
zero measure on $\rh_l(\M^\itS_\theta)$.

We claim that $\M_\gamma^\itU = M_{\nu,l}$ and $[0,\gamma]_U$ does not drop. For if
$M_{\nu,l} \lhd \M_\gamma^\itU$ or
$[0,\gamma]_U$ drops, then
$\itU^\frown \la F \ra$ is a +1-iteration tree on $(M^-,\Sigma)$ whose last model is
$(\M_\theta^\itS,\Lambda_\theta)$ and whose main branch has a large drop, while
$i$ is a +1-elementary map from $(M^-,\Sigma)$ to $(\M_\theta^\itS,\Lambda_\theta)$.
This contradicts Dodd-Jensen, in the form of Lemma \ref{plusonedoddjensen}.
The relevant diagram is now

\begin{center}
\begin{tikzpicture}
    \matrix (m) [matrix of math nodes, row sep=3em,
    column sep=3.0em, text height=1.5ex, text depth=0.25ex]
{ \M_\theta^\itT& \M_\theta^\itS & M_{\nu,l}\\
M & M & M\\ };
     \path[->,font=\scriptsize]
(m-1-2) edge node[above]{$\pi_\theta$}(m-1-1)
(m-1-3) edge node [above]{\text{$i_F$}} (m-1-2)
(m-2-1) edge node [left]{$i^*$} (m-1-1)
(m-2-2) edge node[left]{$i$}(m-1-2)
(m-2-3) edge
             node [left] {$j$} (m-1-3);
\end{tikzpicture}
\end{center}

From the claim, we have that $i_F\circ j$ is defined (and total on $M$). Dodd-Jensen yields $i = i_F \circ j$. For $i_F \circ j$ is a +1-iteration map
on $(M^-,\Sigma)$,
so $i_F\circ j(\eta) \le i(\eta)$ for all $\eta$. But $i^*$ is a
+1-iteration map on $(M^-,
\Sigma)$, so $i^*(\eta) = \pi_\theta\circ i(\eta) \le
\pi_\theta \circ i_F \circ j(\eta)$, so $i(\eta) \le
i_F \circ j(\eta)$. Thus $i= i_F \circ j$.

On the other hand, the generators of the extender of $i$ are contained
in $\sup i``\rho_l(M) = \rho_l(M_{\nu,l})$, while $i_F \circ j$
has the generator $\crit(F) > \rho_l(M_{\nu,l})$. So
$i \neq i_F \circ j$, contradiction.
$\hfill    \square$

\bigskip
\noindent
{\em Case 3.} $\M_\theta^\itS = M_{\nu,l}$ and $\M_\gamma^\itU$ has type 2.

\medskip
\noindent

{\em Proof.} Let $(\M_\gamma^\itU,\Sigma_\gamma^\itU) = \Ult((M_{\nu,l},\Omega_{\nu,l}), F)$, where $F$ is the order
zero measure on $\rh_l(\M^\itU_\gamma)$. The relevant diagram is now
\begin{center}
\begin{tikzpicture}
    \matrix (m) [matrix of math nodes, row sep=3em,
    column sep=3.0em, text height=1.5ex, text depth=0.25ex]
{ & Q & \\
\M_\theta^\itT& M_{\nu,l} & \M_\gamma^\itU\\
M & M & M\\};
     \path[->,font=\scriptsize]
(m-2-2) edge node[above]{$\pi_\theta$}(m-2-1)
(m-2-2) edge node [above]{\text{$i_F$}} (m-2-3)
(m-3-1) edge node [left]{$i^*$} (m-2-1)
(m-3-2) edge node[left]{$i$}(m-2-2)
(m-3-3) edge
             node [left] {$j$} (m-2-3)
(m-2-1) edge node[left]{$\pi_\theta(F)$}(m-1-2)
(m-2-3) edge node[right]{$\sigma$}(m-1-2);
\end{tikzpicture}
\end{center}
Here $Q = \Ult(\M_\theta^\itT,\pi_\theta(F))$ and
$\sigma \colon \Ult(M_{\nu,l},F) \to \Ult(\M_\theta^\itT,\pi_\theta(F))$
is the copy map.

We shall use Dodd-Jensen to show that 
$i \within \rho_l(M) = j \within \rho_l(M)$.

First, $j$ is a +1-iteration map
on $(M^-,\Sigma)$, and
$i_F \circ i \colon (M^-,\Sigma) \to (\Ult(M_{\nu,l},F),\Sigma_\gamma^\itU)$ is
+1-elementary, so $j(\eta) \le i_F \circ i(\eta)$ for all $\eta$. But
$\crit(F) > \rho_l(M_{\nu,l})$, so
\[
\forall \eta <\rho_l(M) (j(\eta) \le i(\eta)).
\]

On the other hand,
$i_{\pi_\theta(F)} \circ i^*$ is a +1-iteration map from
$(M^-,\Sigma)$ to $(Q,\Omega)$, where 
$\Omega =
\Sigma_{\itT^\frown \la \pi_\theta(F) \ra}$.
Here $\Omega$ is
a level $l$ strategy, so $(Q,\Omega)$ is a type 2 pair
generated by $(\M_\theta^\itT,\Sigma_\theta^\itT)$.
 $\Omega$ is a +1-strategy for $Q^-$.
 Since
 \begin{align*}
     \Omega_{\nu,l}&= \Omega^{i_{\pi_{\theta}(F)} \circ \pi_\theta}\\
        &= \Omega^{\sigma \circ i_F}\\
        \intertext{ we have }
\Sigma_\gamma^\itU &= \Omega^\sigma.
\end{align*}
So $\sigma \circ j$ is elementary in the category of mouse pairs.
By Dodd-Jensen, 
\[
i_{\pi_\theta(F)} \circ i^*(\eta) = i_{\pi_\theta(F)} \circ \pi_\theta \circ i(\eta) \le \sigma \circ i_F \circ i (\eta)=\sigma\circ j(\eta)
\]
for all $\eta$. But $i_{\pi_\theta(F)}$ 
and $i_F$ are the identity on $\rho_l(M_{\nu,l})$, and $\sigma$ agrees with
$\pi_\theta$ on $\rho_l(M_{\nu,l})$. Thus
\[
\forall \eta < \rho_l(M)(i(\eta) \le j(\eta)).
\]

Thus $i \within \rho_l(M) = j \within \rho_l(M)$. But the generators of
$i$ and $j$ are all below
$\rho_l(M_{\nu,l}) = \sup i``\rho_l(M)
= \sup j ``\rho_l(M)$. It follows that
$\M_\theta^\itS = \M_\gamma^\itU$,
contrary to our case hypothesis.
$\hfill     \square$.

\bigskip
\noindent
{\em Case 4.} $\M_\theta^\itS$ and $\M_\gamma^\itU$ are both of type 2.

\medskip
\noindent
{\em Proof.} We shall show that $i=j$. This then leads to the same contradiction
we arrived at in case 1.

Let $F$ and $G$ be the order zero
measures on the $M_{\nu,l}$ sequence such that
$\M_\theta^\itS = \Ult(M_{\nu,l},F)$ and 
$\M_\gamma^\itU = \Ult(M_{\nu,l},G)$ respectively.
Let $F_0 = i_F(F)$ and $G_0 = i_G(G)$, and 

\begin{align*}
P &= \Ult(\M_\theta^\itS,F_0) = \Ult(M_{\nu,l},F^+),\\
Q &= \Ult(\M_\gamma^\itU, G_0) = \Ult(M_{\nu,l},G^+),\\
R &= \Ult(P, i_{F^+}(G^+)) = \Ult(Q,i_{G^+}(F^+)),\\
\intertext{ and }
\Omega &= (\Omega_{\nu,l})_{\la F^+\ra,\la i_{F^+}(G^+)\ra} = (\Omega_{\nu,l})_{\la G^+\ra, \la i_{G^+}(F^+)\ra}.
\end{align*}
$\Omega_{\nu,l}$ quasi-normalizes well when considered as a +1-strategy for $M_{\nu,l}^-$,
and $\la F^+, i_{F^+}(G^+)\ra$ is a stack of $\lambda$-separated trees,
so the identity on the last displayed line is justified. $\Omega$ is a level $l$
strategy for $R$, that is, a +1-strategy for $R^-$.
The following diagram describes our situation:
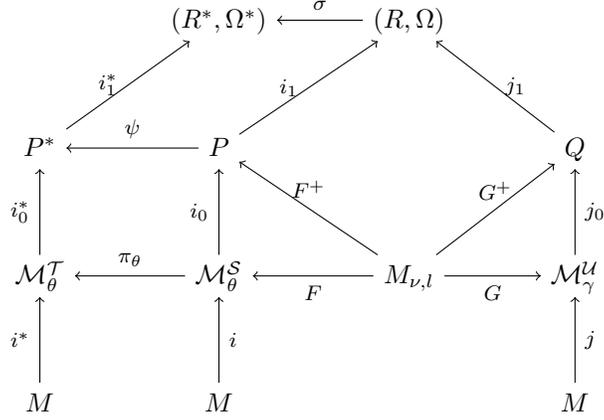
\begin{center}
\begin{tikzpicture}
    \matrix (m) [matrix of math nodes, row sep=3em,
    column sep=3.0em, text height=1.5ex, text depth=0.25ex]
{ & (R^*,\Omega^*) & (R,\Omega) & \\
P^* & P & & Q\\
\M_\theta^\itT& \M_\theta^\itS & M_{\nu,l} & \M_\gamma^\itU\\
M & M & & M\\};
     \path[->,font=\scriptsize]
(m-1-3) edge node[above]{$\sigma$}(m-1-2)
(m-2-1) edge node[left]{$i_1^*$}(m-1-2)
(m-2-2) edge node[above]{$\psi$}(m-2-1)
(m-2-2) edge node[left]{$i_1$}(m-1-3)
(m-2-4) edge node[right]{$j_1$}(m-1-3)
(m-3-1) edge node[left]{$i_0^*$}(m-2-1)
(m-3-2) edge node[left]{$i_0$}(m-2-2)
(m-3-3) edge node[above]{$F^+$}(m-2-2)
(m-3-3) edge node[above]{$G^+$}(m-2-4)
(m-3-3) edge node[below]{$F$}(m-3-2)
(m-3-3) edge node[below]{$G$}(m-3-4)
(m-3-2) edge node[above]{$\pi_\theta$}(m-3-1)
(m-3-4) edge node[right]{$j_0$}(m-2-4)
(m-4-1) edge node[left]{$i^*$}(m-3-1)
(m-4-2) edge node[right]{$i$}(m-3-2)
(m-4-4) edge node [right]{$j$}(m-3-4);
\end{tikzpicture}
\end{center}
Here $i_0, i_1, j_0$, and $j_1$ are the ultrapower maps.
$i^*_0$ and $i_1^*$ come from copying the ultrapowers giving rise
to $i_0$ and $i_1$. The copy maps are $\psi$ and $\sigma$.
We claim that $i_1 \circ i_0 \circ i = j_1 \circ j_0 \circ j$.
For $j_1 \circ j_0 \circ j(\eta) \le i_1 \circ i_0 \circ i(\eta)$ for all
$\eta$, because $j_1 \circ j_0 \circ j$ is an iteration map
and $\Sigma = \Omega^{i_1\circ i_0 \circ i}$, so that
$i_1\circ i_0 \circ i$ is +1-elementary as a map from $(M^-,\Sigma)$
to $(R,\Omega)$. On the other hand,
\begin{align*}
i_1^* \circ i_0^* \circ i^*(\eta) &= \sigma \circ i_1\circ i_0 \circ i (\eta)\\
    & \le \sigma \circ j_1 \circ j_0\circ j(\eta)\\
\end{align*}
because $i_1^* \circ i_0^* \circ i^*$ is a +1-iteration map on $(M^-,\Sigma)$
and $\sigma \circ j_1 \circ j_0\circ j$ is nearly +1-elementary on $(M^-,\Sigma)$.

Thus $i_1 \circ i_0 \circ i = j_1 \circ j_0 \circ j$. Let
\begin{align*}
\rho &= \rho_l(M_{\nu,l});\\
\intertext{ then the generators of $i$ and $j$ are contained in $\rho$,
while the generators of $i_0,i_1,j_0,$ and $j_1$ are all strictly above $\rho$.
So }
       E_i \restriction \rho &= E_{i_1\circ i_0 \circ i} \restriction \rho\\
                             &= E_{j_1 \circ j_0 \circ j} \restriction \rho \\
                            &= E_j \restriction \rho.
\end{align*}
Thus $i=j$ and $\M_\theta^\itS = \M_\gamma^\itU$. The stability of $\theta$
then leads to a contradiction, as in case 1.
$\hfill     \square$

This completes the proof of Claim 8.
$\hfill     \square$

Let $\xi$ be as in Claim 8, and let $\tau$ be such that $(\M^\itS_\xi,\Lambda_\xi)=(\M^\U_\tau,\Sigma_\tau^\U)$.
We have
 $e^\itS_\xi = e^\U_\tau$ by the
proof in claim 8. 

\medskip
\noindent
{\em Claim 9.} $\tau < \gamma$, and $\alpha_\xi \le \lh(E_\tau^\itU) <
\alpha_\xi^{+,\M_\tau^\itU}$.

\medskip
\noindent
{\em Proof.} Suppose $\tau = \gamma$. Let $B$ code $\M_{\xi+1}^\itS$ as a subset of $\alpha_\xi$, namely
for $Q = \M_{\xi+1}^\itS$,
\[
B = \text{Th}_1^{\hat{Q}^l}(\alpha_\xi \cup p_1(\hat{Q}^l)).
\]
Then $B \in \M_\xi^\itS$ because $\Phi_\xi$ is problematic, so
$B \in \M_\tau^\itU$. But $B \notin \M_{\xi+1}^\itS$, so
$B \notin \M_\theta^\itS$, so $B \notin M_{\nu,l}$. Suppose $\theta > \xi+1$, then $\M^\itS_\theta = M_{\nu,l}$ or $M_{\nu,l}$ is a type 1 core of $\M^\itS_\theta$; in either case, $M_{\nu,l}$ is not $l+1$-sound. This means either $M_{\nu,l} = \M^\U_\tau$ or $M_{\nu,l}$ is a type 1 core of $\M^\U_\tau$. In both cases, $\M_\theta^\itS, M_{\nu,l},$ and $\M_\gamma^\itU$ agree to their common value for $\rho_l > \alpha_\xi$, so $B\in M^\itS_\theta$. Contradiction. 

So $\theta = \xi+1$ and therefore, $(\M^\itS_{\xi+1},\Lambda_{\xi+1})=(M_{\nu,l},\Omega_{\nu,l})\lhd (\M^\itU_\tau,\Sigma_\tau)$ (the possibility that $(\M^\itS_{\xi+1},\Lambda_{\xi+1})$ is a type 1 core of $(\M^\itU_\tau,\Sigma_\tau)$ cannot occur because $M^\itS_{\xi+1}\in \M^\itU_\tau$). This contradicts the problematicity of $\Phi_\xi$. So we must have $\tau < \gamma$.


Note that $(\M_\xi^\itS,\Lambda_\xi)$, 
$(\M_{\xi+1}^\itS,\Lambda_{\xi+1})$, and
$(\M_\tau^\U,\Sigma_\tau^\U)$ all agree with $M_{\nu,l}$ below
$\alpha_\xi$. (Possibly not at $\alpha_\xi$.) This
is because otherwise $\lambda_\xi < \alpha_\xi$,
and $\xi+1$ is a dead node in $\itS$. It follows that
$\alpha_\xi \le \lh(E_\tau^\itU)$.

If $\alpha_\xi < \lh(E_\tau^\itU)$, then 
$\M_{\tau+1}^\itU$ agrees with $\M_\gamma^\itU$ up to their
common value $\eta$ for $\alpha_\xi^+$. Suppose for contradiction that $\lh(E^\itU_\tau) > \alpha_\xi^{+,\M^\itU_\tau}$. So $\eta = \alpha_\xi^{+,\M^\itU_\tau}$. Suppose $M_{\nu,l}$ is $\M^\itU_\gamma$ or its type 1 core, then note that $B \in
\M_\tau^\itU - \M_{\gamma}^\itU$, so $\eta <
\alpha_\xi^{+,\M_\tau^\itU}$. Contradiction.  Now suppose $M_{\nu,l}\lhd \M^\itU_\gamma$. Then by the way our comparison works, $o(M_{\nu,l})\geq \lh(E^\itU_\tau) \geq \alpha_\xi^{++,\M^\itU_\gamma}$. But $\rho(M_{\nu,l})=\rho(\M^\itS_{\xi+1})=\rho(\M^\itS_\theta) \leq \alpha_\xi$ and $M_{\nu,l}$ is sound. Contradiction.

Thus
$\lh(E_\tau^\itU) < \alpha_\xi^{+,\M_\tau^\itU}$ as desired.
$ \hfill    \square$

Now let
\[
\rho = \rho(\M_{\xi+1}^\itS) = \rho_{l+1}(\M_{\xi+1}^\itS).
\]
Thus $\rho = \rho(\M_\theta^\itS)= \rho(M_{\nu,l})$
as well.

\medskip
\noindent
{\em Claim 10.} Either $\rho = \alpha_\xi$, or $\rho^{+,\M_{\xi+1}^\itS} = \alpha_\xi$.
Moreover, $\lh(E_\beta^\itT) \le \rho$ for all $\beta < \tau$.

\medskip
\noindent
{\em Proof.} $\rho \le \alpha_\xi$ because $\M_{\xi+1}^\itS$ is $\alpha_\xi$-sound,
so we are done unless $\rho < \alpha_\xi$.

Suppose $\rho < \alpha_\xi$.
Since $\rho$ is a cardinal in
$\M_{\xi+1}^\itS$ and $\sigma_\xi \within \alpha_\xi = \text{ id}$, $\rho$
is a cardinal in $\M_{\xi}^\itS$. But
$|\alpha_\xi| \le \rho$ in $\M_\xi^\itS$, so 
$|\alpha_\xi| = \rho$ in $\M_\xi^\itS$. Since
$\alpha_\xi$ is a cardinal in $\M_{\xi+1}^\itS$,
$\alpha_\xi = \rho^{+,\M_{\xi+1}^\itS}$.
$\hfill    \square$

The following elementary fact will help:

\begin{proposition}\label{projectaprop} Let $\itW$ be a $\lambda$-separated
tree, let $\mu+1 \le_W \eta$, and
let $\M_\eta^\itW|\lh(E_\mu^\itW)
\unlhd N \unlhd \M_\eta^\itW$; then
\begin{itemize}
    \item[(a)] $\lh(E_\mu^\itW) \leq \rho^-(N)$, and
    \item[(b)]
$\rho(N)$ is not in the open interval
$(\crit(E_\mu^\itW), \lh(E_\mu^\itW))$.
\end{itemize}
\end{proposition}

\begin{proof} Let $E_\alpha = E_\alpha^\itW$ and
$M_\alpha = \M_\alpha^\itW$. If
$\rho(M_{\mu+1}^*) \le \crit(E_\mu)$ then
$\rho(M_{\mu+1}) = \rho(M_{\mu+1}^*)$.
If $\crit(E_\mu) < \rho(M_{\mu+1}^*)$, then
\[
\lh(E_\mu) < \sup i_{\mu+1}^* `` \rho(M_{\mu+1}^*) = \rho(M_{\mu+1}).
\]
We use here that $E_\mu$ has plus type; otherwise
$\lh(E_\mu) = \rho(M_{\mu+1})$ is possible.\footnote{But even in the
$\lambda$-separated case, it is possible that $\rho(M_\eta) =
\lh(E_\mu)$ for $\mu+1 <_W \eta$.} Note also that
$\lh(E_\mu) < \rho^-(M_{\mu+1})$ in both cases. So the proposition holds when $\eta = \mu+1$ and $N= \M_{\mu+1}$.
If $\M_{\mu+1}|\lh(E_\mu^\itW)
\unlhd N \lhd \M_\eta^\itW$, then
$\lh(E_\mu) \le \rho(N)$ because
$\lh(E_\mu)$ is a cardinal in
$M_{\mu+1}$ and $\lh(E_\mu) \le\rho^-(N)$. Thus the proposition holds
when $\eta = \mu+1$.

We now proceed by induction on $\eta$.
Let $\mu+1 <_W \beta+1$ and $W\tpred(\beta+1) = \alpha$.
By induction, (a) and (b) hold
at $\eta=\alpha$. We claim that
they hold at $\eta = \beta+1$. 
For since $M_\alpha|\lh(E_\mu) \unlhd M_{\beta+1}^* \unlhd M_\alpha$, we have 
$\lh(E_\mu) < \rho^-(M_{\beta+1}^*)$ and
$\rho(M_{\beta+1}^*)$ is not
in the open interval
$(\crit(E_\mu),\lh(E_\mu))$.\footnote{But
$\rho(M_{\beta+1}^*) = \lh(E_\nu)$ is possible, which is why
the proposition allows $\rho(M_\eta^\itW) = \lh(E_\mu)$.}
But then, taking the $E_\beta$ ultrapower, we see
that
$\lh(E_\mu) \le \lh(E_\beta)
\le \rho^-(M_{\beta+1})$ and
$\rho(M_{\beta+1})$ is not in the open interval
$(\crit(E_\mu),\lh(E_\mu))$. The
same is true for $N \unlhd M_{\beta+1}$ such that $\lh(E_\beta) \le o(N)$
because $\lh(E_\beta)$ is a cardinal
in $M_{\beta+1}$, and hence for
$N \unlhd M_{\beta+1}$ such that
$\lh(E_\mu) \le o(N)$ by coherence.

The induction hypotheses (a) and (b) clearly pass
through limits. 
\end{proof}

\bigskip
\noindent
{\em Claim 11.} $E_{\alpha_\xi}^{\M_\xi^\itS} \neq \emptyset.$

\medskip
\noindent
{\em Proof.} Suppose $E_{\alpha_\xi}^{\M_\xi^\itS} = \emptyset$.
 Thus 
\[
\rho \le \alpha_\xi \le \hat{\lambda}(E_\tau^\itU) < \lh(E_\tau^\itU).
\]
Let $F= E_\mu^\itU$ be the first extender
used in $[0,\gamma]_U$ such that $\lh(F) \ge \alpha_\xi$.
(Thus $\lh(F) > \alpha_\xi$.) We must have
$\crit(F) < \hat{\lambda}(E_\tau^\itU)$, since otherwise
$U\tpred(\mu+1) \ge \tau+1$, so some extender used
in $[0,U\tpred(\mu+1)]_U$ has length $\ge \alpha_\xi$. 

We claim that
$\rho \le \crit(F)$. First notice that
there is an $N \unlhd \M_\gamma^\itU$
such that $\lh(F) \le o(N)$ and
$\rho(N) = \rho$; namely,
$N= \M_\gamma^\itU$ if $\M_\gamma^\itU$
has type 2 with $M_{\nu,l}$ as its
type 1 core, and $N=M_{\nu,l}$ otherwise.
So by Proposition \ref{projectaprop},
$\rho$ is not in the open interval
$(\crit(F),\lh(F))$. But $\rho < \lh(F)$ because $\alpha_\xi < \lh(F)$,
so
\[
\rho \le \crit(F).
\]

By Claim 10, $\lh(E_\beta^\itT) \le \rho$ for all $\beta < \tau$, so we get
\[
U\tpred(\mu+1) = \tau.
\]

By Claim 9, $\alpha_\xi^{+,\M_\gamma^\itU} = \alpha_\xi^{+,\M_{\tau+1}^\itU}
< \alpha_\xi^{+,\M_\tau^\itU}$. It follows that $\M_{\mu+1}^{*,\itU} \lhd
\M_\tau^\itU|\alpha_\xi^{+,\M_\tau^\itU}$, so $\itU$ drops at $\mu+1$.
This implies that $\M_\gamma^\itU$ has type 1 and is unsound, so
$M_{\nu,l} \unlhd \M_\gamma^\itU$.
If $M_{\nu,l} \lhd \M_\gamma^\itU$, then $\lh(F) \le \rho(M_{\nu,l}) = \rho$
by \ref{projectaprop}. Thus
\[
\M_\gamma^\itU = M_{\nu,l}.
\]
Moreover $\rho = \rho(M_{\nu,l})$, so there can be no further dropping
along $[\mu+1,\gamma]_\itU$, and thus
\begin{align*}
\rho &=\rho_{l+1}(\M_{\mu+1}^{*,\itU}),\\
\M_{\mu+1}^{*,\itU} &= \mfc_{l+1}(M_{\nu,l}),\\
\intertext{ and }
 i_0 & =_{\text{df}} i_{\mu+1,\gamma} \circ i^{*,\itU}_{\mu+1}
\end{align*}
is the anticore map.\footnote{At this
point we know that $\M_\gamma^\itU = M_{\nu,l}$ is not $\alpha_\xi$-sound,
so $\xi+1 <_S \theta.$}

Taking the type 1 core of a type 2 premouse of degree $l$
commutes with taking its standard $l+1$ core, so
\begin{align*}
C(\M_{\xi+1}^\itS) &= C(\mfc_{l+1}(\M_\theta^\itS))\\
            &=    \mfc_{l+1}(C(\M_\theta^\itS))\\
             &= \mfc_{l+1}(M_{\nu,l})\\
              &= \M_{\mu+1}^{*,\itU}.
\end{align*}
 So $C(\M_{\xi+1}^\itS) \lhd \M_\tau^\itU = \M_{\xi}^\itS$,
which means that $\Phi_\xi^-$ is not problematic. To see
that $\Phi_\xi$ is not problematic, consider the diagram

\begin{center}
\begin{tikzpicture}[node distance=2cm, auto]
 \node (A) {$M^\itS_\theta$};
\node (B) [right of=A] {$M_{\nu,l}$};
\draw[->] (B) to node {$i_D$} (A);
  \node (C) [right of=B, node distance = 1cm] {$=$};
 \node (E) [right of=C, node distance=1cm]{$\M^\U_\gamma$};
 \node (L) [below of=A]{$\M^\itS_{\xi+1}$};
\node (F)[right of=L]{$C(\M^\itS_{\xi+1})$};
\node (G)[right of=F, node distance=1cm]{$=$};
\node (H)[right of=G, node distance=1cm]{$\M^{*,\U}_{\eta+1}$};
\draw[->] (F) to node {$i_{\bar{D}}$} (L);
\draw[->] (F) to node {$i_0$} (B);
\draw[->] (L) to node {$i^\itS_{\xi+1,\theta}$} (A);
\draw[->] (H) to node {$i_0$} (E);
\end{tikzpicture}
\end{center}

Here $\bar{D} = D(\M_{\xi+1}^\itS)$ if $\M_{\xi+1}^\itS$ has type 2,
and $\bar{D}$ is principal otherwise. $D=i_{\xi+1,\theta}^\itS(\bar{D})$.
Note $\M_\theta^\itS$ has type 2 iff $\M_{\xi+1}^\itS$ has type 2,
because $\M_{\xi+1}^\itS$ is stable and $\rho(\M_{\xi+1}^\itS)
\le \crit(i_{\xi+1,\theta}^\itS)$. We now calculate
\begin{align*}
(\Lambda_\xi)_{C(\M_{\xi+1}^\itS)} &=
(\Sigma_\tau^\itU)_{\M_{\mu+1}^{*,\itU}} \\  &= \Omega_{\nu,l}^{i_0}\\
 &= \Lambda_\theta^{i_D \circ i_0}\\
&=\Lambda_\theta^{i^\itS_{\xi+1,\theta} \circ i_{\bar{D}}}\\
 &= \Lambda_{\xi+1}^{i_{\bar{D}}}.\\
\end{align*}
Line 1 holds because $(\M_\xi^\itS,\Lambda_\xi)
= (\M_\tau^\itU,\Sigma_\tau^\itU)$. 
Line 3 holds by pullback consistency for
$\Sigma$ and the fact that $\Sigma_\tau^\itU = \Omega_{\nu,l}$. Line 3 holds because $\Lambda_\theta^{i_D} =\Omega_{\nu,l}$. The last line holds by Claim 4. 
Thus $\Phi_\xi$ is not problematic, contradiction.
$ \hfill     \square$

The final claim completes our proof by contradiction.

\bigskip
\noindent
{\em Claim 12.} $E_{\alpha_\xi}^{\M_\xi^\itS} = \emptyset.$

\medskip
\noindent
{\em Proof.} Otherwise $E_{\alpha_\xi}^{\M_\xi^\itS} = (E_\tau^\U)^-$, so
$\lh((E_\tau^\U)^-) = \alpha_\xi$. By Claim 10,
\[
\hat{\lambda}(E_\tau^\itU) \le \rho.
\]

      Let $F= E_\mu^\itU$ be the first extender
used in $[0,\gamma]_U$ such that $\lh(F) \ge \alpha_\xi$. We
claim that $\mu = \tau$. For if not, 
\[
\alpha_\xi < \lh(F),
\]
and  $F$ is incompatible with $E_\tau^\itU$,
so that
\[
\crit(F) < \hat{\lambda}(E_\tau^\U) \le \rho \le \alpha_\xi.
\]
But this means that there
is an $N \unlhd \M_\gamma^\itU$ such that $\lh(F) \le o(N)$ and
$\rho(N)$ is in the open interval
$(\crit(F),\lh(F))$, contrary to Proposition \ref{projectaprop}.\footnote{Namely,
$N= \M_{\nu,l}$ if $M_{\nu,l}
\lhd \M^\itU_\gamma$, and $N=\M_\gamma^\itU$ otherwise.}

Thus $F=E_\tau^\itU$ and $\tau+1 \le_U \gamma$. But
\[
\rho(\M_{\tau+1}) \notin (\crit(F),\lh(F)],
\]
and 
\[
\rho^-(\M_{\tau+1}^\itU) > \lh(F).
\]
(Here we use that $F$ has plus type
to rule out $\rho(\M_{\tau+1}^\itU)=
\lh(F)$. This implies that whenever
$N \unlhd M_{\tau+1}^\itU$,
$\lh(F) \le o(N)$, and $\rho(N) = \lh(F)$, then $N \in M_{\tau+1}^\itU$.
That in turn implies that $\tau+1 \neq \gamma$.\footnote{Otherwise $N= M_{\nu,l}$ or
$N = \M_\gamma^\itU$ is a counterexample.}

Now let $G = E_\eta^\itU$, where $\eta+1 \le_U \gamma$ and
$U\tpred(\eta+1)=\tau+1$. $\crit(G) > \hat{\lambda}(F)$,
so $\crit(G) > \alpha_\xi \ge \rho$.
Since $\rho = \rho(N)$ where either $N = \M_\gamma^\itU$ or $N=M_{\nu,l} \unlhd \M_\gamma^\itU$,
we get $\rho = \rho(\M_{\tau+1}^\itU)$,
from \ref{projectaprop}. Moreover,
$\itU$ must drop at $\eta+1$, with
$\rho(\M_{\eta+1}^{*,\itU}) = \rho$, and have no further drops.
It follows that $\M_\gamma^\itU$ has type 1 and is not
$l+1$-sound, 
$\M_{\gamma}^\itU = M_{\nu,l}$, and
\begin{align*}
\M_{\eta+1}^{*,\itU} &= \mfc_{l+1}(M_{\nu,l})\\
                     & = \mfc_{l+1}(C(\M_\theta^\itS))\\
                      &= C(\M_{\xi+1}^\itS).
\end{align*}
Moreover the anticore map from $\M_{\eta+1}^{*,\itU}$ to
$M_{\nu,l}$ is
\[
i_0 = i_{\eta+1,\gamma}^\itU \circ i_{\eta+1}^{*,\itU}.
\]

There is enough agreement between $\M_{\tau+1}^\itU$
and $\Ult_0(\M_\tau^\itU,F)$ that 
$\M_{\eta+1}^{*,\itU} \lhd \Ult(\M_\tau^\itU,F)$. This implies
that $\Phi_\xi^-$ is not problematic, because the ``ultrapower away"
conclusion holds. To see that $\Phi_\xi$ is not problematic, 
consider first the diagram

\begin{center}
\begin{tikzpicture}[node distance=2cm, auto]
 \node (A) {$M^\itS_\theta$};
\node (B) [right of=A] {$M_{\nu,l}$};
\draw[->] (B) to node {$i_D$} (A);
  \node (C) [right of=B, node distance = 1cm] {$=$};
 \node (E) [right of=C, node distance=1cm]{$\M^\U_\gamma$};
 \node (L) [below of=A]{$\M^\itS_{\xi+1}$};
\node (F)[right of=L]{$C(\M^\itS_{\xi+1})$};
\node (G)[right of=F, node distance=1cm]{$=$};
\node (H)[right of=G, node distance=1cm]{$\M^{*,\U}_{\eta+1}$};
\draw[->] (F) to node {$i_{\bar{D}}$} (L);
\draw[->] (F) to node {$i_0$} (B);
\draw[->] (L) to node {$i^\itS_{\xi+1,\theta}$} (A);
\draw[->] (H) to node {$i_0$} (E);
\end{tikzpicture}
\end{center}

(Again, $\bar{D} = D(\M_{\xi+1}^\itS)$ if $\M_{\xi+1}^\itS$ has type 2 and
$\bar{D}$ is principal otherwise, and $D=i_{\xi+1,\theta}^\itS(\bar{D})$.)
Calculating as above, we get
\begin{align*}
(\Sigma_{\tau+1}^\itU)_{\M_{\eta+1}^{*,\itU}} &= \Omega_{\nu,l}^{i_0}\\
                       &= \Lambda_\theta^{i_D \circ i_0}\\
                   &=\Lambda_\theta^{i_{\xi+1,\theta}^\itS \circ i_{\bar{D}}}\\
                   &= \Lambda_{\xi+1}^{i_{\bar{D}}}.
\end{align*}

Since $\Lambda_\xi = \Sigma_\tau^\itU$,
what we must see is that
\begin{equation}\label{eqn:ten}
((\Sigma_{\tau}^\itU)_{\la F\ra})_{\M_{\eta+1}^{*,\itU}}
                   = (\Sigma_{\tau+1}^\itU)_{\M_{\eta+1}^{*,\itU}}.
\end{equation}
That is, the tail of $\Sigma$ after the length 2 stack $\la \itU \within \tau+1, \la F \ra \ra$
agrees with the tail of $\Sigma$ after the length 1 stack $ \itU \within \tau+2 $,
as far as trees based on $\M_{\eta+1}^{*,\itU}$ goes. This follows from the fact that
$\Sigma$ normalizes well. For let $\itW = W(\itU \within \tau+1, F)$ and consider the
embedding normalization diagram
\begin{center}
\begin{tikzpicture}
    \matrix (m) [matrix of math nodes, row sep=3em,
    column sep=3.0em, text height=1.5ex, text depth=0.25ex]
{ M & \M_\tau^\itU & \Ult(\M_\tau^\itU, F)\\
& & \M_\beta^\itW\\};
     \path[->,font=\scriptsize]
(m-1-1) edge node[above]{$i_{0,\tau}^\itU$}(m-1-2)
(m-1-2) edge
             node [above] {$i_F$} (m-1-3)
(m-1-1) edge node[below]{$i_{0,\beta}^\itW$}(m-2-3)
(m-1-3) edge node[right]{$\sigma$}(m-2-3);
\end{tikzpicture}
\end{center}
$E_\tau^\itW = F = E_\tau^\itU$,
so $\itW \within \tau+2 = \itU \within \tau+2$.
Let 
\[
\alpha = \alpha_\xi^{+, \Ult(\M_\tau^\itU,F)}.
\]
It is not hard to see that $\alpha < \lh(E_{\tau+1}^\itW)$, so since
$\Sigma$ is strategy coherent,
\[
(\Sigma_{\tau+1}^\itU)_{\M_{\eta+1}^{*,\itU}} = (\Sigma_{\beta}^\itW)_{\M_{\eta+1}^{*,\itU}}.
\]
Because $\Sigma$ normalizes well,
\[
(\Sigma_{\tau}^\itU)_{\la F\ra} = (\Sigma_{\beta}^\itW)^\sigma.
\]
But $\sigma \within \alpha = \text{ id}$ by the elementary properties
of embedding normalization, so
\[
((\Sigma_{\tau}^\itU)_{\la F\ra})_{\M_{\eta+1}^{*,\itU}} = (\Sigma_{\beta}^\itW)_{\M_{\eta+1}^{*,\itU}}.
\]
Putting these equalities together, we get
equation \ref{eqn:ten}. 

Thus $\Phi_\xi$ is not problematic, a contradiction.
$\hfill     \square$

Claims 11 and 12 are the contradiction that finishes our proof of the
Condensation Theorem, \ref{thm:cond_lem}.
\end{proof}

We can drop the hypothesis that $\crit(\pi) < \rho_{\textrm{deg}(H)}(H)$ from Theorem
\ref{thm:cond_lem}, at the cost of omitting its conclusions
concerning condensation of the external strategies. This will be useful
in the proof of square and full normalization. 

\begin{theorem}\label{condensationtheorem} Assume $\adp$, and let
$(M,\Lambda)$ be a mouse pair with scope $\hc$. Let $H$ be a
sound premouse, $\pi \colon H \to M$ be nearly elementary,
and suppose that
\begin{itemize}
\item[(1)] $\rho(H) \le \crit(\pi)$, and
\item[(2)] $H \in M$.
\end{itemize}
Then either
\begin{itemize}
\item[(a)] $C(H) \lhd M$, or
\item[(b)] $C(H) \lhd \Ult(M,E_\alpha^M)$, where $\alpha = \crit(\pi)$.
\end{itemize}
\end{theorem}

\begin{proof} Let $\alpha = \crit(\pi)$, $n$ be largest such that $\alpha < \rho_n(H)$, and
$n \le \textrm{deg}(H)$. Let $G$ and $N$ be the same as $H$ and $M$, except
that $\textrm{deg}(G) = n = \textrm{deg}(N)$. Let $\Psi = \Sigma_N^\pi$.
The hypotheses of \ref{thm:cond_lem} hold of
$(G,\Psi)$, $(N,\Sigma_N)$, and $\pi$. (We have $H \in M$ by
\ref{alphacorecase}, hence $G \in N$, hence $G$ is not the
$\alpha$-core of $N$.) Hence one of the conclusions of
\ref{thm:cond_lem} holds of them.

    If it is conclusion (a), then $C(G) \lhd N$, which easily implies
    $C(H) \lhd M$. If it is (b), then $C(G) \lhd Ult_0(N, \dot{E}^M_\alpha)$
    yields $C(H) \lhd Ult_0(M, \dot{E}^M_\alpha)$. 
\end{proof}

\bibliographystyle{plain}
\bibliography{Rmicebib}
\end{document}